\documentclass[11pt]{article_ET_m}

\usepackage[left=2.5cm, right=2.5cm, top=2.5cm, bottom=2.5cm]{geometry}
\usepackage[dvipdf]{graphicx}
\usepackage{latexsym}
\usepackage{amsmath,amssymb,amsfonts,amsthm}
\usepackage{xcolor}
\usepackage{mathrsfs}
\usepackage{bbm}
\usepackage{bm}
\usepackage{enumitem}
\usepackage{mathtools}
\usepackage{comment}

\usepackage{graphicx}
\usepackage{amstext}
\usepackage[french,english]{babel}
\usepackage[utf8]{inputenc}
\usepackage[T1]{fontenc}
\usepackage[numbers,comma,sort]{natbib}

\usepackage{pdfsync}

\numberwithin{equation}{section}
\allowdisplaybreaks

\definecolor{db}{RGB}{0, 0, 130}
\usepackage[colorlinks=true,citecolor=red,
linkcolor=db,urlcolor=blue,pdfstartview=FitH]{hyperref}

\newcommand{\R}{\mathbb R}

\newcommand{\N}{\mathbb N}

\newcommand{\E}{\mathbb E}
\newcommand{\PP}{\mathbb P}
\newcommand{\Q}{\mathbb Q}
\newcommand{\T}{\mathbb{T}}

\newcommand{\dd}{\mathrm{d}}

\newtheorem{theorem}{Theorem}[section]

\newtheorem{proposition}[theorem]{Proposition}
\newtheorem{lemma}[theorem]{Lemma}
\newtheorem{corollary}[theorem]{Corollary}
\newtheorem{definition}[theorem]{Definition}

\newtheorem{assumptions}[theorem]{Assumption}

\newtheorem{remark}[theorem]{Remark}
\newtheorem{example}[theorem]{Example}

\usepackage{todonotes}

\author{
Nicoleta Cazacu%
\footnote{CMAP, Ecole polytechnique, CNRS, I.P. Paris,
91128 Palaiseau, France.
\texttt{nicoleta.cazacu@polytechnique.edu},
\texttt{milica.tomasevic@polytechnique.edu}.}
\and
Alexandre Richard%
\footnote{Universit\'e Paris-Saclay, CentraleSup\'elec and
CNRS FRA-3487, France.
\texttt{alexandre.richard@centralesupelec.fr}.}
\and
Milica Toma\v{s}evi\'c%
\footnotemark[1]
}

\title{Uniform-in-time propagation of chaos for singular $L^k$ interaction kernels with rates}

\begin{document}

\date{\today}

\maketitle

\begin{abstract}
Based on a semigroup approach, we prove quantitative uniform-in-time
propagation of chaos for two systems of interacting particles with singular interaction kernels $K$ of type $L^k$, $k>d$, and a contractive Lipschitz confinement force field. First, we consider a system with moderate interaction, which means that the particles interact through the mollified kernel $K\ast V^N$, for some approximation of identity $V^N = N^{d\alpha} V(N^\alpha \cdot),\, \alpha>0$. 
Assuming that $\|K\|_{L^k}$ is sufficiently small, we obtain that for any
$\varepsilon>0$ the uniform-in-time rate is of order
\(
N^{-\min\left(\alpha,\frac12-\frac{d\alpha}{2}-\varepsilon\right)}
\), plus an exponentially small in time contribution of the initial condition.
Second, for the classical, non-mollified particle system we establish 
uniform-in-time propagation of chaos with  rate 
\(
N^{-\min\left(\alpha,\frac12-\frac{d\alpha}{2}-\varepsilon\right)} + \|K-K\ast V^N\|_{L^{k}}\), plus an exponentially small in time contribution of the initial condition.
The method relies on the regularity properties of the solution of the
associated nonlinear Fokker-Planck equation combined with estimates of the Duhamel formulation of the mollified empirical measure, together with partial Girsanov transforms for 
the non-mollified system of particles.
 \end{abstract}

\vspace{0.3cm} \\
\noindent {\bf MSC2020 subject classification:} 60H10, 60H15, 60K35, 35Q84, 82C22

\setcounter{tocdepth}{2}
\renewcommand\contentsname{}
\vspace{-1cm}

\tableofcontents

\section{Introduction} \label{Intro}

Systems of stochastically interacting particles arise in a wide range
of applications, from statistical physics and plasma dynamics to 
mathematical biology, opinion dynamics and, more recently, the training of 
large machine-learning models, see e.g.~\cite{Sznitman,Stevens,FournierJourdain,CJLZ21,GerberEtAl}. A recurring 
question is that of their mean-field limit: when the number
 of particles $N$ becomes large, the behaviour of the whole system
  is expected to be well approximated by that of a single typical
  particle whose law solves a nonlinear, McKean-Vlasov-type equation. 
  Making this approximation quantitative and understanding whether it
   remains valid uniformly over arbitrarily long time horizons is the 
   object of a vast and active literature, see references below. The difficulty is amplified when the interaction between particles is singular, for example when the
    interaction kernel is no longer bounded or Lipschitz but only has some integrability properties. The goal of this paper is to provide quantitative uniform-in-time propagation of chaos results for a class of integrable interaction kernels.

Let us introduce the following nonlinear Fokker-Planck equation on $\R^d$:
\begin{equation}\label{nlfp-intro}
\partial_t u_t =\Delta u_t + \nabla \cdot \big(u_t 
\big(F + K * u_t\big)\big), 
\end{equation}
where $F:\R^d \to \R^d$ is a confining force field and $K:\R^d \to \R^d$ is an  
interaction kernel. According to 
Itô's formula, Equation~\eqref{nlfp-intro} describes the time evolution of the probability distribution $u_t$ of a particle whose dynamics is given by the following 
McKean-Vlasov SDE:
\begin{equation}\label{mckean-intro}
\left\{
\begin{array}{ll}
\dd X_t = -F(X_t) \dd t- (K * u_t)(X_t) \dd t+ \sqrt{2} \dd W_t, \quad \forall t >0,\\
\text{Law}(X_t)=u_t,
\end{array}
\right.
\end{equation}
where $(W_t)_{t \geq 0}$ is a standard Brownian motion in $\R^d$. Equation~\eqref{mckean-intro} can be seen as a typical particle obtained in the
large $N$ limit of interacting particle systems:
\begin{itemize}
\item on the one hand, the classical setting, see \emph{e.g.}~\cite{Sznitman}, reads
\begin{equation}\label{particle-system-intro-gen}
\dd Y_t^{i,N} = -F(Y_t^{i,N}) \dd t- \frac{1}{N} \sum_{j=1}^N K(Y_t^{i,N} - Y_t^{j,N}) \dd t+ \sqrt{2} \dd W_t^i, \quad i=1,\ldots,N,
\end{equation}
where $(W_t^i)_{t \geq 0}$ are independent standard Brownian motions in $\R^d$;

\item on the other hand, the so-called moderate regime \emph{à la} Oelschl\"ager considers particles with a shorter range of interaction modulated by a mollifier $V^N$ as follows
\begin{equation}\label{particle-system-intro}
\dd X_t^{i,N} = -F(X_t^{i,N}) \dd t- \frac{1}{N} \sum_{j=1}^N K \ast
 V^N(X_t^{i,N} - X_t^{j,N}) \dd t+ \sqrt{2} \dd W_t^i, \quad i=1,\ldots,N.
\end{equation}
\end{itemize}
Throughout this work, it will be assumed that $F$ is Lipschitz and contractive (not necessarily derived from a potential), and that $K\in L^k(\R^d)$ for $k>d$ which corresponds to the Ladyzhenskaya-Prodi-Serrin condition for time-homogeneous coefficients (see also Krylov-R\"ockner~\cite{KR05} for stochastic differential equations). We emphasise here that there is no additional structural hypothesis on $K$. Such particle systems are known to be well-posed and to propagate chaos  on finite time intervals, \emph{i.e.} the empirical measure of~\eqref{mckean-intro} converges to the law of the solution of~\eqref{nlfp-intro}, see~\cite{T23,Mauri,HaoRocknerZhang}. 
To the best of our knowledge, however, uniform-in-time propagation of
chaos on the whole space for general $L^k$ kernels, without any structural assumption such
as a gradient or divergence-free form, has not been addressed so far.

\medskip

The aim of this work is thus to provide quantitative uniform-in-time propagation of chaos results for the particle systems~\eqref{particle-system-intro-gen} and~\eqref{particle-system-intro}. In the spirit of~\cite{ORT,ORT2}, 
we further develop the semigroup approach in combination with the recent results of~\cite{HW25}. Let us start with a simple observation: in~\eqref{mckean-intro}, as the kernel is in $L^k$, one expects the law $u_{t}$ to be in the dual space $L^{k^*}$. This leads us to consider the approximation of $u$ in such functional spaces, and naturally smooth the empirical measure $\mu^N$ of our particle systems with a mollifier. Let us introduce $V^N(\cdot) \coloneqq N^{\alpha d} V(N^{\alpha} \cdot)$ with $\alpha \in (0,1)$ and $u_{t}^N = V^N \ast \mu_{t}^N$. Our goal will be to estimate quantitatively the distance between $u_t$ and $u^N_t$ in the dual space, as well as the Wasserstein distance between the law of $N$ particles and $u_{t}^{\otimes N}$. Interestingly, our proof relies on cross-estimating these two quantities to derive the following results:
\begin{enumerate}[label=(\alph*)]
\item In the moderate regime~\eqref{particle-system-intro}, which is particularly convenient as $u^N_{t}$ appears in the drift, under a smallness hypothesis on $ \lVert K\rVert_{L^k}, k>d$, assuming $\alpha<1/(d\vee 2)$ and for any initial data with some moment assumptions (but no i.i.d. nor exchangeability condition), we obtain
that for any $t>0$ and any $N\geq1$,
\begin{align}
\label{eq:first-rate}
    \big\|\|u_t^N-u_t\|_{L^{k^*}}\big\|_{L^2(\PP)}
    &\lesssim (1\wedge t)^{-\frac12-\frac{d}{2k}}
    \left( e^{-ct}
    \mathbb W_2^N\big(\mathcal L(X_0^N),u_0^{\otimes N}\big)
    + N^{-\min\left(\alpha,\frac12-\frac{d\alpha}{2}-\varepsilon\right)}
    \right)
\end{align}
and
\begin{equation*}
\mathbb{W}_2^N\bigl(\mathcal{L}_{X_t^N},u_t^{\otimes N}\bigr)
\lesssim
\Big(e^{-ct} \mathbb{W}_2^N\bigl(\mathcal{L}_{X_0^N},u_0^{\otimes N}\bigr)
+ N^{-\min\left(\alpha,\frac12-\frac{d\alpha}{2}-\varepsilon\right)}
\Big),
\end{equation*}
where $\mathbb{W}^N_2$ is the renormalised $2$-Wasserstein distance and 
$X^N \coloneqq (X^{1,N},\cdots,X^{N,N})$ is the system of $N$ moderately 
interacting particles~\eqref{particle-system-intro}. See Theorem~\ref{thm:main}.

\item In the classical setting~\eqref{particle-system-intro-gen}, we build upon the proof of the moderate regime, but as the interaction is genuinely singular for finite $N$, we need additional ingredients such as regularity estimates on the law of any couple of particles.
Under the same hypotheses and assuming further that $k>2\vee d$, we obtain that for any $t>0$ and any $N\geq1$,
\begin{equation*}
\mathbb{W}_2^N\bigl(\mathcal{L}_{Y_t^N},u_t^{\otimes N}\bigr)
\lesssim \Big(e^{-ct}
\mathbb{W}_2^N\bigl(
\mathcal{L}_{Y_0^N},u_0^{\otimes N}
\bigr) +N^{-\min\left(\alpha,\frac12-
\frac{d\alpha}{2}-\varepsilon\right)}
+\|K-K\ast V^N\|_{L^k}\Big),
\end{equation*}
and for $v_t^N\coloneqq V^N\ast \big(\frac1N\sum_{i=1}^N\delta_{Y_t^{i,N}} \big)$,
\begin{align*}
    \left\|\|v_t^N-u_t\|_{k^*}\right\|_{L^2(\PP)}
    &\lesssim (1\wedge t)^{-\frac12-\frac{d}{2k}}
    \left( e^{-ct} \mathbb W_2^N\big(\mathcal L(Y_0^N),u_0^{\otimes N}\big)
    +    N^{-\min\left(\alpha,\frac12-\frac{d\alpha}{2} -\varepsilon\right)}
    +\|K-K\ast V^N\|_{L^k} \right).
\end{align*}
The complete statement is in Theorem~\ref{thm:main-non-mollified}.
\end{enumerate}
An important consequence of the previous convergences is that if we further assume that the initial family
\(\bigl(X_0^{i,N},X_0^i\bigr)_{\{i=1,\ldots,N\}}\)
is exchangeable, then the coupled dynamics remain exchangeable and we get
\begin{equation*}
\E\left[|X_t^{1,N}-X_t^1|^2\right]^{\frac12}
\leq C\Big(e^{-ct}\E\left[|X_0^{1,N}-X_0^1|^2\right]^{\frac12}
+N^{-\min\left(\alpha,\frac12-\frac{d\alpha}{2}-\varepsilon\right)}\Big),
\end{equation*}
and similarly for the non-mollified particle system with the additional error $\|K-K\ast V^N\|_{L^k}$. A more general statement holds without exchangeability, see Corollary~\ref{cor:pathwise-poc}.

We observe that the same rate $N^{-\min\left(\alpha,\frac12-\frac{d\alpha}{2}-\varepsilon\right)}$ appears in~\eqref{eq:first-rate} and subsequent equations. In the $L^{k^*}$ norm, this rate is almost optimal, in the sense that up to $\varepsilon$, this is exactly the rate of convergence in $L^{k^*}$ of the empirical density (obtained by mollification with $V^N$) of $N$ i.i.d. samples in $\R^d$ drawn from a Lipschitz density $u$.
This rate is then transferred to the Wasserstein distance since we close the estimations simultaneously in the $L^{k^*}$ norm and the Wasserstein distance.

For the non-mollified particles, no uniform-in-time propagation of chaos is known in general for Krylov-R\"ockner kernels (see discussion on the literature below), which is in particular related to the difficulty to obtain logarithmic Sobolev inequalities (LSI) in this context. We provide a rate, which is implicit in the sense that it involves $\|K-K\ast V^N\|_{L^k}$ and this quantity can only be said to converge to $0$ for a generic $K\in L^k, k<\infty$. Adding some regularity/structure to $K$, for instance assuming that it has some Besov regularity, one can deduce a rate of convergence: we discuss the case of Riesz-type kernels in Example~\ref{ex:riesz}.

We also mention that our results hold under a smallness assumption on $ \lVert K\rVert_{L^k}$. This seems unavoidable, since otherwise there could be non-uniqueness of the stationary laws, which would make any uniform-in-time propagation of chaos impossible. However, playing with the strength of the coercive drift $F$ allows for larger values of $ \lVert K\rVert_{L^k}$, as we explain in Remark~\ref{rk:smallnessK}.

\medskip

Let us now describe how our results relate to the existing literature.
First, for moderate interacting particle systems, propagation-of-chaos results are
 currently available only on finite time intervals. The first 
 rigorous derivation of a limiting equation for a moderately 
 interacting particle system dates back to \citet{O85}, in the context
  of the viscous porous medium equation. Since then, numerous 
  propagation-of-chaos results have been established for 
  mollified interacting particle systems, initially for regular 
  interaction kernels; see, for instance, \citet{MC87} and 
  \citet{JM98}. Results were obtained also for singular 
  interactions. Indeed, using a semigroup approach, \citet{ORT} obtained 
  quantitative propagation-of-chaos estimates for interaction kernels
   belonging to $L^p(B)\cap L^q(B^c)$, allowing for sub-Coulomb,
    Coulomb, and super-Coulomb singularities up to order $d-1$. Note that
    the method of~\cite{ORT} also covers the general 
    $L^k$, $k>d$, interaction kernels considered in this work.
    A similar result for the Burgers equation was established by
     \citet{ORT2} and in the case of kinetic equations for kernels lying in some mixed Besov space of negative
      regularity by~\citet{HJMRZ24}. This framework permits to approximate solutions of nonlinear Fokker-Planck equations by smoothed empirical measures in Lebesgue norms, and is well adapted for numerical approximation~\cite{Cazacu}. Recently, using a relative-entropy method, \citet{CHH25} derived quantitative
      propagation-of-chaos estimates for the viscous porous medium 
      equation and also for interactions with Coulomb-type singularities. We insist on the fact that all the aforementioned results hold on finite time intervals.

Second, finite-in-time propagation of chaos for non-mollified $L^k\, (k>d)$
kernels has been treated in the literature.
On $\R^d$, non-quantitative propagation of chaos has been obtained by \citet{T23} and \citet{HaoRocknerZhang}. 
On the torus $\mathbb{T}^d$, \citet{H22} obtained sharp quantitative propagation of chaos in relative entropy for this class of kernels, with the significant additional condition that the interaction kernel should be divergence free. The compact setting used in \cite{H22} allows to derive an LSI, see~\cite{GBM25}, and it seems possible to derive a uniform-in-time result, see Paragraph~1.1.1 in \cite{H22}.

Third, let us further evoke results on uniform-in-time propagation of chaos in the non-mollified setting. For regular  kernels this has been extensively studied under various and sometimes minimal  contractivity assumptions; see, \cite{Malrieu01,M03,BCV07,CGM08,BGG13} and the references therein. On the torus, for Lipschitz kernels with small
divergence, \citet{LF23} achieved sharp uniform-in-time results using BBGKY hierarchy; \citet{DelarueTse} obtained uniform-in-time weak propagation of chaos for smooth interactions without any convexity, under a stability assumption on the invariant measure of the limiting equation. The case of uniform-in-time propagation of chaos for singular kernels is even more delicate. Continuing with the periodic setting $\mathbb{T}^d$, for $K \in L^1$ which is divergence-free and deriving from a bounded potential, \citet{GBM25} established uniform-in-time propagation of chaos. Using modulated free energy methods, \citet{CRS25,CRS25bis} obtained uniform-in-time propagation of chaos for repulsive Riesz kernels $\nabla |\cdot|^s$ with $s \in (d-2,d)$ and  for log attractive kernels. Moreover, \citet{BBJS26} proved uniform-in-time propagation of chaos for singular kernels (including Coulomb kernel) with a non explicit rate of convergence: the result includes kernels more singular than $L^k(\T^d), k>d$, but that need to be divergence-free, and with exchangeable, lower-bounded initial densities (not directly extendable to the unbounded domains). Interestingly, the divergence-free assumptions allows to avoid a smallness condition on the kernel (as we impose here), but they also show that without such structural assumption, even for a smooth kernel, uniform propagation of chaos may not hold.
On the whole space $\mathbb{R}^d$, there is
 a need for a confinement potential verifying certain regularity assumption of convexity type.
Using a coupling methods, \citet{DEGZ20} prove uniform-in-time 
propagation of chaos for interacting mean-field particle systems
 with possibly non-convex confinement and interaction potentials 
 which are smooth and symmetric. 
Using entropy/energy methods, \citet{RS23} obtained 
global-in-time mean-field convergence for repulsive singular Riesz-type 
diffusive flows of the form $-\log|x|$ or $|x|^{-s}$ (where $0 < s < d-2$).
 For systems exhibiting convex-at-infinity confinement potentials and
  Lipschitz or bounded interactions, \citet{LF23} provided
   sharp uniform-in-time propagation of chaos using BBGKY hierarchy. Moreover,
   establishing and using time-uniform log-Sobolev inequalities, \citet{MRW24},
    established sharp in time propagation of chaos for smooth
    interactions and also for vortex interactions with quadratic
    confinement.

\smallskip

Lastly, we would like to stress that the ergodicity of the associated
 McKean-Vlasov equation is 
a key ingredient in order to derive uniform-in-time propagation of
 chaos results. For $L^k\,(k>d)$, 
interaction kernels, the ergodicity of the McKean-Vlasov equation has been recently 
established by \citet{HW25}. In particular,
they use semigroup estimates and a
  contraction argument to obtain global-in-time 
  results. At the level of interacting particle 
  systems, contractive semigroup techniques 
  have also been employed to study long-time
   behavior by \citet{BGP14}, for the reversible
    mean-field plane-rotator model, whose 
    limiting equation possesses a stable manifold
     of stationary solutions, and by \citet{LP21},
      in a non-reversible setting in which the 
      limiting McKean-Vlasov equation admits a 
      stable periodic orbit rather than a 
      stationary equilibrium.

\paragraph{Organisation of the paper.}
The rest of the paper is organised as follows. Section~\ref{main-result}
introduces the notation, the standing assumptions and states the main results,
Theorem~\ref{thm:main} (for moderately interacting particles), Theorem~\ref{thm:main-non-mollified} (for classical particles) and Corollary~\ref{cor:pathwise-poc}, together with the scheme of their proofs. Section~\ref{sec:pde} gathers the regularity estimates on the nonlinear Fokker-Planck equation.
Section~\ref{sec:micro-macro} controls the error between the microscopic and macroscopic mild
formulations, and Section~\ref{sec:int-sto} the stochastic-integral and moment terms.
Theorem~\ref{thm:main} is then proved in Section~\ref{sec:main-res}, 
Theorem~\ref{thm:main-non-mollified} is proved in Section~\ref{sec:non-mollified}, and finally Section~\ref{sec:corollary} contains the proof of Corollary \ref{cor:pathwise-poc}.
The well-posedness of the McKean-Vlasov equation is established in
Appendix~\ref{app:ex-uniq}, 
some auxiliary lemmas are proved in Appendix~\ref{app:lemmas}, and the partial Girsanov estimates are provided in Appendix~\ref{annex:part-girsanov}.

\paragraph{Acknowledgments.}
This work is partially supported by the SDAIM project jointly funded by the S\~ao Paulo Research Foundation (FAPESP) and the French National Research Agency (ANR) through the grant ANR-22-CE40-0015.

\paragraph{Declaration of AI usage.} AI (Claude and ChatGPT) was initially used only for minor typesetting and suggested the trick involving the function $g$ in Section \ref{sec:main-res} (which shortened the length of some computations).
In the final stage of our preparation of the paper, we submitted it to Claude Opus 5.5, which found typos, a few minor issues (in particular a mistake in the treatment of the case $k=\infty$ in Proposition \ref{prop:lip-bound}) and suggested an improvement on the rate of convergence of the averaged stochastic convolutions when $k<2$ ($d=1$), which corresponds to the last part of the proof of Proposition~\ref{lemma:stochastic-integral}, allowing precisely to obtain $N^{-\min(\alpha,\frac 1 2 - 
\frac{d \alpha}{2}-\varepsilon)}$ instead of 
$N^{-\min(\alpha,\frac 1 2 - 
\frac{d \alpha}{k \wedge 2}-\varepsilon)}$.

\section{Main result}\label{main-result}
Here we present the notations and definitions used throughout the paper,
as well as the main result of this work, Theorem~\ref{thm:main} followed 
by the scheme of its proof.
\subsection{Notations and definitions}

Throughout the paper, we use the following notations and conventions:
\begin{itemize}
    \item For $k \in (1, \infty]$, denote by $k^*$ the conjugate
    exponent of $k$ defined by $\frac{1}{k} + \frac{1}{k^*} = 1$.
    \item $\mathcal{L}_X$: the law of the random variable $X$.
    \item $\mathcal{P} (\R^d)$: the set of probability measures
     on $\R^d$ equipped with the weak topology.
    \item $\mathcal{P}_p(\R^d)$: the set of probability measures 
    on $\R^d$ with finite $p$-th moment.
    \item $\mathcal{C}([0,T];\mathcal{P}(\R^d))$: 
    the set of weakly continuous functions from $[0,T]$ to 
    $\mathcal{P}(\R^d)$.
    \item $L^k(\R^d)$: the set of $k$-integrable functions on 
    $\R^d$, and $\|\cdot\|_{L^k}$ the associated norm.
    \item Denote by $\tilde{L}^k(\R^{d})$ the set of 
    locally $k$-integrable functions on $\R^{d}$ in the following
     sense:\\ 
    $\sup_{z \in \R^d } \| \mathbbm{1}_{B(z,1)} f \|_{L^k(\R^{d})} < 
    \infty$, and 
    let $\|\cdot\|_{\tilde{L}^k} \coloneqq \sup_{z \in \R^d } \| \mathbbm{1}_{B(z,1)} \cdot \|_{L^k(\R^{d})} $ be the associated norm.
    \item For $\mu \in \mathcal{P}(\R^d)$, let \(\|\mu \|_{k^*}\coloneqq 
    \sup_{\|f\|_{L^k} \leq 1} |\langle \mu, f \rangle|\)
    be the dual norm of $L^k(\R^d)$, which can be identified with the $L^{k^*}(\R^d)$-norm of the density of $\mu$ with respect to the Lebesgue measure, if it exists.
    \item  For $\mu \in \mathcal{P}(\R^d)$, let 
    $\|\mu \tilde\|_{{k}^*}\coloneqq \sup_{\|f\|_{\tilde{L}^k} 
    \leq 1} |\langle \mu, f \rangle|$, the dual norm of 
    $\tilde{L}^k(\R^d)$.
    \item $W_p^s(\R^d)$: the Sobolev space of order 
    $s$ with integrability $p$ on $\R^d$
    and $\|\cdot\|_{W_p^s}$ the associated natural norm.
    \item For $\beta \in \R$ and $p \geq 1$, define the Bessel potential operator $(I-\Delta)^\beta$ on tempered distributions by
    \begin{equation}
    \label{eq:defBessel}
    (I-\Delta)^\beta u \coloneqq \mathcal{F}^{-1}\left((1+|\cdot|^2)^{\beta}
    \mathcal{F}u(\cdot)\right) ,
    \end{equation}
    where $\mathcal{F}$ denotes the Fourier transform and $\mathcal{F}^{-1}$ its inverse. Let $H_p^\beta(\R^d)$ be the Bessel potential space defined by
    \[H_p^\beta(\R^d)
    \coloneqq
    \left\{u \text{ tempered distribution}:
    ~ (I-\Delta)^{\frac{\beta}{2}} u \in L^p(\R^d)
    \right\},\]
  endowed with the norm \(\|u\|_{H_p^\beta} \coloneqq
    \big\|(I-\Delta)^{\frac{\beta}{2}} u
    \big\|_{L^p(\R^d)}\).
    
    \item For $\mu, \nu \in \mathcal{P}_m(\R^d)$, the 
    $m$-Wasserstein distance $\mathbb{W}_m$ is defined by
\begin{equation*}
    \mathbb{W}_m(\mu,\nu) \coloneqq \inf_{\pi \in \Pi(\mu,\nu)} \left(
    \int_{\R^d \times   \R^d} |x-y|^m \, \dd \pi(x,y) \,
   \right)^{\frac{1}{m}},
    \end{equation*}
    where $\Pi(\mu,\nu)$ is the set of couplings between 
    $\mu$ and $\nu$.
    \item For $\mu,\nu$ two probability measures 
    in $\mathcal{P}_{2}(\R^{dN})$,
    the renormalised $2$-Wasserstein distance $\mathbb{W}^N_2$ 
     is defined by
\begin{equation*}
\mathbb{W}^N_2(\mu,\nu) \coloneqq \inf_{\pi \in \Pi(\mu,\nu)} \bigg(
    \int_{(\R^{d})^N \times (\R^{d})^N} d^N(x,y)^2 \, \dd \pi(x,y) \, 
   \bigg)^{\frac{1}{2}},
\end{equation*}
where $d^N(x,y) \coloneqq 
\left(\frac{1}{N}\sum_{i=1}^N |x_i - y_i|^2\right)^{1/2}$.
\end{itemize}

\subsection{Uniform-in-time propagation of chaos}

Consider the following McKean-Vlasov equation on $\R^d$:
\begin{equation}\label{mckean-vlasov}
    \dd X_t = -F(X_t) \dd t- (K * \mathcal{L}_{X_t})(X_t) \dd t+ \sqrt{2} \dd W_t, \quad t>0,
\end{equation}
where \((W_t)_{t \geq 0}\) is a $d$-dimensional standard Brownian motion 
on a complete filtered probability space $(\Omega, \mathcal{F},
(\mathcal{F}_t)_{t \geq 0}, \mathbb{P})$. Here, $K$ and $F$ satisfy the following assumptions:

\begin{assumptions}\label{assum:potentials}
\leavevmode\par\vspace{-6.9pt}
\begin{enumerate}[label=(\roman*)]
    \item \(F : \R^d \to \R^d\) is Lipschitz continuous and 
    \(\exists \lambda >0, \forall x,y \in \R^d\), 
    \(\langle F(x) - F(y), x-y \rangle \geq \lambda |x-y|^2\).
    \item $K : \R^d \to \R^d$ is in $L^k(\mathbb{R}^d)$ for some \(k>d\).
\end{enumerate}
\end{assumptions}

Solutions of the McKean-Vlasov equation \eqref{mckean-vlasov}
will be considered in the following sense:
\begin{definition}[Global strong solution]\label{def-mckean}
    Fix $(W_t)_{t \geq 0}$ a $d$-dimensional standard
    Brownian motion on a complete filtered probability 
    space $(\Omega, \mathcal{F}, (\mathcal{F}_t)_{t \geq 0}, 
    \mathbb{P})$.
    We call $(X_t)_{t \geq 0}$ a global strong solution 
    of \eqref{mckean-vlasov} if it is an $(\mathcal{F}_t)$-adapted continuous process 
    taking values in $\R^d$ and such that the following conditions hold:
	\begin{enumerate}[label=(\roman*)]
	\item The initial distribution \(\mathcal{L}_{X_0} \in \mathcal{P}(\R^d)\)
    and for any $T \in (0, \infty),\, $
    $\mathbb{E}[\int_0^T (|F(X_s)| + 
    |K * \mathcal{L}_{X_s}(X_s)|) \dd s] < \infty$ and
     $\mathbb{P}$-a.s., 
    \begin{equation*}
        X_t = X_0 - \int_0^t F(X_s) \dd s- 
        \int_0^t (K * \mathcal{L}_{X_s})(X_s) \dd s+ \sqrt{2} W_t, \quad t \in [0,T].
    \end{equation*}
	
	\item For all $T \in (0, \infty)$, 
    $(\mathcal{L}_{X_t})_{t \in [0,T]} \in \mathcal{C}_{T}$
    where 	
    \begin{equation}
    \label{def:C_T}
    \mathcal{C}_{T} \coloneqq \{ \mu \in \mathcal{C}{([0,T];\mathcal{P}(\R^d))} : 
    \sup_{t \in (0,T]} t^{\frac{d}{2k}}\|\mu_t\|_{k^*} < \infty \}.
    \end{equation}
	\end{enumerate}
\end{definition}

 \begin{remark}
    In Definition~\ref{def-mckean}, condition (ii) together with Young's convolution inequality and 
    $\int_0^T s^{-\frac{d}{2k}}\,\dd s < \infty$,
    ensures that $\mathbb{E}[\int_0^T  
    |K * \mathcal{L}_{X_s}(X_s)| \dd s] < \infty$ is verified in (i). Hence
     it is enough to verify $\mathbb{E}[\int_0^T  |F(X_s)| \dd s] < \infty$ in (i).
 \end{remark}

We now state the well-posedness of~\eqref{mckean-vlasov}, whose law solves
the Fokker-Planck equation~\eqref{nlfp-intro} in the sense of distributions.
The existence part relies entirely on Theorem~2.1(4) in~\cite{HRW25}: under
Assumption~\ref{assum:potentials}, the initial datum $u_0$ may be any element of $\mathcal{P}(\R^d)$,
with no additional density or moment requirement. The only difference is that \cite{HRW25} work in the localized Lebesgue spaces $\tilde{L}^k(\R^d)$, whereas we use the usual $L^k(\R^d)$ spaces. Since $\|\mu\|_{k^*}\le\|\mu\tilde\|_{k^*}$, the solution from~\cite{HRW25} already belongs to our solution class $\mathcal{C}_{T}$, so existence is covered by~\cite{HRW25}.

Uniqueness, on the other hand, is not a direct consequence of \cite{HRW25}.
Because our class $\mathcal{C}_{T}$ is defined through the non-localized norm
$\|\cdot\|_{k^*}$, it is strictly larger than the class in which their
uniqueness statement holds; one must therefore rule out solutions living in this
larger class. This requires an independent proof, which is carried out in Appendix~\ref{app:ex-uniq}. The fact that the initial condition may be any probability measure also allows a direct comparison with earlier well-posedness results obtained under stronger assumptions on $u_0$.
\begin{proposition}\label{prop:ex-uniq}
Under Assumption~\ref{assum:potentials},  for any $u_0 \in 
\mathcal{P}(\R^d)$, there exists a unique global strong solution to \eqref{mckean-vlasov} with $\mathcal{L}_{X_0}=u_0$, in the sense of Definition~\ref{def-mckean}.
\end{proposition}

\begin{remark}
    Since for $k \in (d,\infty)$ and for all $t>0$, the law of $X_t$ satisfies the following bound 
    \[ \sup_{f \in L^k(\R^d), 
    \|f\|_{L^k} \leq 1} |\langle \mathcal{L}_{X_t}, f \rangle| 
     < \infty,\]
    by the Riesz representation theorem
     it means that $\mathcal{L}_{X_t}$ has a density $u_t \in L^{k^*}(\R^d)$ 
    with respect to the Lebesgue measure. For 
    $k = \infty$, the kernel is bounded and $\mathcal{L}_{X_t}$ has a density $u_t \in L^1(\R^d)$ for all $t>0$.
   In the following, we identify the 
measure $\mathcal{L}_{X_t}$ with its density $u_t$, and we use the
notation $u_t$ for both. 
 \end{remark}

We recall that the system of $N$ mollified 
interacting particles is defined by \eqref{particle-system-intro},
with $(W^i)_{i \in \N}$ independent standard Brownian motions 
on a complete filtered probability space \((\Omega, \mathcal{F},
(\mathcal{F}_t)_{t \geq 0}, \mathbb{P})\) and
 $(V^N)_{N \in \N}$ a sequence of 
mollifiers defined by \[V^N(\cdot)\coloneqq N^{d \alpha}
 V(N^\alpha \cdot),\] for some 
$\alpha > 0$ and $V$ a smooth probability density with compact support 
on $\R^d$. 
The drift coefficients of the system of SDEs \eqref{particle-system-intro} being regular thanks to the mollification, the strong well-posedness of the particle system is ensured for any initial condition.

In all the following, we denote by $k$ the integrability parameter of the interaction kernel $K$ defined 
in Assumption \ref{assum:potentials}(ii). The assumptions 
on the mollification parameter $\alpha$ and on the initial 
conditions $(X_{0}^{1,N},\dots, X_{0}^{N,N})$ are detailed below.
\begin{assumptions}\label{assum:alpha}
The parameter $\alpha$ satisfies $ 0<\alpha<\frac{1}{d \vee 2}$.
\end{assumptions}

\begin{assumptions}\label{assum:initial} The initial conditions 
    verify
     $u_0 \in \mathcal{P}_2(\mathbb{R}^d),\;
\mathcal{L}_{X_0^N} \in \mathcal{P}_2(\mathbb{R}^{dN}),$ and we denote 
\(P \coloneqq  \int_{\R^d} |x|^2 u_0(\dd x)\).
Moreover, if $k>2$,
for $q = 2 \vee (d+1)(1-\frac{2}{k})$, assume that
\[ Q \coloneqq  \sup_{N \in \N} \max_{1 \leq i \leq N} \E[ |X_0^{i,N}|^q] < \infty.\]
\end{assumptions}

We can now state the first main result of this work, 
which gives a rate in the uniform-in-time propagation of 
chaos on the law of the particle system 
\eqref{particle-system-intro}
towards the solution of the nonlinear PDE~\eqref{nlfp-intro}.
\begin{theorem}
\label{thm:main}
Consider the mollified interacting particle system \eqref{particle-system-intro}.
Suppose that Assumptions~\ref{assum:potentials},~\ref{assum:alpha} and~\ref{assum:initial} hold.
Then there exists a constant
\(M=M(d,k,F)\in(0,1)\)
such that whenever
\(\|K\|_{L^k}\leq M,\)
for any $\varepsilon>0$, there exist constants
\(C=C(d,k,F,\alpha,V,P,Q,\varepsilon)>0\) and
\(c=c(d,k,F,M)>0\)
such that for any $t\geq0$ and any $N\geq1$,
\[\mathbb{W}_2^N\bigl(\mathcal{L}_{X_t^N},u_t^{\otimes N}\bigr)
\leq C\Big(e^{-ct}\mathbb{W}_2^N\bigl(\mathcal{L}_{X_0^N},u_0^{\otimes N}\bigr)
+N^{-\min\left(\alpha,\frac12-
\frac{d\alpha}{2}-\varepsilon\right)}
\Big).\]

Moreover, denoting by
\(u_t^N\coloneqq V^N\ast \big(\frac1N\sum_{i=1}^N\delta_{X_t^{i,N}} \big)\),
we have, for any $t>0$ and any $N\geq1$,
\begin{align*}
    \left\|\|u_t^N-u_t\|_{k^*}\right\|_{L^2(\PP)}
    &\leq C(1\wedge t)^{-\frac12-\frac{d}{2k}}
    \left( e^{-ct}
    \mathbb W_2^N\left(\mathcal L_{X_0^N},u_0^{\otimes N}\right)
    + N^{-\min\left(\alpha,\frac12-\frac{d\alpha}{2}-\varepsilon\right)}
    \right).
\end{align*}
\end{theorem}

The proof of Theorem~\ref{thm:main} will be carried out in Section~\ref{sec:main-res} and the strategy of proof is explained in Subsection~\ref{sec:strategy-proof}.

\begin{remark} 
Ignoring the arbitrarily small loss $\varepsilon$, the exponent
\(\min\left(\alpha,\frac12-\frac{d\alpha}{2}\right)\)
is maximized for
\(\alpha^\ast=\frac{1}{d+2}.\)
Then for $\alpha^\ast=\frac1{d+2}$, the corresponding
rate is $N^{-\frac1{d+2}+\varepsilon}$.
\end{remark}

\begin{remark}
\label{rk:smallnessK}
More precisely, the parameter $M$ depends on $F$ only through the Lipschitz constant $|F|_{Lip}$ and the coercivity constant $\lambda$, so that $M\equiv M(d,k,\lambda,|F|_{Lip})$. 
The condition $\|K\|_{L^k}\le M<1$ is not a fundamental restriction: it comes from working on unit time intervals. Performing a space-time rescaling $Z^\lambda_s \coloneqq \lambda^{1/2}X_{s/\lambda}$ (this works also for $X^{i,N}$, but we omit the details) shows that Theorem~\ref{thm:main} holds whenever
$\|K\|_{L^k}\le \lambda^{\frac12(1-\frac dk)}\,M(d,k,1,|F|_{Lip}/\lambda)$. Hence if $F$ is well conditioned so that $|F|_{Lip}/\lambda \sim 1$, choosing $\lambda$ large enough allows for larger values of $\|K\|_{L^k}$.
\end{remark}

For the non-mollified interacting particle system \eqref{particle-system-intro-gen}, we have the following similar result.

\begin{theorem}\label{thm:main-non-mollified}
Consider the non-mollified interacting particle system \eqref{particle-system-intro-gen}.
Suppose that Assumptions~\ref{assum:potentials},
\ref{assum:alpha}, \ref{assum:initial} hold and that $k \in(2\vee d,\infty]$.
Then there exists a constant
\(M=M(d,k,F)\in(0,1)\)
such that whenever
\(\|K\|_{L^k}\leq M,\)
for any $\varepsilon>0$, there exist constants
\(C=C(d,k,F,\alpha,V,P,Q,\varepsilon)>0\) and
\(c=c(d,k,F,M)>0\)
such that for any $t\geq0$ and any $N\geq1$,
\[\mathbb{W}_2^N\bigl(\mathcal{L}_{Y_t^N},u_t^{\otimes N}\bigr)
\leq
C\Big(e^{-ct}
\mathbb{W}_2^N\bigl(
\mathcal{L}_{Y_0^N},u_0^{\otimes N}
\bigr)
+N^{-\min\left(\alpha,\frac12-
\frac{d\alpha}{2}-\varepsilon\right)}
+\|K-K\ast V^N\|_{L^k}
\Big).\]

Moreover, denoting by
\(v_t^N\coloneqq V^N\ast \big(\frac1N\sum_{i=1}^N\delta_{Y_t^{i,N}} \big)\), we have for any $t>0$ and any $N\geq1$,
\begin{align*}
    \left\|\|v_t^N-u_t\|_{k^*}\right\|_{L^2(\PP)}
    &\leq
    C(1\wedge t)^{-\frac12-\frac{d}{2k}}
    \left(
    e^{-ct}
    \mathbb W_2^N\left(\mathcal L_{Y_0^N},u_0^{\otimes N}\right)
    +
    N^{-\min\left(\alpha,\frac12-\frac{d\alpha}{2}
    -\varepsilon\right)}
    +\|K-K\ast V^N\|_{L^k}
    \right).
\end{align*}
\end{theorem}
The proof of Theorem~\ref{thm:main-non-mollified} is 
carried out in Section~\ref{sec:non-mollified} and the strategy of proof, similar to that of Theorem~\ref{thm:main}, is explained in Subsection~\ref{sec:strategy-proof}.

\begin{remark}\label{k-infty} 
Under the assumptions of Theorem~\ref{thm:main-non-mollified}, with
$k \in (2 \vee d, \infty)$, the term $\|K - K \ast V^N\|_{L^k}$ tends to zero as $N \to \infty$. The theorem therefore yields uniform-in-time propagation of chaos, with quantitative rates available under additional assumptions on $K$. 
However, for a general $K \in L^\infty$, without any additional hypothesis (continuity for example) it is not true that 
$\|K - K \ast V^N\|_{L^\infty}$ converges to $0$ as $N$ goes to $\infty$.
We believe that the proof
in Section~\ref{sec:non-mollified} could be adapted to
replace the $L^\infty$ norm with a weaker one in which
$K - K \ast V^N$ converges to zero as $N \to \infty$, but we do not pursue this direction to keep the paper focused on the above results.
\end{remark}

The following corollary is deduced from Theorems~\ref{thm:main} and \ref{thm:main-non-mollified}, and its proof is postponed to Section~\ref{sec:corollary}, after
the proofs of the theorems.

\begin{corollary}
\label{cor:pathwise-poc}
\begin{enumerate}
\item[(i)]Consider the particle system \eqref{particle-system-intro}.
Let $(X^i)_{\{i=1,\ldots,N\}}$ be $N$ i.i.d strong solutions of the McKean-Vlasov SDE \eqref{mckean-vlasov}, constructed on the same
probability space and driven by the same Brownian motions
$(W^i)_{\{i=1,\ldots,N\}}$ as the corresponding particles
$(X^{i,N})_{\{i=1,\ldots,N\}}$ and with initial law $u_0$.

Under the assumptions of Theorem~\ref{thm:main}, for any
$\varepsilon>0$ there exist constants $C,c>0$ such that for any
$t\geq0$ and any $N\geq1$,
\[\Big(\frac1N\sum_{i=1}^N \E\left[|X_t^{i,N}-X_t^i|^2\right]\Big)^{\frac12}
\leq C\bigg(e^{-ct}\Big(\frac1N\sum_{i=1}^N
\E\left[|X_0^{i,N}-X_0^i|^2\right]\Big)^{\frac12}
+ N^{-\min\left(\alpha,\frac12-\frac{d\alpha}{2}-\varepsilon\right)}\bigg).\]

\item[(ii)] Consider the particle system \eqref{particle-system-intro-gen}.
Let $(X^i)_{\{i=1,\ldots,N\}}$ be $N$ i.i.d strong solutions of the McKean-Vlasov SDE \eqref{mckean-vlasov}, constructed on the same
probability space and driven by the same Brownian motions
$(W^i)_{\{i=1,\ldots,N\}}$ as the corresponding particles
$(Y^{i,N})_{\{i=1,\ldots,N\}}$ and with initial law $u_0$.

Under the assumptions of Theorem~\ref{thm:main-non-mollified}, for any $\varepsilon>0$ there exist constants $C,c>0$ such that for any
$t\geq0$ and any $N\geq1$,
 \begin{align*}
 \bigg(\frac1N\sum_{i=1}^N \E\left[|Y_t^{i,N}-X_t^i|^2\right]\bigg)^{\frac12}
 &\leq C\biggl(e^{-ct}
 \Big(\frac1N\sum_{i=1}^N \E\left[|Y_0^{i,N}-X_0^i|^2\right]\Big)^{\frac12}\\
 &\quad +N^{-\min\left(\alpha,\frac12-\frac{d\alpha}{2}-\varepsilon\right)}+\|K-K*V^N\|_{L^k}\biggr).
 \end{align*}
\end{enumerate}
\end{corollary}

\begin{remark}
As another,  immediate, consequence of Theorem~\ref{thm:main} (respectively
Theorem~\ref{thm:main-non-mollified}), 
we derive a bound for
\[\limsup_{t\to\infty}
\mathbb{W}_2^N\bigl(\mathcal{L}_{X_t^N},u_t^{\otimes N}\bigr)\quad
(\text{respectively for } \limsup_{t\to\infty}
\mathbb{W}_2^N\bigl(\mathcal{L}_{Y_t^N},u_t^{\otimes N}\bigr)).\]
Moreover, since $u_t\to u_\infty$ in $\mathbb{W}_2$ as $t\to\infty$ (see \cite{HW25}), if the mollified (respectively non-mollified) $N$-particle system  converges in $\mathbb{W}_2^N$
to an invariant law $\pi^N$, then we can derive an explicit bound for
\(\mathbb{W}_2^N\bigl(\pi^N,u_\infty^{\otimes N}\bigr).\)
\end{remark}

\begin{example}[Riesz-like singular kernels]
\label{ex:riesz}
Let $d\geq2$, $s\in(0,1)$, and let
$\chi\in C_c^\infty(\R^d)$ be such that $\chi=1$ in a neighborhood
of the origin and consider
\[K(x) \coloneqq a\,\chi(x)\frac{x}{|x|^{s+1}},.\]
Then, for every \(d<k<\frac{d}{s},\)
we have $K\in L^k(\R^d)$. Taking $|a|$ sufficiently small,
the smallness assumption of Theorem~\ref{thm:main-non-mollified}
is satisfied. For a fixed $h \in \R^d$,
splitting $\R^d$ into
$\{|\cdot|\leq |h|\}$ and $\{|\cdot|>|h|\}$, and applying the
mean value theorem on the latter region gives the 
following translation estimate
\(\|K(\cdot-h)-K\|_{L^k}\lesssim |h|^{\frac{d}{k}-s}\).
Consequently, by the definition of $V^N$ and Minkowski's inequality,
\[\|K-K*V^N\|_{L^k}\lesssim N^{-\alpha(\frac{d}{k}-s)}.\]
Plugging this rate into
 Theorem~\ref{thm:main-non-mollified} gives,
for any $\varepsilon>0$,
\[\mathbb W_2^N
\left(\mathcal L(Y_t^N),u_t^{\otimes N}\right)
\leq C\left(e^{-ct}\mathbb W_2^N
\left(\mathcal L(Y_0^N),u_0^{\otimes N}\right)
+ N^{-\min\left(\alpha,
\frac12-\frac{d\alpha}{2}-\varepsilon,
\alpha\left(\frac dk-s\right)\right)}\right).\]
The optimal $\alpha$ is given by
\(\alpha_*=\frac{1}{d+2(\frac dk-s)},\)
which leads to the rate
\(N^{-\frac{\frac dk-s}{d+2(\frac dk-s)}+\varepsilon}.\)
By letting $k$ be arbitrarly close to $d$,
$\frac dk-s$ can be taken arbitrarily close to $1-s$,
and we obtain the following best rate: 
\[N^{-\frac{1-s}{d+2-2s}+\varepsilon}.\]
\end{example}

\subsection{Scheme and organisation of the proof}\label{sec:strategy-proof}

The proof of Theorem~\ref{thm:main} proceeds in several steps.
First, we introduce a synchronous coupling between the particle 
system and the limiting nonlinear process. This coupling 
yields the estimate
\begin{equation}\label{coupling}
\mathbb{W}_2^N \bigl(\mathcal{L}_{X_t^N},u_t^{\otimes N}\bigr)^2
\leq
\frac1N\sum_{i=1}^N \E \big[|X_t^{i,N}- X_t^i|^2\big],
\end{equation}
where $(X_t^i)_{i \in \{1,\ldots,N\}}$ are independent copies of the
limiting process, and $X^i$ has the same driving noise as $X^{i,N}$.

Next, applying It\^o's formula to the coupled dynamics,
restricting to the time interval $t\in(0,1]$
and using the Lipschitz continuity of $K \ast u_t$, with Lipschitz constant $C \|K\|_{L^k} t^{-\frac12-\frac{d}{2k}}$ from Proposition~\ref{prop:lip-bound}, we 
derive for all $i \in \{1,\ldots,N\}$ an estimate for
\(\E \big[|X_t^{i,N}- X_t^i|^2\big]\)
in terms of 
\(\|u_t^N-u_t\|_{k^*}.\) Note that Proposition~\ref{prop:lip-bound} is a key ingredient and its proof uses 
results of \cite{HRW25}.
Precisely, we obtain
 \begin{align}
 \label{eq:scheme} 
 \frac{\dd}{\dd t} |X_t^{i,N} - X_t^i|^2 
 \leq 2(C \|K\|_{L^k} t^{-\frac12-\frac{d}{2k}} -  \lambda) |X_t^{i,N} -
  X_t^i|^2 +
 2 \|K\|_{L^k} \|u_t^N - u_t\|_{k^*} |X_t^{i,N} - X_t^i|.
 \end{align}

In order to control \(\|u_t^N-u_t\|_{k^*}\), we exploit the mild formulations 
 satisfied by $u_t^N$ and $u_t$, together with the 
regularization properties of the semigroup associated to \eqref{eq:ou-sde}. 
This leads to a quantitative a.s. bound on
\[\|u_t^N-u_t\|_{k^*}\] 
in terms of the initial error
\(\mathbb{W}_2^N \bigl(\mathcal{L}_{X_0^N},u_0^{\otimes N}\bigr)\),
the particle number $N$ and a stochastic integral term. This 
estimate is detailed in Proposition~\ref{prop:mild-formulation}.

Substituting the bounds of Proposition~\ref{prop:mild-formulation}
into the coupling Inequality~\eqref{coupling}, we observe that
$L^2(\PP)$ estimates for the initial error and for the stochastic integral term are needed.
The initial error is controlled in Lemma~\ref{W1-bound}, and the stochastic integral term is controlled in 
Proposition~\ref{lemma:stochastic-integral}.

Combining these estimates,
we obtain, for any $t\in(0,1]$,
\[\mathbb{W}_2^N \bigl(\mathcal{L}_{X_t^N},u_t^{\otimes N}\bigr)
\leq \rho_t(\|K\|_{L^k})
\,\mathbb{W}_2^N \bigl(\mathcal{L}_{X_0^N},u_0^{\otimes N}\bigr)
+C\,N^{-v},\]
for some explicit constant $v>0$.
The function $\rho_t$ satisfies $\rho_1(\|K\|_{L^k})<1$ whenever
$\|K\|_{L^k}$ is sufficiently small.
Therefore, the 
above estimate provides a strict contraction over a unit time 
interval. We emphasise that the term 
$\|K\|_{L^k}$ appearing in front of $\|u_t^N-u_t\|_{k^*}$ in \eqref{eq:scheme}
is crucial in order to obtain this contraction property when choosing $\|K\|_{L^k}$ 
small enough, similarly to \cite{HW25}.
By iterating this inequality through the 
semigroup property, we obtain the exponential-in-time contraction 
estimate stated in Theorem~\ref{thm:main}.\\

The proof of Theorem~\ref{thm:main-non-mollified} follows the same scheme the proof of Theorem~\ref{thm:main}. Indeed, the only 
difference is that, instead of the mollified empirical measure, we now have the non-mollified empirical measure $\mu_t^N$ in 
the drift term of the particles and hence 
we have the term $\big|(K\ast \mu_t^N - K \ast u_t)(X_t^{i,N})\big|$ in \eqref{eq:scheme}, instead of $\|K\|_{L^k} \|u_t^N - u_t\|_{k^*}.$
This new term is handdled
by introducing the pivot term $K*v_t^N(Y_t^{i,N})$, where $v_t^N$
is the mollified empirical measure
of the non-mollified system of particles.
This leads to controlling on the one hand,
\(\mathbb E \big[  |(K-K*V^N)(Y_t^{i,N}-Y_t^{j,N})|^2\big]\)
and on the other hand,
$\|K\|_{L^k} \left\|\|v_t^N-u_t\|_{k^*}\right\|_{L^2(\PP)}$.

In Step~1 of the proof of proof of Theorem~\ref{thm:main-non-mollified} (see Section~\ref{sec:non-mollified}), we control the first
term \(\mathbb E \left[  |(K-K*V^N)(Y_t^{i,N}-Y_t^{j,N})|^2\right]\).
Precisely, using a partial Girsanov argument, we derive heat-like estimates
for the density of the difference of two particles. This yields,
for $t\in(0,1]$,
\[\Biggl[\E\biggl(\frac1N\sum_{i=1}^N
\left|K\ast(\mu_t^N-v_t^N)(Y_t^{i,N})\right|^2\biggr)\Biggr]^{1/2}\leq
Ct^{-\frac{d}{2k}}\|K-K\ast V^N\|_{L^k}
+C\|K\|_{L^k}N^{-\frac12+\frac{d\alpha}{k}}.\]

In Step~2, we control the second term by adapting the mild
formulation of Section~4 to the non-mollified particle system.
This gives an estimate of $\|v_t^N-u_t\|_{k^*}$ similar to the one
obtained in the mollified case, with additional terms coming from
$K\ast(\mu_t^N-u_t^N)$. More precisely, for $t\in(0,1]$,
\begin{align*}
\left\|\|v_t^N-u_t\|_{k^*}\right\|_{L^2(\PP)}
&\leq
C\Theta(t)\mathbb W_2^N
\left(\mathcal L(Y_0^N),u_0^{\otimes N}\right)\\
&+C\left(\Theta(t)N^{-v}
+t^{\frac12-\frac{d}{k}}\|K-K\ast V^N\|_{L^k}
+\|K\|_{L^k}N^{-\frac12+\frac{d\alpha}{k}}\right),
\end{align*}
where $\Theta(t)$ is a negative power of $t$ and is defined in \eqref{theta}.
The uniform moment estimate needed to control the stochastic integral
is also recovered for the non-mollified particle system by a similar
partial Girsanov argument.

Combining the estimates obtained in Steps~1 and~2 with the coupling
estimate, we recover the same contraction mechanism as in the mollified case,
with the additional error $\|K-K\ast V^N\|_{L^k}$. Taking
$\|K\|_{L^k}$ sufficiently small and iterating the estimate on
successive time intervals of length $1$, we obtain
\[\mathbb W_2^N
\left(\mathcal L(Y_t^N),u_t^{\otimes N}\right)
\leq C\left(e^{-ct}\mathbb W_2^N
\left(\mathcal L(Y_0^N),u_0^{\otimes N}\right)
+N^{-v}+\|K-K\ast V^N\|_{L^k}\right),\]
which gives Theorem~\ref{thm:main-non-mollified}.
\section{Regularity of the nonlinear Fokker-Planck PDE}\label{sec:pde}

The purpose of this section is to establish the spatial regularity of the solution $u$ of the nonlinear Fokker-Planck equation~\eqref{nlfp-intro}
and of the drift $K\ast u_t$ it generates. The main result is the
time-integrable Lipschitz bound of Proposition~\ref{prop:lip-bound}, which is an
important regularity input in the proof of Theorem~\ref{thm:main}. To start with, the only information used on $u$ is that for every $T>0$, it belongs to the class $\mathcal{C}_{T}$, defined in~\eqref{def:C_T}. In particular no smoothness of $u$ is assumed \emph{a priori}, and the regularity is gained through the smoothing properties of the underlying semigroups. 
 
 The section is organised as follows. We first introduce the linear (decoupled) SDE obtained by freezing a general flow of measures $v$ in \eqref{mckean-vlasov}, together with its transition semigroup $P^v$ and the generalised Ornstein-Uhlenbeck semigroup $\hat{P}$, and recall the Duhamel 
formula relating them. We then gather classical smoothing bounds on $\hat{P}$ 
(Lemma~\ref{lem:OU-bounds}) and more recent bounds on $P^v$ from
    \cite{HRW25,HW25} (Lemma~\ref{lem:Pu}), which leads to deriving a mild formulation for $u$. We finally bootstrap these bounds first into Bessel regularity estimate for $u_t$ (Lemma~\ref{prop:lip-bound1}) 
and then into the time-integrable Lipschitz bound for $K\ast u_t$ (Proposition~\ref{prop:lip-bound}).

\subsection{Preliminary results on decoupled semigroups}

Let $(v_{t})_{t\geq 0}$ be a flow of probability measures on $\R^d$,
and consider the associated decoupled SDE
\begin{equation}\label{eq:decoupled-sde}
    \dd X^v_t = -F(X^v_t)\, \dd t- (K * v_t)(X^v_t)\, \dd t+ \sqrt{2}\, \dd W_t, \quad t>0, \  \mathcal{L}_{X^v_0} \in \mathcal{P}(\R^d),
\end{equation}
obtained by freezing the flow $v_t$ in the interaction term. Since
$(v_{t})_{t\geq 0}$ is fixed, \eqref{eq:decoupled-sde} is a linear SDE with a
time-dependent drift $b_t(x) \coloneqq  -F(x) - (K\ast v_t)(x)$.
We work with a general flow $v$, and because the estimates of this section
depend on $v$ only through $K\ast v$, they will be shown to hold for all
$v\in\mathcal{C}_{T}$. The case of interest is $(v_t)_{t\ge0}=(u_t)_{t\ge0}$, where
$u_t=\mathcal{L}_{X_t}$ is the flow of marginals of the unique global strong
solution of the McKean-Vlasov equation \eqref{mckean-vlasov} started from
$u_0\in\mathcal{P}(\R^d)$: for this choice, and with $X^v$ started from $u_0$,
the decoupled SDE \eqref{eq:decoupled-sde} coincides with \eqref{mckean-vlasov}, so
that $X^v=X$ is the McKean-Vlasov process itself. This is the flow to which the
lemmas below will be applied in the proof of Proposition~\ref{prop:lip-bound}.

For a generic $v\in \mathcal{C}_{T}$, the well-posedness of~\eqref{eq:decoupled-sde} follows from the fact that $(K\ast v_{t})_{t\geq 0}$ lies in a suitable Krylov-Röckner class, as we now make precise.

\begin{lemma}\label{lem:decoupled-wp}
Suppose that Assumption~\ref{assum:potentials} hold and
that $X_0 \in L^2(\PP)$. 
For $T>0$, let $v\in\mathcal{C}_{T}$, as defined in~\eqref{def:C_T}.
Then,  for any $q\in[1,\tfrac{2k}{d})$,
\begin{equation}\label{eq:Kastu-LqLinfty}
K\ast v \in L^q\bigl([0,T];L^\infty(\R^d)\bigr).
\end{equation}
In particular there exists $q>2$ for which \eqref{eq:Kastu-LqLinfty} holds, and
consequently the decoupled SDE \eqref{eq:decoupled-sde} admits a unique strong
solution.
\end{lemma}

\begin{proof}
For a.e.\ $x\in\R^d$ and every $t>0$, since $K\in L^k(\R^d)$ and $L^{k^*}$
is the dual of $L^k$,
\begin{equation*}
|K\ast v_t(x)|
= \bigl|\langle v_t, K(x-\cdot)\rangle\bigr|
\le  \|K\|_{L^k}\,\|v_t\|_{k^*},
\end{equation*}
so that $\|K\ast v_t\|_{L^\infty}\le \|K\|_{L^k}\,\|v_t\|_{k^*}$. As $v\in\mathcal{C}_{T}$, 
there is $C>0$ with $\|v_t\|_{k^*}\le C\,t^{-\frac{d}{2k}}$ on $(0,T]$. Thus, for
$q\ge1$,
\[
\int_0^T \|K\ast v_t\|_{L^\infty}^q \,\dd t
\le (C\|K\|_{L^k})^q \int_0^T t^{-\frac{qd}{2k}}\,\dd t,
\quad\text{which is finite if and only if }\ \frac{qd}{2k}<1,
\]
that is $q<\tfrac{2k}{d}$, which proves \eqref{eq:Kastu-LqLinfty}. Since $k>d$ we
have $\tfrac{2k}{d}>2$, so one may fix $q\in(2,\tfrac{2k}{d})$; the exponents
$(p,q)=(\infty,q)$ then satisfy the Ladyzhenskaya-Prodi-Serrin condition
$\frac{d}{p}+\frac{2}{q}<1$. The total drift $b_t=-F-K\ast v_t$ splits into the
singular part $-K\ast v\in L^q([0,T];L^\infty(\R^d))$ and the locally bounded
Lipschitz field $-F$. In particular $b$ is locally in
$L^q([0,T];L^p(\R^d))$. As \cite{KR05} only requires local $L^q_{t} L^p_{x}$ integrability of the drift under the above condition, it provides a
unique local strong solution. Finally, the one-sided
coercivity bound in Assumption~\ref{assum:potentials}(i)
rules out the possibility of finite-time blow-up through 
the same moment estimate.  Indeed, for
the local solution, It\^o's formula yields
\[\dd |X_t|^2
\le \Bigl(-2\lambda |X_t|^2
+2\bigl(|F(0)|+\|K\ast v_t\|_{L^\infty}\bigr)|X_t|
+2d
\Bigr)\,\dd t
+\dd M_t,\]
for some local martingale $M$. Since, by
\eqref{eq:Kastu-LqLinfty},
\(t\longmapsto \|K\ast v_t\|_{L^\infty}\in L^q(0,T)\) for some $q>2$, 
Young's inequality and Grönwall's lemma together with the fact that $X_0\in L^2(\PP)$ give
\(\sup_{t\in[0,T]}\E|X_t|^2<\infty.\)
Hence the local solution cannot explode before time $T$, and therefore extends
to a global strong solution.
\end{proof}

Denote by $P^v_{s,t}$ the semigroup associated to~\eqref{eq:decoupled-sde} and by
$P^{v,*}_{s,t}$ its dual acting on measures. As a reference diffusion, we also
introduce the SDE
\begin{equation}\label{eq:ou-sde}
    \dd Z_t = -F(Z_t)\, \dd t+ \sqrt{2}\, \dd W_t, \quad t>0, \, \mathcal{L}_{Z_0} \in \mathcal{P}(\R^d),
\end{equation}
whose semigroup and dual are denoted by $\hat{P}_{t}$ and $\hat{P}_{t}^*$. By a
slight abuse, we call the solution of~\eqref{eq:ou-sde} the
(generalised) Ornstein-Uhlenbeck process: its semigroup exhibits regularisation
properties similar to those of the classical Ornstein-Uhlenbeck semigroup, and
$F$ plays the role of the gradient of a confining potential. We denote by
$\hat{p}_t(x,y)$ the transition kernel of $\hat{P}_t$, defined by
$\hat{P}_tf(x)=\int_{\R^d}f(y)\hat{p}_t(x,y)\,\dd y$.

The two following lemmas gather estimates on the semigroups $\hat P$ and $P^v$ borrowed from \cite{HRW25,HW25,MPZ21}. We start with smoothing bounds on the (generalised) Ornstein-Uhlenbeck semigroup $\hat P$ and pointwise estimates on its kernel $\hat p$. For the latter, we need as in~\cite{MPZ21} the following Gaussian kernel:
\begin{equation}\label{eq:gaussian-kernel}
g^F_\kappa(t, x, y) \coloneqq  (\kappa^{-1}\pi t)^{-\frac{d}{2}} \exp
\left( \frac{-\kappa|y - \theta_t(x)|^2}{t} \right),
\quad t>0,\ x,y\in\R^d,
\end{equation}
where $\theta_t(x)$ solves $\dd\theta_t(x)=-F(\theta_t(x))\,\dd t$,
$\theta_0(x)=x$. Bounds on the associated semigroup, denoted by $G^{\kappa,F}$, will also be useful.

\begin{lemma}\label{lem:OU-bounds}
Let $(\hat P_t)_{t\ge 0}$ be the semigroup of \eqref{eq:ou-sde} and $\hat p_t$ its
transition kernel. For $T>0$,  the
following holds:
\begin{enumerate}[label=\textup{(\roman*)}]
\item\label{item:heat-bounds}
 there exist $C_1,C_2,C_3,\kappa_1,\kappa_2,\kappa_3>0$ such that  for all $0 < t \le T$,
\begin{equation*}
\left\{
\begin{aligned}
C_1^{-1} g^F_{\kappa_1^{-1}}(t,x,y)
&\leq \hat{p}_t(x,y)
\leq C_1 g^F_{\kappa_1}(t,x,y)\\
\left|\nabla_y \hat{p}_t(x,y)\right|
&\leq C_2 t^{-\frac{1}{2}} g^F_{\kappa_2}(t,x,y)\\
\left|\nabla_x \hat{p}_t(x,y)\right|
&\leq C_3 t^{-\frac{1}{2}} g^F_{\kappa_3}(t,x,y)\, ;
\end{aligned}
\right.
\end{equation*}

\item\label{item:heat-holder}
 there exist $C,\kappa>0$ such that for any $\beta\in(0,1)$, $j\in\{0,1\}$, any $x,y_1,y_2\in\R^d$ and any $0 < t \le T$,
\begin{equation*}
\bigl|\nabla_x^j \hat p_t(x,y_1) - \nabla_x^j \hat p_t(x,y_2)\bigr|
\le C\,\frac{|y_1-y_2|^{\beta}}{t^{\frac{\beta+j}{2}}}
\bigl(g^F_\kappa(t,x,y_1)+g^F_\kappa(t,x,y_2)\bigr)\, ;
\end{equation*}

\item\label{item:sgG-Lp}
 for any $1 \le p_1 \le p_2 \le \infty$ and any $\kappa>0$, there exists $C>0$ such that for all $0 < t \le T$,
\begin{equation*}%
\| G^{\kappa,F}_t\|_{L^{p_1}(\mathbb{R}^d)\to L^{p_2}(\mathbb{R}^d)}
\le C\,t^{-\frac{d}{2}\left(\frac1{p_1}-\frac1{p_2}\right)}
 ~\text{ and }~ \|\nabla G^{\kappa,F}_t\|_{L^{p_1}(\mathbb{R}^d)\to L^{p_2}(\mathbb{R}^d)}
\le C\,t^{-\frac12-\frac{d}{2}\left(\frac1{p_1}-\frac1{p_2}\right)} \, ;
\end{equation*}

\item\label{item:OU-Lp}  for any $1 \le p_1 \le p_2 \le \infty$, there exists $C>0$ such that for all $0 < t \le T$,
\begin{equation*}%
\|\hat P_t\|_{L^{p_1}(\mathbb{R}^d)\to L^{p_2}(\mathbb{R}^d)}
\le C\,t^{-\frac{d}{2}\left(\frac1{p_1}-\frac1{p_2}\right)}
  ~\text{ and }~ \|\nabla \hat P_t\|_{L^{p_1}(\mathbb{R}^d)\to L^{p_2}(\mathbb{R}^d)} 
\le C\,t^{-\frac12-\frac{d}{2}\left(\frac1{p_1}-\frac1{p_2}\right)} \, ;
\end{equation*}

\item\label{item:Phat-Wasserstein} 
for any $k>d$, there exists $C>0$ such that for all $0 < t \le T$ and $\mu,\nu \in \mathcal{P}_{1}$,
\begin{equation*}
\|\hat{P}_t^*(\mu - \nu)\|_{k^*} \leq C t^{-\frac{1}{2}- \frac{d}{2k}}
\mathbb{W}_1(\mu, \nu).
\end{equation*}
\end{enumerate}
\end{lemma}
\begin{proof}
The two-sided Gaussian comparison and the derivative bounds of item~\ref{item:heat-bounds} come directly from \cite[Theorem~1.2(i),(ii),(iv)]{MPZ21}.

In~\ref{item:heat-holder}, the bound for $j=0$ can be deduced by interpolation from \ref{item:heat-bounds} (as detailed in \cite[Lemma~3.3]{HW25}). For $j=1$, this is precisely \cite[Lemma~A.1.(C1)]{MPZ21}.

The estimates in~\ref{item:sgG-Lp} are direct consequences of classical Gaussian $L^p$ bounds, using the fact that $(\theta_{t})_{t\in [0,T]}$ is a family of diffeomorphisms on $\R^d$ verifying $\sup_{t\in [0,T]} \max\big(\lVert \nabla \theta_{t} \rVert_{\infty}, \lVert \nabla \theta_{t}^{-1} \rVert_{\infty} \big)<\infty$.
 
Item~\ref{item:OU-Lp} gives bounds for the (generalised) Ornstein-Uhlenbeck semigroup that can be deduced from the upper bound on $\hat{p}_{t}$ in~\ref{item:heat-bounds}; see the proof of \cite[Lemma~5.3]{HRW25} for a detailed argument in the localized spaces $\tilde L^{p}$. 

Finally, \ref{item:Phat-Wasserstein} can be taken for instance from \cite[Theorem~2.3(1)]{HRW25} with the choice of parameters $p=\infty$ and $q=1$ therein (note that \cite[Theorem~2.3]{HRW25} holds in greater generality for semigroups with singular drifts).
\end{proof}

We next present properties of the semigroup $P^v$ associated to~\eqref{eq:decoupled-sde}.
Recall first, from \cite[Proposition~5.5]{HRW25}, the Duhamel formula
relating $P^v$ to $\hat P$: for any bounded measurable $f \in \mathcal{B}_b(\R^d)$ and any
$0\le s\le t\le T$,
\begin{equation}\label{eq:duhamel}
    P^v_{s,t} f = \hat{P}_{t-s} f - \int_s^t P^v_{s,r} (K * v_r \cdot
    \nabla \hat{P}_{t-r} f)\, \dd r.
\end{equation}

\begin{lemma}\label{lem:Pu}
Suppose Assumption~\ref{assum:potentials} holds. Let $T>0$ and $v\in\mathcal{C}_{T}$. Then the following holds:
\begin{enumerate}[label=\textup{(\roman*)}]
\item\label{item:semigroup-v-bound} for any $1<p_1\le\infty$, there exists $C>0$ such that for any $0\le s<t\le T$,
\begin{equation}\label{eq:Pu-smoothing}
\|P^v_{s,t}\|_{L^{p_1}(\R^d)\to L^{\infty}(\R^d)}
\le C\,(t-s)^{-\frac{d}{2p_1}};
\end{equation}

\item\label{item:Duhamel-mu-v} for any $\mu_0\in\mathcal{P}(\R^d)$, the marginal flow
$\mu_t\coloneqq P^{v,*}_{0,t}\mu_0$ of \eqref{eq:decoupled-sde} started
from $\mu_0$ admits a density (still denoted by $\mu_t$) for any $t\in(0,T]$ that satisfies, for a.e.\
$y\in\R^d$, the Duhamel formula
\begin{equation}\label{eq:duhamel-u}
\mu_t(y) = \int_{\R^d} \hat{p}_{t}(x,y) \, \mu_0(\dd x) -
 \int_0^t \int_{\R^d} \bigl(K\ast v_r(x)\bigr) \cdot \nabla_x \hat{p}_{t-r}(x,y)
  \, \mu_r(x)\, \dd x\,\dd r,
\end{equation}
where $\hat p$ is the transition kernel of $\hat P$.
\end{enumerate}
\end{lemma}
\begin{proof}
By Lemma~\ref{lem:decoupled-wp}, the drift $b_t=-F-K\ast v_t$ of
\eqref{eq:decoupled-sde} has its singular part $K\ast v$ in $L^q([0,T];L^\infty)$ for
some $q>2$, cf.\ \eqref{eq:Kastu-LqLinfty}. In this case the transition semigroup of \eqref{eq:decoupled-sde} inherits the same Gaussian upper bounds as the Ornstein-Uhlenbeck semigroup $\hat P$: by \cite[Proposition~5.4]{HRW25}, one has, in the localized spaces,
\begin{equation}\label{eq:Pu-localized}
\|P^v_{s,t}\|_{\tilde L^{p_1}\to \tilde L^{p_2}}
\le C\,(t-s)^{-\frac{d}{2}\left(\frac1{p_1}-\frac1{p_2}\right)},
\qquad 1<p_1\le p_2\le\infty .
\end{equation}
We only use this estimate with $p_2=\infty$, so in that
case the passage from the localized to the usual Lebesgue spaces is direct
since $\tilde L^{\infty}=L^{\infty}$ and $\|g\|_{\tilde L^{p_1}}\le\|g\|_{L^{p_1}}$,
so that \eqref{eq:Pu-localized} with $p_2=\infty$ gives~\eqref{eq:Pu-smoothing} with the same constant. This proves \ref{item:semigroup-v-bound}.

For \ref{item:Duhamel-mu-v}, set $\mu_t\coloneqq P^{v,*}_{0,t}\mu_0$. First, observe that thanks to the first part of 
the lemma, $ \lVert \mu_{t} \rVert_{p_{1}^*}<\infty$ for any $p_{1}>1$ and $t>0$. Hence $\mu_{t}$ admits a density. 
Applying the Duhamel formula \eqref{eq:duhamel} with $s=0$, testing against $f\in\mathcal{B}_b(\R^d)$ and pairing
with $\mu_0$ gives
\[\int_{\R^d} f(y)\, \mu_t(y)\,\dd y
= \int_{\R^d} \hat P_{t} f(x)\,\mu_0(\dd x)
- \int_0^t \int_{\R^d} \bigl(K\ast v_r(x)\bigr)
\cdot\nabla\hat P_{t-r}f(x)\, \mu_r(x)\,\dd x\,\dd r,\]
where we used $\mu_0(P^v_{0,r}\,\cdot)=\mu_r(\cdot)$ in the last term. Inserting the
kernel representation $\hat P_\tau f(x)=\int_{\R^d}\hat p_\tau(x,y)f(y)\,\dd y$,
so that $\nabla\hat P_\tau f(x)=\int_{\R^d}\nabla_x\hat p_\tau(x,y)f(y)\,\dd y$, and
using Fubini's theorem (justified by the bounds of Lemma~\ref{lem:OU-bounds} and
$v\in\mathcal{C}_T$), the right-hand side equals $\int_{\R^d} f(y)\,\Phi_t(y)\,\dd y$
with $\Phi_t(y)$ the right-hand side of \eqref{eq:duhamel-u}. As $f$ is arbitrary,
\eqref{eq:duhamel-u} holds for a.e.\ $y$. 
\end{proof}

\subsection{Bessel regularity of the PDE and Lipschitz bound on $K\ast u$}

First, we derive a Bessel regularity estimate for the solution $u$ of the Fokker-Planck
PDE \eqref{nlfp-intro}. This is a first step towards the time-integrable 
Lipschitz bound of Proposition~\ref{prop:lip-bound}. The proof is based on the regularity properties of the global strong solution 
of the McKean-Vlasov equation \eqref{mckean-vlasov} and the 
Duhamel formula \eqref{eq:duhamel}.
\begin{proposition}\label{prop:lip-bound1}
Suppose Assumption \ref{assum:potentials} holds and that $u_0\in\mathcal{P}_2(\R^d)$.
Let $(\theta, q) \in (0,1) \times (1,\infty)$ be
 such that $\frac{1}{q^*}+\frac{1}{k}<\frac{2}{d}$.
For any $\beta \in (\theta,1)$, there exists 
$C > 0$ such that any global strong 
solution $u$ of~\eqref{mckean-vlasov} satisfies for any $t \in (0,1]$,
\begin{equation*}
\|u_t\|_{H^\theta_q} \leq Ct^{-\frac{\beta}{2} - \frac{d}{2q^*}},
\end{equation*}
where $C$ does not depend on $u_0$.
Moreover, if $k=\infty$, then for any $\theta\in(0,1)$ and any $\beta \in (\theta,1)$, there exists $C>0$ such that for any $t\in(0,1]$,
\begin{equation}\label{eq:ut-Wtheta1}
    \|u_t\|_{W^\theta_1}
    \leq C
    t^{-\frac{\beta}{2}},
\end{equation}
where $C$ does not depend on $u_0$.
\end{proposition}
\begin{proof}
We apply the lemmas of the previous subsection to the flow $v=u$, the marginals of the McKean-Vlasov solution \eqref{mckean-vlasov}; this is admissible since
$u\in\mathcal{C}_{T}$ for every $T>0$, and for this choice the marginal flow of
\eqref{eq:decoupled-sde} started from $u_0$ is $u$ itself, so that
$\mu_t=P^{u,*}_{0,t}u_0=u_t$ in Lemma~\ref{lem:Pu}\ref{item:Duhamel-mu-v}. Let
 $(\theta, q) \in (0,1) \times (1,\infty)$.
In order to justify that we can bound $\|u_t\|_{H^\theta_q}$ by the integral quantity~\eqref{eq:Bessel-Slobod} below, let us recall that $H^\theta_{q} = F^\theta_{q,2}$, where $F^\theta_{q,2}$ is a Triebel-Lizorkin space, see \cite[p.8]{RunstSickel}. For any $\varepsilon>0$, we have $F^{\theta+\varepsilon}_{q,q}\hookrightarrow F^\theta_{q,2} = H^\theta_{q}$, see \cite[Proposition, p.29]{RunstSickel}. 
Using this embedding and the fact that $F^{\theta+\varepsilon}_{q,q}$ is a Slobodejkij's space (\cite[p.12 and Proposition p.14]{RunstSickel}), it comes that
there exists $C>0$ (depending only on $q,\theta,d,\varepsilon$) such that
\begin{align}
	\label{eq:Bessel-Slobod}
    \|u_t\|_{H^\theta_q} & \leq C \|u_t\|_{L^q} + C
    \left( \int_{\R^d \times \R^d} 
    \frac{|u_t(x) - u_t(y)|^q}{|x-y|^{d + (\theta+\varepsilon) q}}  \dd x\, \dd y
    \right)^\frac{1}{q} \\ 
    & \leq C \|u_t\|_{L^q} + C \left(
        \int_{|x-y| \leq 1} 
        \frac{|u_t(x) - u_t(y)|^q}{|x-y|^{d + (\theta+\varepsilon) q}}  \dd x\, \dd y
    + \int_{|x-y| > 1} \frac{|u_t(x) -
     u_t(y)|^q}{|x-y|^{d + (\theta+\varepsilon) q}} \dd x\, \dd y 
      \right)^\frac{1}{q} .\nonumber
\end{align}
Then by the triangle inequality, by symmetry and using that $d+(\theta+\varepsilon) q>d$,
\begin{align}
	\label{eq:ut-hs}
        \|u_t\|_{H^\theta_q}& \leq C \|u_t\|_{L^q} + C
    \left( \int_{|x-y| \leq 1} 
        \frac{|u_t(x) - u_t(y)|^q}{|x-y|^{d + (\theta+\varepsilon) q}}  \dd x\, \dd y
    + \int_{\R^d} |u_t(x)|^{ q} \int_{ |x-y| > 1} 
    \frac{1}{|x-y|^{d + (\theta+\varepsilon) q}} \dd y\,  \dd x\right)^\frac{1}{q} \nonumber \\
    & \leq C \|u_t\|_{L^q} + C \left(
        \int_{|x-y| \leq 1} 
        \frac{|u_t(x) - u_t(y)|^q}{|x-y|^{d + (\theta+\varepsilon) q}}  \dd x\, \dd y
    \right)^{\frac{1}{q}}.
\end{align}
We first bound $\|u_t\|_{L^q}$ using the ultracontractivity of the decoupled
semigroup.
Indeed, $u_t=P^{u,*}_{0,t}u_0$, so by duality, 
Lemma~\ref{lem:Pu}\ref{item:semigroup-v-bound} (estimate \eqref{eq:Pu-smoothing} with $p_1=q^*$) and 
using that $\|u_0\|_{TV}=1$, there
exists $C>0$ such that for any $t\in(0,1]$,
 \begin{equation}\label{ut-time-reg}
 \|u_t\|_{L^q} = \sup_{ \|f\|_{L^{q^*}} \leq 1}
 |\langle P^{u,*}_{0,t} u_0, f \rangle|
 \leq \|u_0\|_{TV} \|P^u_{0,t}\|_{L^{q^*}\to L^{\infty}}
 \leq C \|u_0\|_{TV} t^{-\frac{d}{2q^*}} = C
 t^{-\frac{d}{2q^*}},
 \end{equation}
 where $C$ does not depend on $u_0$. Indeed, 
 the constant $C$ in \eqref{eq:Pu-smoothing} depends on $u$ via the $L^q([0,T];L^\infty)$ norm of $K\ast u$, which at its turn is bounded by 
 $\|u_t\|_{L^{k^*}}$. Duhamel formula \eqref{eq:duhamel-u}, Lemma~\ref{lem:OU-bounds}\ref{item:sgG-Lp}, together with Grönwall's inequality give that $\|u_t\|_{L^{k^*}}$ is bounded by a constant depending on $\|u_0\|_{TV}=1$.

Now to bound the integral term in~\eqref{eq:ut-hs}, we use that for $t>0$, the density $u_t$ satisfies the Duhamel formula \eqref{eq:duhamel-u}, see Lemma~\ref{lem:Pu}\ref{item:Duhamel-mu-v}.
Besides, Lemma~\ref{lem:OU-bounds}\ref{item:heat-holder} (with $j=0$ for $\hat p$ and $j=1$ for $\nabla_x\hat p$) gives that for almost any $y_1,y_2 \in \R^d$ and any
 $\beta \in (0,1)$, we have
\begin{align*}
    &|u_t(y_1) - u_t(y_2)|\\
    &\quad \leq \int |\hat{p}_{t}(x,y_1) - \hat{p}_{t}(x,y_2)|\, u_0(\dd x) +
    \int_0^t \int u_r(x) |K\ast u_r(x)||\nabla_x \hat{p}_{t-r}(x,y_1) - 
    \nabla_x \hat{p}_{t-r}(x,y_2)|\, \dd x\, \dd r\\
    &\quad \leq C |y_1-y_2|^{\beta} t^{-\frac{1}{2}\beta} \int 
     (g_k^F(t,x,y_1)+ g_\kappa^F(t,x,y_2))\, u_0(\dd x)\\
    &\qquad + C |y_1-y_2|^{\beta} \int_0^t 
    (t-r)^{-\frac{1}{2}(1+\beta)} \int u_r(x) |K\ast u_r(x)| 
    (g_\kappa^F(t-r,x,y_1) + g_\kappa^F(t-r,x,y_2))\, \dd x\, \dd r\\
    &\quad\leq C |y_1-y_2|^{\beta} t^{-\frac{1}{2}\beta}
     \int  (g_\kappa^F(t,x,y_1)+ g_\kappa^F(t,x,y_2))\, u_0(\dd x)\\
    & \qquad+ C |y_1-y_2|^{\beta} \int_0^t 
    (t-r)^{-\frac{1}{2}(1+\beta)} \|K\|_{L^k} \|u_r\|_{L^{k^*}}  
    \int u_r(x) (g_\kappa^F(t-r,x,y_1) + g_\kappa^F(t-r,x,y_2))\,  \dd x\, \dd r.
\end{align*}
Using the above inequality, inequality \eqref{ut-time-reg} and 
for $1>\beta>\theta+\varepsilon >0$,
 it comes that
\begin{align*}
    &\left(\int_{|x-y| \leq 1} 
        \frac{|u_t(x) - u_t(y)|^q}{|x-y|^{d + (\theta+\varepsilon) q}} \dd x\,
         \dd y  \right)^\frac{1}{q}\\
    & \leq C t^{-\frac{1}{2}\beta} \left(\int_{|x-y| \leq 1} \frac{|x-y|^{\beta q}}{|x-y|^{d + (\theta+\varepsilon) q}}  \Big( \int (g_\kappa^F(t,z,x)+ g_\kappa^F(t,z,y))\, u_0(\dd z) \Big)^q  \dd x\, \dd y\right)^\frac{1}{q} \\
&\, + C \bigg(\int_{|x-y| \leq 1} \frac{|x-y|^{\beta q}}{|x-y|^{d + (\theta+\varepsilon) q}} 
	\Big(\int_0^t \frac{\|K\|_{L^k} \|u_r\|_{L^{k^*}}}{(t-r)^{\frac{1}{2}(1+\beta)}} \int u_r(z) (g_\kappa^F(t-r,z,x) + g_\kappa^F(t-r,z,y)) \dd z\, \dd r\Big)^q \dd x\, \dd y \bigg)^\frac{1}{q}.
\end{align*}
Denote by $G^{F,\kappa,*}$ the dual of the semigroup $G^{F,\kappa}$
 defined in~\eqref{eq:gaussian-kernel}, which acts on measures. Using the symmetry in $x$ and $y$,
  that $\int_{|z| \leq 1} |z|^{\beta q - d - (\theta+\varepsilon) q} \dd z < \infty$ and
  \eqref{ut-time-reg} with $q=k^*$, we obtain
\begin{align*}
    \left(\int_{|x-y| \leq 1} 
        \frac{|u_t(x) - u_t(y)|^q}{|x-y|^{d + (\theta+\varepsilon) q}} \dd x\,
         \dd y  \right)^\frac{1}{q}
&\leq C t^{-\frac{1}{2}\beta} \left(\int \Big(\int g_\kappa^F(t,z,x)\, u_0(\dd z) \Big)^q \dd x\right)^{\frac{1}{q}} \\
&\quad + C \int_0^t \frac{\|u_r\|_{L^{k^*}}}{(t-r)^{\frac{1}{2}(1+\beta)}}  
    \left(\int \left(\int u_r(z) g_\kappa^F(t-r,z,x)\, \dd z\right)^q \dd x\right)^{\frac{1}{q}} \dd r\\
& \leq C t^{-\frac{1}{2}\beta} \| G^{F,\kappa,*}_{t} u_0\|_{L^q} +
 C \int_0^t (t-r)^{-\frac{1}{2}(1+\beta)} r^{-\frac{d}{2k}} 
    \| G^{F,\kappa,*}_{t-r} u_r\|_{L^q} \,\dd r .
\end{align*}
By duality, we have $\| G^{F,\kappa,*}_{t} u_0\|_{L^q} \leq \|u_0\|_{TV} 
\|G^{F,\kappa}_{t}\|_{L^{q^*} \to L^{\infty}}$ and 
similarly $\| G^{F,\kappa,*}_{t-r} u_r\|_{L^q} 
\leq \lVert u_{r} \rVert_{L^q}  \|G^{F,
\kappa}_{t-r}\|_{L^{q^*} \to L^{q^*}}$. 
Hence, choosing $q$ such that $\frac{1}{q^*}+\frac{1}{k}<\frac{2}{d}$ and 
applying
 Lemma~\ref{lem:OU-bounds}\ref{item:sgG-Lp} gives
\begin{align}\nonumber
\left(\int_{|x-y| \leq 1} 
        \frac{|u_t(x) - u_t(y)|^q}{|x-y|^{d + (\theta+\varepsilon) q}} \dd x\,
         \dd y  \right)^\frac{1}{q}
\nonumber& \leq C t^{-\frac{1}{2}\beta-\frac{d}{2 q^*}} + 
C \int_0^t (t-r)^{-\frac{1}{2}(1+\beta)} r^{-\frac{d}{2q^*}-
    \frac{d}{2k}} \dd r  \\
  \nonumber  &\leq C t^{-\frac{1}{2}\beta-\frac{d}{2 q^*}} + 
C t^{\frac{1}{2}(1-\beta)-\frac{d}{2q^*}-
    \frac{d}{2k}} \\ \label{ut-time-reg2}
    &\leq C t^{-\frac{1}{2}\beta-\frac{d}{2 q^*}}.  
\end{align}
Plugging~\eqref{ut-time-reg} and~\eqref{ut-time-reg2} into~\eqref{eq:ut-hs}, we obtain that for any $(\theta, q) \in (0,1) \times (1,\infty)$ which satisfy $\frac{1}{q^*}+\frac{1}{k}<\frac{2}{d}$ and 
 for any $\beta \in (\theta,1)$, there exists $C>0$ such that for any $t \in (0,1]$,
\begin{equation}\label{ut-time-reg3}
\|u_t\|_{H^\theta_q} \leq C
 t^{-\frac{1}{2}\beta-\frac{d}{2 q^*}}.
\end{equation}

Assume now that $k=\infty$, and fix $\theta\in(0,1)$
and $\beta \in (\theta,1)$.
We prove \eqref{eq:ut-Wtheta1} directly by estimating
the Slobodeckij seminorm.
 First, since $u_r$ is a probability density, and $K \in L^\infty$, Young's inequality gives
\(\|u_r(K\ast u_r)\|_{L^1} \leq \|K\|_{L^\infty}.\)
Repeating the computation leading to \eqref{ut-time-reg2}
with $q=1$, with denominator $|x-y|^{d+\theta}$,
and with exponent $\beta$ in the kernel
difference estimates, we obtain
\begin{align*}
    \iint_{|x-y|\leq 1}
    \frac{|u_t(x)-u_t(y)|}{|x-y|^{d+\theta}}
    \,\dd x\,\dd y
    &\leq C t^{-\frac{\beta}{2}}
    + C\|K\|_{L^\infty}
    \int_0^t
    (t-r)^{-\frac{1+\beta}{2}}
    \,\dd r \\
    &\leq C t^{-\frac{\beta}{2}}
    + C
    t^{\frac{1-\beta}{2}}.
\end{align*}
Here the spatial integral is finite because
$\int_{|h|\leq 1}|h|^{-d+(\beta-\theta)}\,\dd h<\infty$,
and the time integral is finite because
$\beta<1$.
Moreover,
\[\iint_{|x-y|>1}
    \frac{|u_t(x)-u_t(y)|}{|x-y|^{d+\theta}}
    \,\dd x\,\dd y
    \leq
    2\|u_t\|_{L^1}
    \int_{|h|>1}|h|^{-d-\theta}\,\dd h
    \leq C.\]
Summing the obtained bound on the Slobodeckij seminorm and the $L^1$ norm which is equal to $1$ and using $t\leq 1$ proves
\eqref{eq:ut-Wtheta1}.
\end{proof}

\begin{remark}
\label{rem:track-K}
    For latter use, tracking the dependency 
    in $\|K\|_{L^k}$ in the constants denoted by $C$ is important.
    In doing so for \eqref{ut-time-reg3}, we see that the constant
    $C$ is of form $C_0(1+\|K\|_{L^k})$, where 
    $C_0$ is independent of $K$.
\end{remark}

Now let us state the Lipschitz regularity of $K \ast u_t$. 
The proof is based on Proposition~\ref{prop:lip-bound1} 
together with Schauder-type estimates for $\nabla \hat{P}_t$. This result 
is a key ingredient in the proof of Theorem~\ref{thm:main}.
In particular, it is essential that the bound in \eqref{ut-lip} does not
depend on the initial condition $u_0$, so that, in the proof of Theorem~\ref{thm:main},  it can be iterated on any time interval of length $1$.

\begin{proposition}\label{prop:lip-bound} 
Suppose Assumption~\ref{assum:potentials} holds and that $u_0\in\mathcal{P}_2(\R^d)$. Then there exists $C>0$ such that for any $t \in (0,1]$, any global strong 
solution $u$ of~\eqref{mckean-vlasov} satisfies
    \begin{equation}\label{ut-lip}
    \|K\ast u_t\|_{W^1_\infty} \leq C \|K\|_{L^k} t^{-\frac12 -\frac{d}{2k}},
    \end{equation}
where $C$ does not depend on $u_0$.
\end{proposition}

\begin{proof}  
Young's convolution inequality  gives 
 \(\|  K \ast u_t\|_{W^1_\infty} \lesssim  \|K\|_{L^k}
  \| u_t\|_{W^1_{k^*}}.\) Hence it is enough 
  to bound \(\| u_t\|_{W^1_{k^*}}\).\\
\textbf{
Case $k \in (1,\infty)$: 
}Using Duhamel's formula \eqref{eq:duhamel-u}, 
\begin{align}
\label{ut-W1k}
\|u_{t}\|_{W^1_{k^*}} &\leq \|\hat{P}^*_t u_0\|_{W^1_{k^*}} 
+ \int_0^t \left\|\int_{\R^d} \nabla_x \hat{p}_{t-r}(x,\cdot) 
\cdot(u_r(x) K\ast u_r(x))\, \dd x
\right\|_{W^1_{k^*}} \dd r .
\end{align}
For the initial condition part of \eqref{ut-W1k}, we proceed as follows: using Lemma~\ref{lem:OU-bounds}\ref{item:OU-Lp},
\begin{equation*}
\|\hat{P}^*_t u_0\|_{L^{k^*}} \leq C \|\hat{P}_t\|_{L^{k} \to L^{\infty}} \|u_0\|_{TV} 
\leq C t^{-\frac{d}{2k}} \|u_0\|_{TV} = C t^{-\frac{d}{2k}}.
\end{equation*}
Then, using Lemma~\ref{lem:OU-bounds}\ref{item:heat-bounds} and
Lemma~\ref{lem:OU-bounds}\ref{item:sgG-Lp}, 
we have 
\begin{equation*}
\|\nabla \hat{P}^*_t u_0\|_{L^{k^*}} = \left\|\int\nabla_y\hat p_t(x,\cdot)\, u_0(\dd x)
\right\|_{L^{k^*}}\le Ct^{-\frac12}\|G^{F,\kappa,*}_t u_0\|_{L^{k^*}}\le Ct^{-\frac12-\frac d{2k}}.
\end{equation*}
Hence, combining both bounds, there exists $C>0$ such that for any $t \in (0,1]$,
\begin{equation}\label{hatP-bound-initial}
\|\hat{P}^*_t u_0\|_{W^1_{k^*}} \leq C t^{-\frac12-\frac d{2k}}.
\end{equation}

For the second term in the right-hand side of~\eqref{ut-W1k}, we will use 
Proposition~\ref{prop:lip-bound1} together with bounds 
on the semigroup $\hat{P}^*_t$ from Lemma~\ref{lem:OU-bounds}.
Take $f \in L^{k^*}$, and define 
$T_{t-r}f(y)\coloneqq\int_{\R^d}\nabla_x\hat p_{t-r}(x,y)\cdot f(x)\,\dd x$.
Using Lemma~\ref{lem:OU-bounds}\ref{item:heat-bounds} and
Lemma~\ref{lem:OU-bounds}\ref{item:sgG-Lp}, it comes
\begin{equation}\label{reg1Lk} 
\|T_{t-r}f\|_{L^{k^*}} =
\left\| \int_{\R^d} \nabla_x 
\hat{p}_{t-r}(x,\cdot) \cdot f(x)\, \dd x\right\|_{L^{k^*}} 
\leq C (t-r)^{-\frac12} \|G^{F,\kappa,*}_{t-r} f\|_{L^{k^*}} 
\leq C (t-r)^{-\frac12} \|f\|_{L^{k^*}}.
\end{equation}
For $f \in W^1_{k^*}$,
integrating by parts in the $x$-variable and using Lemma~\ref{lem:OU-bounds}\ref{item:sgG-Lp}, we have 
\begin{equation}\label{reg1W1k}
\|T_{t-r}f\|_{L^{k^*}} =
\left\| \int_{\R^d} \nabla_x \hat{p}_{t-r}(x,\cdot) \cdot f(x)\, \dd x\right\|_{L^{k^*}}
\leq C \|G^{F,\kappa,*}_{t-r} f\|_{W^1_{k^*}} \leq C \|f\|_{W^1_{k^*}}.
\end{equation}
For the derivative term,
using the semigroup property,
$\hat p_{t-r}(x,y)
=\int_{\R^d}\hat p_{\frac{t-r}{2}}(x,z)
\hat p_{\frac{t-r}{2}}(z,y)\,\dd z$ 
and taking now the gradient with respect to the $y$-variable gives
\[\nabla_yT_{t-r}f(y)
=\int_{\R^d}
\nabla_y\hat p_{\frac{t-r}{2}}(z,y)
\left(
\int_{\R^d}
\nabla_x\hat p_{\frac{t-r}{2}}(x,z)\cdot f(x)\,\dd x
\right)\dd z.\]
Therefore, for $f \in L^{k^*}$,
using Lemma~\ref{lem:OU-bounds}\ref{item:heat-bounds},
Lemma~\ref{lem:OU-bounds}\ref{item:sgG-Lp}, and then
\eqref{reg1Lk} with time $\frac{t-r}{2}$, we obtain
\begin{align}
\|\nabla_yT_{t-r}f\|_{L^{k^*}}
&\leq C (t-r)^{-\frac12}
\left\|G^{F,\kappa,*}_{\frac{t-r}{2}}
\left(\int_{\R^d}
\nabla_x\hat p_{\frac{t-r}{2}}(x,\cdot)
\cdot f(x)\,\dd x
\right)\right\|_{L^{k^*}}
\nonumber\\
&\leq
C(t-r)^{-\frac12} \|G^{F,\kappa,*}_{\frac{t-r}{2}}\|_{L^{k^*}\to L^{k^*}}
\left\|\int_{\R^d}\nabla_x\hat p_{\frac{t-r}{2}}(x,\cdot)
\cdot f(x)\,\dd x\right\|_{L^{k^*}}
\nonumber\\
&\leq
C(t-r)^{-1}\|f\|_{L^{k^*}}.
\label{reg2Lk}
\end{align}
Combining \eqref{reg2Lk} with \eqref{reg1Lk} yields
\[\|T_{t-r}f\|_{W^1_{k^*}}
\leq
C(t-r)^{-1}\|f\|_{L^{k^*}}.\]
On the other hand, integrating by parts in the $x$-variable gives
\[\nabla_yT_{t-r}f(y)
=-\int_{\R^d}
\nabla_y\hat p_{\frac{t-r}{2}}(z,y)
\left(\int_{\R^d}\hat p_{\frac{t-r}{2}}(x,z)\nabla_x\cdot f(x)\,\dd x
\right)\dd z.\]
Thus, for $f \in W^1_{k^*}$,
using Lemma~\ref{lem:OU-bounds}\ref{item:heat-bounds},
 Lemma~\ref{lem:OU-bounds}\ref{item:sgG-Lp} and 
 Lemma~\ref{lem:OU-bounds}\ref{item:OU-Lp}
\begin{align}
\|\nabla_yT_{t-r}f\|_{L^{k^*}}
&\leq C (t-r)^{-\frac12}
\left\|G^{F,\kappa,*}_{\frac{t-r}{2}}
\left(
\int_{\R^d}
\hat p_{\frac{t-r}{2}}(x,\cdot)\nabla_x
\cdot f(x)\,\dd x
\right)\right\|_{L^{k^*}}
\nonumber\\
&\leq
C(t-r)^{-\frac12} \|G^{F,\kappa,*}_{\frac{t-r}{2}}\|_{L^{k^*}\to L^{k^*}}
\left\| \int_{\R^d}
\hat p_{\frac{t-r}{2}}(x,\cdot)\nabla_x
\cdot f(x)\,\dd x \right\|_{L^{k^*}}
\nonumber\\
&\leq
C(t-r)^{-\frac12} \|\hat{P}^*_{\frac{t-r}{2}}\|_{L^{k^*}\to L^{k^*}} \|f\|_{W^1_{k^*}} \leq
C(t-r)^{-\frac12}\|f\|_{W^1_{k^*}}.
\label{reg2W1k}
\end{align}
Combining \eqref{reg2W1k} with \eqref{reg1W1k}, we get
\[\|T_{t-r}f\|_{W^1_{k^*}}
\leq
C(t-r)^{-\frac12}\|f\|_{W^1_{k^*}}.\]
To sum up, we have obtained the following two bounds for all $f \in W^1_{k^*}$, $g \in L^{k^*}$ and all  $0<r<t\le 1$:
\[\|T_{t-r}g\|_{W^1_{k^*}}
\leq C(t-r)^{-1}\|g \|_{L^{k^*}},
\qquad
\|T_{t-r}f\|_{W^1_{k^*}}
\leq C(t-r)^{-\frac12}\|f\|_{W^1_{k^*}}.\]
These two bounds allow us to use
 the complex interpolation theorem, see \cite[Section 2.5.2]{RunstSickel},
 and obtain, for all $\theta \in (0,1)$ and all $f \in H^\theta_{k^*}$,
\begin{align} \nonumber
    \left\| T_{t-r} f
    \right\|_{W^1_{k^*}} &\leq
    \|T_{t-r}\|_{H^\theta_{k^*} \rightarrow W^1_{k^*}} \|f\|_{H^\theta_{k^*}}
 \leq C \|T_{t-r}\|_{L^{k^*} \rightarrow W^1_{k^*}}^{1-\theta} \|T_{t-r}\|_{W^1_{k^*}\rightarrow
  W^1_{k^*}}^\theta \|f\|_{H^\theta_{k^*}} \\\label{interpolated}
   &\leq C (t-r)^{-1 + \frac{\theta}{2}} \|f\|_{H^\theta_{k^*}}.
\end{align}
Let us now fix $\theta < 1 \wedge 2(1-\frac{d}{k})$.
Hence, returning to \eqref{ut-W1k}
for which we use~\eqref{hatP-bound-initial} and~\eqref{interpolated} with $f = u_r K\ast u_r$, we obtain
\begin{align}\label{eq:mild-reg}
    \|u_{t}\|_{W^1_{k^*}} &\leq C t^{-\frac{1}{2}-\frac{d}{2k}}
    + \int_0^t (t-r)^{-1 + \frac{\theta}{2}} \|u_r K\ast u_r\|_{H^\theta_{k^*}} \dd r .
\end{align}
The Kato-Ponce inequality (see \cite[Theorem 1, eq. (3)]{GO14}) 
reads
\begin{align*}
    \|u_r K\ast u_r\|_{H^\theta_{k^*}} &\leq C(\|u_r\|_{H^\theta_{k^*}}
 \|K\ast u_r\|_{L^\infty} + \|u_r\|_{L^{k^*}} \|K\ast u_r\|_{H^\theta_\infty}).
\end{align*}
By the definition of the Bessel operator~\eqref{eq:defBessel}, $(I-\Delta)^\frac{\theta}{2}(K\ast u_{r}) = K\ast (I-\Delta)^\frac{\theta}{2} u_{r}$, hence the previous inequality combined with Hölder's inequality yields
\begin{align*}
    \|u_r K\ast u_r\|_{H^\theta_{k^*}} 
 & \leq C(\|u_r\|_{H^\theta_{k^*}} \|K\|_{L^k} \|u_r\|_{L^{k^*}} + \|u_r\|_{L^{k^*}} \|K\|_{L^k}
  \|u_r\|_{H^\theta_{k^*}}).
\end{align*}
Hence, using the above estimate, together with Proposition~\ref{prop:lip-bound1} for $q=k^*$ and $\beta \coloneqq \theta + \varepsilon < 1 \wedge 2(1-\frac{d}{k})$, and~\eqref{ut-time-reg} for $q=k^*$, we obtain
\begin{equation}
\label{eq:use-Prop3.4}
\begin{split}
    \|u_{t}\|_{W^1_{k^*}} &\leq C t^{-\frac{1}{2}-\frac{d}{2k}}
    + C \int_0^t (t-r)^{-1 + \frac{\theta}{2}} \|u_r\|_{H^\theta_{k^*}} \|u_r\|_{L^{k^*}} \dd r\\
    & \leq C t^{-\frac{1}{2}-\frac{d}{2k}}
    + C \int_0^t (t-r)^{-1 + \frac{\theta}{2}} r^{-\frac{1}{2}\beta-\frac{d}{2 k}} r^{-\frac{d}{2k}}\, \dd r\\
    & \leq C (t^{-\frac{1}{2}-\frac{d}{2k}}+t^{-\frac{d}{k}-\frac{\varepsilon}{2}}) .
\end{split}
\end{equation}
Finally, taking $0<\varepsilon <1-\frac{d}{k}$, we obtain that there exists $C>0$ such that for any $t \in (0,1]$,
\begin{equation*}
\|u_{t}\|_{W^1_{k^*}} \leq C t^{-\frac{1}{2}-\frac{d}{2k}}.
\end{equation*}

\textbf{Case $k = \infty$:} In this case, going back to \eqref{interpolated}, where instead of the complex interpolation theorem, we use the real interpolation theorem (see \cite[Section 2.5.1] {RunstSickel}), we obtain that
for any $\theta \in (0,1)$ and any $f \in W^\theta_{1}$,
\begin{align*}
    \left\| T_{t-r} f
    \right\|_{W^1_{1}}
   &\leq C (t-r)^{-1 + \frac{\theta}{2}} \|f\|_{W^\theta_{1}}.
\end{align*}
Now \eqref{eq:mild-reg} writes as 
\begin{align}\label{eq:mild-reg2}
    \|u_{t}\|_{W^1_{1}} &\leq C t^{-\frac{1}{2}}
     +C \int_0^t (t-r)^{-1 + \frac{\theta}{2}} \|u_r K\ast u_r\|_{W^\theta_{1}}\, \dd r .
\end{align}
Using $W^\theta_1=B^\theta_{1,1}$, the fractional Leibniz estimate in \cite[Lemma 2.2]{CS22} with $r=p=1$, $q=\infty$ 
and $\alpha = \theta$ with the notations therein,
combined with Young's convolution inequality, gives
\begin{align*}
    \|u_r(K\ast u_r)\|_{W^\theta_1}&
    \leq \|u_r\|_{W^\theta_1}\|K\ast u_r\|_{L^\infty}
    + \|u_r\|_{L^1}\|K\ast u_r\|_{B^\theta_{\infty,1}}\\
    &
    \leq C\|K\|_{L^\infty} \|u_r\|_{L^1} \|u_r\|_{W^\theta_1}.
\end{align*}
Fix $\varepsilon \in (0,1-\theta)$. Replacing the above inequality in \eqref{eq:mild-reg2} and using \eqref{eq:ut-Wtheta1} from Lemma~\ref{prop:lip-bound1} (with $\beta = \theta + \varepsilon$),
we obtain
\begin{align*}
    \|u_{t}\|_{W^1_{1}} &\leq C t^{-\frac{1}{2}}
    + C\|K\|_{L^\infty} \int_0^t (t-r)^{-1 + \frac{\theta}{2}} r^{-\frac{\theta+\varepsilon}{2}}\, \dd r\\
    &\leq C t^{-\frac{1}{2}} + C\|K\|_{L^\infty} t^{-\frac{\varepsilon}{2}} \leq C t^{-\frac{1}{2}},
\end{align*}
where we used $\varepsilon <1$ in the last inequality.
\end{proof}

\begin{remark}
\label{rem:track-K-2}
Again keeping track of the dependence of $C$ in $ \lVert K\rVert_{L^k}$, we see from Equation~\eqref{eq:use-Prop3.4}, where Proposition~\ref{prop:lip-bound1} is used, and from Remark~\ref{rem:track-K}, that in Proposition~\ref{prop:lip-bound}, $C$ is now of the form $C_0(1+\|K\|_{L^k}^2)$, for some $C_0$ which is independent of $K$.
\end{remark}

\section{Bound on the error between microscopic and macroscopic
 mild formulations}\label{sec:micro-macro}
Define the  empirical measure by 
\begin{equation*}
\mu_t^N \coloneqq \frac{1}{N} \sum_{i=1}^N \delta_{X_t^{i,N}},
\end{equation*}
and the mollified empirical measure $u_t^N$ by
\begin{equation}
\label{eq:initial-empirical-measure}
u_t^N(\cdot) \coloneqq V^N \ast \mu_t^N(\cdot) =
\frac{1}{N} \sum_{i=1}^N V^N(\cdot - X_t^{i,N}).
\end{equation}
The following proposition 
gives an a.s. error estimate for the difference between the
 mollified empirical measure $u_t^N$ and the 
 solution $u_t$ of the nonlinear Fokker-Planck equation 
 \eqref{nlfp-intro}. This a.s. bound will be used to derive an $L^2(\PP)$ bound
 for $\|u_t^N - u_t\|_{k^*}$ in the proof of Theorem \ref{thm:main}.
\begin{proposition}
\label{prop:mild-formulation}
Under Assumption~\ref{assum:potentials}, there exists $C > 0$ such that for any $t \in (0,1]$,
\begin{align*}
\|u_t^N - u_t\|_{k^*}
    &\leq C (t^{-\frac{1}{2}- \frac{d}{2k}} +t^{-\frac{d}{k}})
    \left( \mathbb{W}_1(u_0^N,u_0) + N^{-\alpha} \right)
    + \|M_t^N\|_{k^*} + C \int_0^t (t-s)^{-\frac{1}{2}-\frac{d}{2k}} \|M_s^N\|_{k^*} \dd s\quad a.s.,
\end{align*}
where 
$M_t^N(x) \coloneqq  
    \frac{\sqrt{2}}{N}\sum_{i=1}^N \int_0^t 
    \big( \hat{P}^*_{t-s} \,\nabla V^N(\cdot- X_s^{i,N})(x) \,\big)\cdot \dd W_s^i,$\,  for a.e. $x \in \R^d$.
\end{proposition}
\begin{proof}
Let $f$ in $L^{k}(\R^d)$. Since $u_t = P^{u,*}_{0,t} u_0$ and
using \eqref{eq:duhamel}, we have the following mild formulation for $u_t$: 
\begin{equation}\label{eq:ut-mild}
    u_t(f) = \hat{P}_t^*u_0(f) + \int_0^t 
    \hat{P}_{t-s}^* \nabla \left(u_s (K * u_s)\right)(f)\, \dd s.
\end{equation}
In order to derive a similar formulation for $u_t^N$, we apply Itô's formula to
\begin{equation*}
s \mapsto g(s,X_s^N) \coloneqq \hat{P}_{t-s}^*u_s^N(f) 
=\frac1N\sum_{i=1}^N
\int_{\R^d}
\widehat P_{t-s}f(x)
V^N(x-X_s^{i,N})\,\dd x.
\end{equation*}
Although $f$ is only assumed to belong to $L^k(\R^d)$,
the convolution with $V^N$ provides enough regularity
to apply It\^o's formula to $g(s,X_s^N)$. Recall that 
$\hat{P}$ satisfies
\(-\partial_s\hat{P}_{t-s}f=(\Delta-F\cdot\nabla)\hat{P}_{t-s}f.\)
Hence, in It\^o's formula, the term
$-\Delta\hat P_{t-s}f$ arising from the time derivative
cancels with the Laplacian term in the generator, while the term   
$F\cdot\nabla\hat P_{t-s}f$ combines with the drift term
$-F(X_s^{i,N})\cdot\nabla\hat P_{t-s}f$, leading to
$F(X_s^{i,N})-F(x)$. Rewriting this in terms of
$\mu_s^N$ and integrating by parts gives a contribution
involving $F(\cdot)-F(x)$. Namely, we obtain the following
mild formulation for $u_t^N$:
\begin{align}\nonumber
u_t^N(f) &= \hat{P}_t^*u_0^N(f)
+ \int_{\R^d}\int_0^t  f(x) \hat{P}_{t-s}^*\nabla \cdot
 \left\langle \mu_s^N, V^N(x-\cdot)\, K\ast 
u_s^N(\cdot) \right\rangle \,\dd s\, \dd x  \\\nonumber
&\quad
+ \int_{\R^d}\int_0^t f(x)  \hat{P}_{t-s}^*\nabla \cdot
\langle \mu_s^N, V^N(x-\cdot)\,\left(F(\cdot)-
F(x) \right) \rangle \,\dd s \, \dd x \\\label{eq:utn-mild}
&\quad
- \frac{\sqrt{2}}{N}\sum_{i=1}^N \int_0^t 
\left(\int_{\mathbb{R}^d} \hat{P}_{t-s} f(x)\, \nabla
V^N(x - X_s^{i,N}) \, \dd x \right)\cdot \dd W_s^i.
\end{align}
Subtracting \eqref{eq:utn-mild} from \eqref{eq:ut-mild} gives 
\begin{equation*}
\|u_t^N - u_t\|_{k^*} 
 \leq \|\hat{P}_t^*(u_0^N - u_0)\|_{k^*} + \|E_t^{1,N}\|_{k^*} + 
 \|E_t^{2,N}\|_{k^*} + \|M_t^N\|_{k^*},
\end{equation*}
where, for a.e. $x \in \mathbb{R}^d$,
\begin{equation*}
    E_t^{1,N}(x) \coloneqq 
    \int_0^t  \hat{P}_{t-s}^* \nabla \cdot
    \langle \mu_s^N, V^N(x-\cdot)\,\left(F(\cdot)-F(x) 
    \right) \rangle \,\dd s
\end{equation*}
and
\begin{equation*}
E_t^{2,N}(x) \coloneqq 
    \int_0^t \hat{P}_{t-s}^* 
    \nabla \cdot \left(
    \left\langle \mu_s^N, V^N(x-\cdot)\, K\ast u_s^N(\cdot) \right\rangle
   - u_s (K * u_s)(x) \right)\dd s.
\end{equation*}

\textbf{Step 1: Bound for $\|\hat P_t^*(u_0^N - u_0)\|_{k^*}$}

Using Lemma~\ref{lem:OU-bounds}\ref{item:Phat-Wasserstein}, there exists a constant $C >0$ such that for any $t \in (0,1]$,
\begin{equation*}
\|\hat{P}_t^*(u_0^N - u_0)\|_{k^*} \leq C t^{-\frac{1}{2}- \frac{d}{2k}}
\mathbb{W}_1(u_0^N, u_0).
\end{equation*}

\textbf{Step 2: Bound for $\|E_t^{1,N}\|_{k^*}$}

We first record the following immediate consequence of duality that will be used in Step~2 and Step~3: for any
$h\in L^1(\mathbb{R}^d;\mathbb{R}^d)$,
\begin{align}
\label{eq:beginning-Step2} 
\bigl\|\hat{P}_{t-s}^*
         \nabla\cdot h\bigr\|_{k^*} &= \sup_{\|f\|_k \leq 1} \left| 
            \int_{\mathbb{R}^d} h(x)\cdot \nabla \hat{P}_{t-s}f(x)\,\dd x \right| 
            \leq \|\nabla \hat{P}_{t-s}\|_{L^k \to L^\infty} \|h\|_{L^1}.
 \end{align}
To estimate $\|E_t^{1,N}\|_{k^*}$, perform the chain of inequalities below, using first~\eqref{eq:beginning-Step2}, then Lemma~\ref{lem:OU-bounds}\ref{item:OU-Lp} in the second inequality, then the Lipschitz continuity of $F$ in the third inequality,
 and finally that $\int V^N(z) |z| \dd z \lesssim N^{-\alpha}$ in the last inequality, to deduce that there exists $C >0$ such that for any $t \in (0,1]$,
\begin{align*}
\|E_t^{1,N}\|_{k^*} &\leq 
    \int_0^t \| \nabla \hat{P}_{t-s}\|_{L^k \to L^\infty} \int_{\mathbb{R}^d} \int_{\mathbb{R}^d}
      \mu_s^N(\dd y) V^N(x-y)\,
     \left|F(x)-F(y) \right| \dd x \, \dd s
     \\
     &\leq C \int_0^t (t-s)^{-\frac{1}{2}-\frac{d}{2k}}
     \int_{\mathbb{R}^d} \int_{\mathbb{R}^d}
      \mu_s^N(\dd y) V^N(x-y)\,
     \left|F(x)-F(y) \right| \dd x \, \dd s\\
     & \leq C \int_0^t (t-s)^{-\frac{1}{2}-\frac{d}{2k}} \int_{\mathbb{R}^d}  \mu_s^N(\dd y) 
        \int_{\mathbb{R}^d} V^N(x-y)\, |x-y| \, \dd x \, \dd s
      \\
        &\leq C N^{-\alpha }.
\end{align*}

\textbf{Step 3: Bound for $\|E_t^{2,N}\|_{k^*}$}

We decompose, for a.e. $x \in \mathbb{R}^d$, $E_t^{2,N}(x)$ in three parts as follows:
\begin{align*}
E_t^{2,N}(x) &=
    \int_0^t  \hat{P}_{t-s}^* \nabla \cdot 
    \left\langle \mu_s^N, V^N(x-\cdot)\, (K\ast u_s^N(\cdot) - 
    K \ast u_s(\cdot))\right\rangle \dd s
  \\
    &\quad +\int_0^t 
    \hat{P}_{t-s}^*  \nabla\cdot
    \left\langle \mu_s^N, V^N(x-\cdot)\, (K\ast u_s(\cdot) - 
    K \ast u_s(x)) \right\rangle \dd s
  \\
    & \quad + 
    \int_0^t \hat{P}_{t-s}^* \nabla \cdot \left( u_s^N(x) - u_s(x) \right)    
    (K \ast u_s(x)) \, \dd s 
    \eqqcolon F_t^{1,N}(x) + F_t^{2,N}(x) +
     F_t^{3,N}(x).
\end{align*}
For $F^{1,N}$, using~\eqref{eq:beginning-Step2}, and then using Lemma~\ref{lem:OU-bounds}\ref{item:OU-Lp}, we deduce that there exists $C >0$ such that for any $t \in (0,1]$, 
\begin{align*}
    \| F_t^{1,N}\|_{k^*} &\leq
    \int_0^t \|\nabla \hat{P}_{t-s} \|_{L^k \to L^\infty} 
    \int_{\mathbb{R}^d}  \int
     \mu_s^N(\dd y) V^N(x-y)\, \left| K\ast u_s^N(y) - 
    K \ast u_s(y) \right| \dd x \, \dd s\\
    &\leq C \|K\|_{L^k} \int_0^t (t-s)^{-\frac{1}{2}-
    \frac{d}{2k}} \|u_s^N - u_s\|_{k^*}\, \dd s. 
\end{align*}
For $F^{2,N}$, using Lemma~\ref{lem:OU-bounds}\ref{item:OU-Lp}, 
 Proposition~\ref{prop:lip-bound} and
 that $\int V^N(z) |z| \dd z \lesssim N^{-\alpha}$, we deduce 
 that there exists $C >0$ 
such that for any $t \in (0,1]$,
\begin{align*}
   \| F_t^{2,N}\|_{k^*}
    &\leq 
    \int_0^t \| \nabla \hat{P}_{t-s}\|_{L^k \to L^\infty} \int_{\mathbb{R}^d} \int_{\mathbb{R}^d}
      \mu_s^N(\dd y) V^N(x-y)
     \left|K \ast u_s (x)-K \ast u_s(y) \right| \dd x \, \dd s\\
     &\leq C \int_0^t
      (t-s)^{-\frac{1}{2}-\frac{d}{2k}} \, s^{
       -\frac{1}{2}-\frac{d}{2k}}
     \int_{\mathbb{R}^d} \int_{\mathbb{R}^d}
      \mu_s^N(\dd y) V^N(x-y)
     \left|x-y \right| \dd x \, \dd s\\
     &\leq
     C t^{-\frac{1}{2}-\frac{d}{2k}} N^{-\alpha}.
\end{align*}
Finally, for $F^{3,N}$, using Lemma~\ref{lem:OU-bounds}\ref{item:OU-Lp}, we derive a bound similar to the first term, 
\begin{align*}
    \| F_t^{3,N}\|_{k^*}
    &\leq \int_0^t \| \nabla \hat{P}_{t-s} \|_{L^k \to L^\infty} 
    \|K \ast u_s\|_{L^k} \|u_s^N - u_s\|_{k^*}\, \dd s\\
    & \leq C  \int_0^t (t-s)^{-\frac{1}{2}-\frac{d}{2k}} \|K\|_{L^k} 
    \|u_s\|_{L^1} \|u_s^N - u_s\|_{k^*}\, \dd s\\
    & \leq C \int_0^t (t-s)^{-\frac{1}{2}-\frac{d}{2k}} 
    \|u_s^N - u_s\|_{k^*}\, \dd s.
\end{align*}
Combining the previous bounds, we obtain the following estimate:
\begin{equation*}
\|E_t^{2,N}\|_{k^*} \leq C \int_0^t (t-s)^{-\frac{1}{2}-\frac{d}{2k}} 
\|u_s^N - u_s\|_{k^*}\, \dd s+ C t^{-\frac{1}{2}-\frac{d}{2k}} N^{-\alpha}.
\end{equation*}

\textbf{Step 4: Gr\"onwall's Lemma.}

Combining the above estimates, we obtain that for any $t \in (0,1]$, almost surely,
\begin{align*}
\|u_t^N - u_t\|_{k^*} &\leq C t^{-\frac{1}{2}- \frac{d}{2k}}
\mathbb{W}_1(u_0^N,u_0) + C t^{-\frac{1}{2}-\frac{d}{2k}} N^{-\alpha} + \|M_t^N\|_{k^*} + C \int_0^t 
(t-s)^{-\frac{1}{2}-\frac{d}{2k}} 
\|u_s^N - u_s\|_{k^*}\, \dd s.
\end{align*}
Using a Gr\"onwall's inequality of Volterra type (see \cite[Lemma 7.1.1]{Henry}), for any $t \in (0,1]$, almost surely,
\begin{align}
    \nonumber
\|u_t^N - u_t\|_{k^*}&\leq C t^{-\frac{1}{2}- \frac{d}{2k}}
    \mathbb{W}_1(u_0^N,u_0) + C t^{-\frac{1}{2}-\frac{d}{2k}} N^{-\alpha} + \|M_t^N\|_{k^*}  \\ \nonumber
    &\quad + C \int_0^t (t-s)^{-\frac{1}{2}-\frac{d}{2k}}
    \Big(s^{-\frac{1}{2}-\frac{d}{2k}}
    \mathbb{W}_1(u_0^N,u_0) + s^{-\frac{1}{2}-\frac{d}{2k}}N^{-\alpha} +
     \|M_s^N\|_{k^*}\Big)\, \dd s\\ \label{Gronwall}
    &\leq C (t^{-\frac{1}{2}- \frac{d}{2k}} +
    t^{-\frac{d}{k}}) \left( \mathbb{W}_1(u_0^N,u_0) + 
   N^{-\alpha}\right) + \|M_t^N\|_{k^*} + C \int_0^t (t-
    s)^{-\frac{1}{2}-\frac{d}{2k}} \|M_s^N\|_{k^*}\, \dd s.         
\end{align}
\end{proof}

\begin{remark}
\label{rem:track-K-2}
    As mentioned in Remark~\ref{rem:track-K}, the dependence 
    on $\|K\|_{L^k}$ 
    of the constant $C$ is important.
    In \eqref{Gronwall}, 
    the Grönwall argument
    leads to $C$ of the form 
$\exp(C_0 \|K\|_{L^k})$,
    where $C_0$ is independent of $\|K\|_{L^k}$.
\end{remark}

\section{Control of the moments of particles and of the stochastic integral}
\label{sec:int-sto}
We begin with a lemma which is a key technical result in the proof 
of the uniform-in-time bounds of the moments of the particle system, given
in Lemma~\ref{lem:uniform_bounds} and is a direct consequence of the 
 Burkholder-Davis-Gundy inequality for stochastic 
 integrals with values in Banach spaces of martingale type $2$.
Note that, in Lemma~\ref{lem:uniform_convolution}, we only need to consider the case $k\in(d,+\infty)$, since 
the case $k = \infty$ is easier and is treated separately in the 
proof of Lemma~\ref{lem:uniform_bounds}.

\begin{lemma}
\label{lem:uniform_convolution} 
Define, for any $s \in (0,1]$,
\begin{equation*}
S_s^N \coloneqq  N^{-1} \int_0^s \sum_{i=1}^N  K \ast \hat{P}^*_{s-r} 
\nabla V^N(\cdot - X_r^{i,N}) \cdot \dd W_r^i.
\end{equation*} 
Suppose that $k\in(d,+\infty)$ and that Assumption~\ref{assum:potentials} hold and that Assumption~\ref{assum:alpha} holds. 
Then for any $q\ge2$, there exists a constant $C>0$ such 
that
\begin{equation*}
\sup_{s\in (0,1]} \sup_{N \in \N}
 \left\|\|S_s^N\|_{L^\infty}\right\|_{L^q(\Omega)} \leq C,
\end{equation*}
where $C$ does not depend on the initial condition of the particle system.
\end{lemma}
\begin{proof}
\textbf{Case $k \in [2,\infty)$:}
Using the embedding $H^\beta_k \hookrightarrow L^{\infty}$ 
for $\beta \in (\frac{d}{k},1)$, see~\cite[Theorem~1(a) p.32]{RunstSickel}, it comes
\begin{equation*}
\left\|\|S_s^N\|_{L^\infty}\right\|_{L^q(\Omega)} \leq 
C \left\|\|S_s^N\|_{H^{\beta}_k}\right\|_{L^q(\Omega)}.
\end{equation*}
Since $S_s^N$ is a stochastic convolution integral and as such not a martingale, we freeze time $s$ in the semigroup  and consider the $H^\beta_k$-valued
martingale
\[S^{s,N}_u \coloneqq N^{-1} \sum_{i=1}^N \int_0^u K \ast 
\hat{P}^*_{s-r} \nabla V^N(\cdot - X_r^{i,N}) \cdot \dd W_r^i,\]
for $u \in [0,s]$. Hence, we have 
\begin{align*}
    \left\|\|S_s^N\|_{H^{\beta}_k}\right\|_{L^q(\Omega)} 
&= \left\|\|S^{s,N}_s\|_{H^{\beta}_k}\right\|_{L^q(\Omega)}
\leq \sup_{u \in [0,s]} \left\|\|S^{s,N}_u\|_{H^{\beta}_k}\right\|_{L^q(\Omega)} \\
&=
C N^{-1} \sup_{u \in [0,s]} \bigg\| \Big\| \sum_{i=1}^N 
\int_0^u K \ast \hat{P}^*_{s-r} \nabla V^N(\cdot - X_r^{i,N}) \cdot
\dd W_r^i \Big\|_{H^{\beta}_k}\bigg\|_{L^q(\Omega)} .
\end{align*}
Bessel potential spaces $H_k^\beta(\mathbb{R}^d)$, for $\beta\in\mathbb{R}$ and
 $2<k<\infty$ have martingale type $2$, see for e.g. \cite[Definition 12 and discussion
 afterwards]{GG22}. Therefore, we can apply the 
 BDG inequality in Banach spaces of martingale type $2$, see e.g. \cite[Theorem 4.7]{NVW15}, and obtain
\begin{align}
\label{eq:boundS-kLarge}
\left\|\|S_s^N\|_{H^{\beta}_k}\right\|_{L^q(\Omega)} 
&\leq C N^{-1} \sup_{u \in [0,s]} \bigg\|  \Big(\sum_{i=1}^N 
\int_0^u \Big\|K \ast \hat{P}^*_{s-r} \nabla V^N(\cdot - X_r^{i,N})\Big\|_{H^{\beta}_k}^2 
\dd r\Big)^{\frac{1}{2}}
 \bigg\|_{L^q(\Omega)} \nonumber\\
 &\leq C N^{-1}
\left(
\sum_{i=1}^N
\int_0^s
\left\|
\left\|
K\ast\hat P^*_{s-r}
\nabla V^N(\cdot-X_r^{i,N})
\right\|_{H_k^\beta}
\right\|_{L^q(\Omega)}^2
\,\dd r
\right)^{1/2},
\end{align}
where we used Minkowski's inequality in the last step. The final step of this proof will be carried out after we have reached an inequality similar to~\eqref{eq:boundS-kLarge} in the case $k<2$.

\textbf{Case $k \in (1,2)$:} In this case, we have $d=1$ 
and 
$H^{\beta}_k \hookrightarrow H^{\beta_0}_2 \hookrightarrow L^{\infty}$ for $\beta_0>1/2$ and $\beta-1/k \geq \beta_0 - 1/2$, see \cite[Theorem p.31, (2)]{RunstSickel} and \cite[Theorem~1(a) p.32]{RunstSickel}.
Therefore, 
we use BDG's inequality in $H^{\beta_0}_2$ and 
then the embedding to bound $H^{\beta_0}_2$-norm by $H^{\beta}_k$-norm, 

\begin{align}
\label{eq:boundS-kSmall}
\left\|\|S_s^N\|_{L^\infty}\right\|_{L^q(\Omega)} 
&\leq \left\|\|S_s^N\|_{H^{\beta_0}_2}\right\|_{L^q(\Omega)} \nonumber \\
&\leq N^{-1} \sup_{u \in [0,s]} \bigg\| \Big\| \Big(\sum_{i=1}^N 
\int_0^u |K \ast \hat{P}^*_{s-r} \nabla V^N(\cdot - X_r^{i,N})|^2 
\dd r \Big)^{\frac{1}{2}}
 \Big\|_{H^{\beta_0}_2}\bigg\|_{L^q(\Omega)}\nonumber \\
& \leq N^{-1}
\bigg(
\sum_{i=1}^N
\int_0^s
\Big\|
\big\|
K\ast\hat P^*_{s-r}
\nabla V^N(\cdot-X_r^{i,N})
\big\|_{H^{\beta_0}_2}
\Big\|_{L^q(\Omega)}^2
\,\dd r
\bigg)^{1/2}\nonumber \\
&\leq N^{-1}
\bigg(
\sum_{i=1}^N
\int_0^s
\Big\|
\big\|
K\ast\hat P^*_{s-r}
\nabla V^N(\cdot-X_r^{i,N})
\big\|_{H_k^\beta}
\Big\|_{L^q(\Omega)}^2
\,\dd r
\bigg)^{1/2}.
\end{align}

Hence in both case $k>2$ and $k<2$, we have obtained respectively from~\eqref{eq:boundS-kLarge} an~\eqref{eq:boundS-kSmall} that for all $k>d$, there exists $\beta \in (0,1)$ such that: 
\begin{equation}\label{intermediate_bound}
\left\|\|S_s^N\|_{L^\infty}\right\|_{L^q(\Omega)} \leq N^{-1}
\bigg(
\sum_{i=1}^N
\int_0^s
\Big\|
\big\|
K\ast\hat P^*_{s-r}
\nabla V^N(\cdot-X_r^{i,N})
\big\|_{H_k^\beta}
\Big\|_{L^q(\Omega)}^2
\,\dd r
\bigg)^{1/2}.
\end{equation}
Hence, in \eqref{intermediate_bound}, we need to derive integrable in time bounds
for $\big\|
\|K\ast\hat P^*_{s-r} \nabla V^N(\cdot-X_r^{i,N})\|_{H_k^\beta}\big\|_{L^q(\Omega)}^2$. 
Using the interpolation inequality between $L^k$ and $H^1_k = W^1_k$, see e.g.
 \cite[Proposition 6.6.4]{Haase06}, with $A = I- \Delta$, we have
\begin{equation*}
\big\| K \ast \hat{P}^*_{s-r}\nabla V^N(\cdot-X_r^{i,N}) \big\|_{H^{\beta}_k}
 \leq \big\| K \ast \hat{P}^*_{s-r}\nabla
V^N(\cdot-X_r^{i,N})\big\|_{L^k}^{1-\beta} \big\| 
K \ast \hat{P}^*_{s-r}\nabla V^N(\cdot-X_r^{i,N})\big\|_{W^{1}_k}^{\beta}. 
\end{equation*}
Hence, Young's convolution inequality gives
\begin{equation}\label{eq:interpolation}
\big\| K \ast \hat{P}^*_{s-r}\nabla V^N(\cdot-X_r^{i,N})\big\|_{H^{\beta}_k} 
\leq \|K\|_{L^k} \big\| \hat{P}^*_{s-r}\nabla V^N(\cdot-X_r^{i,N})\big\|_{L^1}^{1-\beta}
 \big\| \hat{P}^*_{s-r}\nabla V^N(\cdot-X_r^{i,N})\big\|_{W^{1}_1}^{\beta}.
\end{equation}
Using Lemma~\ref{lem:OU-bounds}\ref{item:heat-bounds} for the second inequality and $\int_{\R^d}g_\kappa^F(s-r, x, y)\, \dd y = 1$ together with
$\|\nabla V^N\|_{L^1} \leq C N^{\alpha}$ for the last inequality,
 there exist $C, \kappa > 0$ such that 
for all $1>s > r>0$,
\begin{align}\nonumber
\big\| \nabla(\hat{P}^*_{s-r}\nabla V^N(\cdot-X_r^{i,N})) \big\|_{L^1} &= \int_{\R^d}
 \left|\int_{\R^d} \nabla_y \hat{p}_{s-r}(x, y) \nabla V^N(x-X_r^{i,N}) 
 \dd x\right| \dd y\\ \nonumber
& \leq C (s-r)^{-\frac{1}{2}} \int_{\R^d} \int_{\R^d}
 g^F_\kappa(s-r, x, y) |\nabla V^N(x-X_r^{i,N})|  \dd x\, \dd y\\ \label{one}
& \leq C (s-r)^{-\frac{1}{2}} N^{\alpha}.
\end{align}
Similarly, using Lemma \ref{lem:OU-bounds}\ref{item:OU-Lp}, there 
exists a constant $C>0$ 
such that for all $1>s > r>0$,
\begin{equation}\label{two}
   \big\| \hat{P}^*_{s-r} 
        \nabla V^N(\cdot-X_r^{i,N})\big\|_{L^{1}} 
        \leq \|P_{s-r}\|_{L^{\infty} \to L^{\infty}} 
        \| \nabla V^N(\cdot-X_r^{i,N})\|_{L^1}
        \leq C N^{\alpha}.
\end{equation}
Hence, by 
plugging \eqref{one} and \eqref{two} in \eqref{eq:interpolation} 
and taking the $L^q(\Omega)$ norm, we have
\begin{equation*}
\Big\|
\big\|K\ast\hat P^*_{s-r}
\nabla V^N(\cdot-X_r^{i,N})
\big\|_{H_k^\beta}
\Big\|_{L^q(\Omega)}^2 \leq C (s-r)^{-\beta} N^{2 \alpha},
\end{equation*}
which is integrable in time since $\beta < 1$. Hence, in the end
$\big\|\|S_s^N\|_{L^\infty}\big\|_{L^q(\Omega)} \leq 
C N^{-\frac{1}{2}+ \alpha}$, which is uniformly bounded in $N$, see Assumption \ref{assum:alpha}.
\end{proof}

The following lemma gives uniform bounds in $N$ for 
the moments of the particles. This bound will be essential
in the proof of Proposition~\ref{lemma:stochastic-integral}.

\begin{lemma}
\label{lem:uniform_bounds} 
Under Assumptions~\ref{assum:potentials} and~\ref{assum:alpha}, let $q \in [2, \infty)$, and assume that  \( Q = \sup_{N \in \mathbb{N}} \max_{1\le i\le N}
 \E[|X_0^{i,N}|^q] \) is finite. Then the moments of the particles are finite, uniformly in time and $N$, that is
\[ f(Q) \coloneqq \sup_{t >0} \sup_{N \in \mathbb{N}} \max_{1\le i\le N}
\mathbb{E} \big[  |X_t^{i,N}|^q \big] <\infty.
\]  
\end{lemma}
\begin{proof}
Using Itô's formula, the inequality $(F(x) - F(0)) \cdot x \geq \lambda |x|^2$ and Hölder's inequality, we have for all $q \in [2, \infty)$, 
\begin{align*}
\frac{\dd}{\dd t} \mathbb{E}\big[|X_t^{i,N}|^q\big] =&
-q\mathbb{E}\left[|X_t^{i,N}|^{q-2}X_t^{i,N}\cdot F(X_t^{i,N})\right] -
q \mathbb{E}\left[|X_t^{i,N}|^{q-2} X_t^{i,N} \cdot K \ast 
u_t^N(X_t^{i,N})\right]\\
&
+ q(d+q-2) \mathbb{E}\big[|X_t^{i,N}|^{q-2}\big]\\
 \leq & -\lambda q \mathbb{E}\big[|X_t^{i,N}|^q\big]  
 + q |F(0)| \E\big[|X_t^{i,N}|^{q-1}\big] +
 q  \mathbb{E}\big[|X_t^{i,N}|^{q-1} \|K \ast u_t^N\|_{L^\infty}\big]\\
&+ q(d+q-2) \mathbb{E}\big[|X_t^{i,N}|^{q}\big]^{\frac{q-2}{q}}\\
 \leq & -\lambda q \mathbb{E}\big[|X_t^{i,N}|^q\big] + 
 q |F(0)| \E\big[|X_t^{i,N}|^{q}\big]^{\frac{q-1}{q}} +
q  \mathbb{E}\big[|X_t^{i,N}|^{q}\big]^{\frac{q-1}{q}} 
\E\big[ \|K \ast u_t^N\|_{L^\infty}^q\big]^\frac{1}{q}\\
&+ q(d+q-2) \mathbb{E}\big[|X_t^{i,N}|^{q}\big]^{\frac{q-2}{q}} .
\end{align*}
Now consider
\begin{equation*}
g(t) \coloneqq \big(1+\E\big[|X^{i,N}_t|^{q}\big]\big)^{\frac{1}{q}}.
\end{equation*}
Then, using the previous inequality and since $\E\big[|X^{i,N}_t|^{q}\big]^{\frac{q-1}{q}} \leq g(t)^{q-1}$ and $\E\big[|X^{i,N}_t|^{q}\big]^{\frac{q-2}{q}} \leq g(t)^{q-2}$, we have
\begin{align*}
\frac{\dd}{\dd t}g(t) &= \frac{1}{q}(1+\E\big[|X^{i,
N}_t|^{q}\big])^{\frac{1}{q}-1}\frac{\dd}{\dd t}
\E\big[|X^{i,N}_t|^{q}\big] = \frac{1}{q}g(t)^{1-q}\frac{\dd}{\dd t}
\E\big[|X^{i,N}_t|^{q}\big]\\
& \leq g(t)^{1-q}\Big(-\lambda (g(t)^{q}-1) +  |F(0)| g(t)^{q-1} +
\E\big[ \|K \ast u_t^N\|_{L^\infty}^q\big]^\frac{1}{q} g(t)^{q-1}
+ (d+q-2) g(t)^{q-2}\Big).
\end{align*}
Developping the parenthesis and since 
$g(t)^{1-q}\leq 1$ and $g(t)^{-1} \leq 1$, we have
\begin{align*}
\frac{\dd}{\dd t}g(t) &\leq -\lambda g(t) + \lambda g(t)^{1-q}+ |F(0)| + \E\big[ \|K \ast u_t^N\|_{L^\infty}^q\big]^\frac{1}{q}
+ (d+q-2) g(t)^{-1}  \\
&\leq -\lambda g(t) + \E\big[ \|K \ast u_t^N\|_{L^\infty}^q\big]^\frac{1}{q} + (d+q-2) + \lambda+ |F(0)| \\
&\leq -\lambda g(t) + \E\big[\|K \ast u_t^N\|_{L^\infty}^q\big]^\frac{1}{q} + C.
\end{align*}
Using Grönwall's differential inequality, there exists $C >0$ such that
for all $t>0$,
\begin{equation}
\label{eq:moment_bound_g}
g(t) \leq e^{-\lambda t} g(0) + C \int_0^t e^{-\lambda(t-s)}\left(1 + \E[ \|K \ast u_s^N\|_{L^\infty}^q]^\frac{1}{q}\right) \dd s.
\end{equation}  
Now let us bound $\E[ \|K \ast u_s^N\|_{L^\infty}^q]^\frac{1}{q}$. Using
the mild equation~\eqref{eq:utn-mild} for $u_s^N$,
Lemma~\ref{lem:OU-bounds}\ref{item:OU-Lp} and similar computations as in the proof of Proposition~\ref{prop:mild-formulation}, 
there exists $C >0$ such that for all $s \in (0,1]$,
\begin{align*}
 \|K \ast u_s^N\|_{L^\infty} &
 \leq \| K \ast \hat{P}^*_s u_0^N\|_{L^\infty}
  + \int_0^s \left\|K \ast \hat{P}^*_{s-r} \nabla \cdot
\left(\int_{\R^d}
V^N(\cdot-y)\bigl(F(\cdot)-F(y)\bigr)\,
\mu_r^N(\dd y)\right)
\right\|_{L^\infty}\dd r\\
& \quad + \int_0^s
\left\|K \ast \hat{P}^*_{s-r} \nabla \cdot
\left(\int_{\R^d}
V^N(\cdot-y)\,(K\ast u_r^N)(y)\,
\mu_r^N(\dd y)\right)
\right\|_{L^\infty} \dd r\\
& \quad + N^{-1} \bigg\|
         \int_0^s \sum_{i=1}^N  K \ast \hat{P}^*_{s-r}
          \nabla V^N(\cdot - X_r^{i,N}) \cdot 
          \dd W_r^i \bigg\|_{L^\infty}
\\
  &\leq \| K \ast \hat{P}^*_s u_0^N\|_{L^\infty} 
+ C \int_0^s (s-r)^{-\frac{1}{2}-\frac{d}{2k}} N^{-\alpha} \dd r+
C  \int_0^s (s-r)^{-\frac{1}{2}-\frac{d}{2k}} \| K \ast 
u_r^N\|_{L^\infty} \dd r\\
& \quad + C N^{-1} \bigg\| \int_0^s \sum_{i=1}^N  K \ast \hat{P}^*_{s-r} 
\nabla V^N(\cdot - X_r^{i,N}) \cdot \dd W_r^i\bigg\|_{L^\infty}\\
& \leq \|K\|_{L^k} \|\hat{P}_s\|_{L^k \to L^\infty}
 \|u_0^N\|_{TV}
+ C s^{\frac{1}{2}-\frac{d}{2k}} N^{-\alpha} 
+ C \int_0^s (s-r)^{-\frac{1}{2}-\frac{d}{2k}} \| K 
\ast u_r^N\|_{L^\infty} \dd r\\
&\quad + C  N^{-1} \bigg\| \int_0^s \sum_{i=1}^N  K \ast \hat{P}^*_{s-r} 
\nabla V^N(\cdot - X_r^{i,N}) \cdot \dd W_r^i\bigg\|_{L^\infty}\\
& \leq C s^{-\frac{d}{2k}} + C s^{\frac{1}{2}-\frac{d}{2k}}
 N^{-\alpha}
+ C \| S_s^N\|_{L^\infty} + C \int_0^s (s-r)^{-\frac{1}{2}-\frac{d}{2k}} \| K 
\ast u_r^N\|_{L^\infty} \dd r, 
\end{align*}
where $S^N_{s}$ is defined in Lemma~\ref{lem:uniform_convolution}. 
Using Gr\"onwall-Volterra's 
inequality, see \cite[Lemma 7.1.1]{Henry}, 
we have 
\begin{align*}
\|K \ast u_s^N\|_{L^\infty} 
&\leq C s^{-\frac{d}{2k}} + C s^{\frac{1}{2}-\frac{d}{2k}} N^{-\alpha}
+ C  \|S_s^N\|_{L^\infty}\\
&\quad  +C \int_0^s (s-r)^{-\frac{1}{2}-\frac{d}{2k}} 
\left(r^{-\frac{d}{2k}}+r^{\frac{1}{2}-\frac{d}{2k}} N^{-\alpha} +
 \|S_r^N\|_{L^\infty}\right) \dd r .
\end{align*}
Hence, 
\begin{align*}
\left\|\|K \ast u_s^N\|_{L^\infty}\right\|_{L^q(\Omega)} 
&\leq    C s^{-\frac{d}{2k}} +
     C s^{\frac{1}{2}-\frac{d}{2k}}(N^{-\alpha} +1)+
     C \left\|\|S_s^N\|_{L^\infty}\right\|_{L^q(\Omega)} \\
     &\quad +        C \int_0^s (s-r)^{-\frac{1}{2}-\frac{d}{2k}}
        \left\|\|S_r^N\|_{L^\infty}\right\|_{L^q(\Omega)} \dd r.
\end{align*}

If $k < \infty$, using Lemma~\ref{lem:uniform_convolution}, we have
a uniform bound in $N$ and in $r$ for $\left\| \| S_r^N\|_{L^\infty}\right\|_{L^q(\Omega)}$ and hence
\begin{align*}
\left\|\|K \ast u_s^N\|_{L^\infty}\right\|_{L^q(\Omega)} &\leq
    C s^{-\frac{d}{2k}} +
     C \int_0^s (s-r)^{-\frac{1}{2}-\frac{d}{2k}} \dd r
     \leq C s^{-\frac{d}{2k}} + C s^{\frac{1}{2}-\frac{d}{2k}} \leq C s^{-\frac{d}{2k}},
\end{align*}
which is integrable on $(0,1)$ since $k>d$.
Note that if $k = \infty$, $\|K \ast u_s^N\|_{L^\infty} \leq \|K\|_{L^\infty} \|u_s^N\|_{TV}$, 
and Equation~\eqref{eq:moment_bound_g} is then bounded by a 
constant $C>0$ independent of $N$ and $t$.
By plugging the bound on $\big\|\|K \ast u_s^N\|_{L^\infty}\big\|_{L^q(\Omega)}$ in~\eqref{eq:moment_bound_g}, we have for any $t \in (0,1]$,
\begin{align*}
g(t) \leq e^{-\lambda t} g(0) + C,
\end{align*}
where $C$ does not depend on the initial conditions. 
Using the Markov property of the $N$-particle system and since the
constant $C$ above is uniform with respect to the initial condition, the same
estimate can be obtained at any time-interval of length $1$: for every
$n \in \mathbb{N}$,
\[
g(n+1) \leq e^{-\lambda} g(n) + C.
\]
Iterating this inequality yields that there exists a constant $C>0$ such that for all
$t>0$,
\[
g(t)
\leq C e^{-\lambda t} g(0) + C.\]
Hence, for any $t>0$,
\[
\E\bigl[|X_t^{i,N}|^q\bigr]
\leq
C e^{-\lambda q t}
\E\bigl[|X_0^{i,N}|^q\bigr]
+ C.
\]
The result follows by taking $\sup_{N \in \mathbb{N}}
 \max_{1\le i\le N}$ of the previous inequality.
\end{proof}

The following lemma is a technical result that will be used in 
the proof of Proposition~\ref{lemma:stochastic-integral}. For the sake of readability, we postpone its proof to the Appendix.

\begin{lemma}\label{some-bounds}
For any $\varepsilon\in (0,4\alpha)$ and any $q>1$, there exists a constant
$C>0$ such that, for any $t\in(0,1]$ and any measurable function
$a:[0,1]\to\mathbb R^d$,
\begin{align*}
&\int_0^t
\left\|
(1+|\cdot|)^{\frac{d+1}{q}}
\hat P_{t-s}^*
\big(\nabla V^N(\cdot-a_s)\big)
\right\|_{L^2}^2\,\dd s
 \leq
C N^{d\alpha+\frac{\varepsilon}{2}}
\int_0^t
(t-s)^{-1+\frac{\varepsilon}{4\alpha}}
(1+|a_s|)^{2\frac{d+1}{q}}\,\dd s .
\end{align*}
\end{lemma}

The following proposition gives an $L^2(\PP)$ bound 
on $\|M_t^N\|_{k^*}$, where the stochastic convolution integral $M^N$ was defined in 
Proposition~\ref{prop:mild-formulation}. This is
this term which induces the restriction on the 
mollifier parameter $\alpha$, stated in Assumption~\ref{assum:alpha}. Its proof uses the result of Lemma~\ref{lem:uniform_bounds} and of
 Lemma~\ref{some-bounds}. In fact, a careful inspection of the proof of 
 Proposition~\ref{lemma:stochastic-integral} shows that Lemma \ref{lem:uniform_bounds} is only applied when
$k \geq 2$ and in such case, for the moment of order
\(q=(d+1)\frac{k-2}{k}.\)
\begin{proposition}
\label{lemma:stochastic-integral}
Suppose that Assumptions~\ref{assum:potentials}, \ref{assum:alpha} 
and \ref{assum:initial} hold. For any $\varepsilon > 0$, there exists a constant $C>0$
such that for any $t \in (0,1]$,
\[
\bigl\|\|M_t^N\|_{k^*}\bigr\|_{L^2(\PP)}
\leq C N^{-\min(\alpha, \frac{1}{2}-
\frac{d \alpha}{2}-\varepsilon)}
\Big(1+\mathbbm{1}_{k \in (1,2)}\sqrt{ t^{-\frac12-\frac{d}{2k}}+t^{-\frac{d}{k}}}\big\lVert\mathbb{W}_1(u_0^N,u_0) \big\rVert_{L^2(\PP)} \Big),
\]
where $C$ depends on the initial conditions only 
through $f(Q)$, defined in Lemma~\ref{lem:uniform_bounds}.
\end{proposition}
\begin{proof}
Let us assume first that $k \in (2, \infty]$. Let $q$ be such that $\frac{1}{k^*}=\frac{1}{2}+\frac{1}{q}$. 
Applying Hölder's inequality with weight
$(1+|x|)^{-\frac{d+1}{q}}(1+|x|)^{\frac{d+1}{q}}$ and BDG's inequality in 
$L^{2}(\R^d)$, we obtain
\begin{align*}
\big\|\|M_t^N\|_{k^*}\big\|_{L^2(\PP)} &= \bigg\| 
    \Big\| \frac{\sqrt{2}}{N} \sum_{i=1}^N \int_0^t 
    \hat{P}^*_{t-s} \nabla V^N(\cdot - X_s^{i,N}) \cdot \dd W_s^i \Big\|_{L^{k^*}}
    \bigg\|_{L^2(\PP)}\\
    &\leq \frac{\sqrt{2}}{N} \big\| (1+|\cdot|)^{-\frac{d+1}{q}}\big\|_{L^{q}}
     \sup_{u \in [0,t]} 
    \bigg\|\Big\| \sum_{i=1}^N \int_0^u
    (1+|\cdot|)^{\frac{d+1}{q}} \, \hat{P}^*_{t-s} \nabla 
    V^N(\cdot - X_s^{i,N}) \cdot \dd W_s^i \Big\|_{L^{2}} \bigg\|_{L^2(\PP)} \\
    & \leq \frac{\sqrt{2}}{N} \sup_{u \in [0,t]} \bigg\| 
        \Big\| \Big( (1+|\cdot|)^{{2\frac{(d+1)}{q}}} 
        \sum_{i=1}^N \int_0^u 
        |\hat{P}^*_{t-s} \nabla V^N(\cdot - X_s^{i,N})|^2 \dd s
        \Big)^{\frac{1}{2}} \Big\|_{L^{2}} \bigg\|_{L^2(\PP)}.
    \end{align*}
Hence
\begin{align*}
        \left\|\|M_t^N\|_{k^*}\right\|_{L^2(\PP)} 
        & \leq  \frac{\sqrt{2}}{N} \sup_{u \in [0,t]} \bigg\|
            \Big\| \int_0^u \sum_{i=1}^N (1+|\cdot|)^{2\frac{d+1}{q}}
            \big|\hat{P}^*_{t-s} \big(\nabla V^N(\cdot - 
            X_s^{i,N})\big)\big|^2 \dd s\Big\|_{L^{1}}^{\frac{1}{2}} 
            \bigg\|_{L^2(\PP)}
            \\
        & = \frac{\sqrt{2}}{N} \sup_{u \in [0,t]} \left(\E \bigg\| \int_0^u 
            \sum_{i=1}^N (1+|\cdot|)^{2\frac{d+1}{q}}
            \big|\hat{P}^*_{t-s} \big(\nabla V^N(\cdot - 
            X_s^{i,N})\big)\big|^2 \dd s
            \bigg\|_{L^{1}}\right)^{\frac{1}{2}}
            \\
        & \leq C N^{-1} \bigg(\sum_{i=1}^N
            \E  \int_0^t\Big\| (1+|\cdot|)^{\frac{d+1}{q}}
            \hat{P}^*_{t-s} \big(\nabla V^N(\cdot - 
            X_s^{i,N})\big) \Big\|^2_{L^2} \dd s 
            \bigg)^{\frac{1}{2}}.
\end{align*}
Lemma~\ref{some-bounds}, applied pathwise with 
$(a_s)_{s \in [0,t]} = (X_s^{i,N})_{s \in [0,t]}$, yields
\begin{align*}
        \left\|\|M_t^N\|_{k^*}\right\|_{L^2(\PP)}
        &\leq C N^{-1}
        \Big(
        N^{d\alpha+\frac{\varepsilon}{2}}
        \sum_{i=1}^N
        \int_0^t
        (t-s)^{-1+\frac{\varepsilon}{4\alpha}}
        \E\left[
        (1+|X_s^{i,N}|)^{2\frac{d+1}{q}}
        \right]\dd s
        \Big)^{\frac{1}{2}}\\
        & \leq C N^{-\frac{1}{2}+
\frac{d \alpha}{2} + \varepsilon} \Big( \sup_{N\geq1}\max_{1\leq i\leq N}
    \sup_{s\geq0}
    \E\left[
        (1+|X_s^{i,N}|)^{2\frac{d+1}{q}}
    \right]\Big)^\frac12.
\end{align*}
Hence, if
\(2\frac{d+1}{q}\geq2,\)
Assumption~\ref{assum:initial} and
Lemma~\ref{lem:uniform_bounds} applied with the moment of order
$2\frac{d+1}{q}$ yield
\(\sup_{N\geq1}\max_{1\leq i\leq N}
    \sup_{s\geq0}
    \E\big[
        (1+|X_s^{i,N}|)^{2\frac{d+1}{q}}
    \big] \leq C.\)
If
\(2\frac{d+1}{q}<2,\)
then by H\"older's inequality,
\(\E\big[
        (1+|X_s^{i,N}|)^{2\frac{d+1}{q}}
    \big]
    \leq  \E\big[
        (1+|X_s^{i,N}|)^2
    \big]^{\frac{d+1}{q}},\)
and Assumption~\ref{assum:initial} with
Lemma~\ref{lem:uniform_bounds} applied with the moment of order $2$,
again gives a bound uniform in $N$, $i$, and $s$.
Therefore, in both cases,
\[f(Q) = \sup_{N\geq1}\max_{1\leq i\leq N}
    \sup_{s\geq0}
    \E\left[
        (1+|X_s^{i,N}|)^{2\frac{d+1}{q}}
    \right]
    \leq C.\]
Consequently, there exists a constant $C>0$ such that for any $t \in (0,1]$,
\[
\big\|\|M_t^N\|_{k^*}\big\|_{L^2(\PP)} \leq C N^{-\frac{1}{2}+
\frac{d \alpha}{2} + \varepsilon},
\]
where $C$ does not depend on $N$ and $t$.

Now assume that $k \in (1,2]$: this is possible only if $d=1$ and in this case the computations are similar but
 easier. Indeed, using BDG's inequality in $L^{k^*}(\R^d)$,
  we have 
 \begin{align*}
\big\|\|M_t^N\|_{k^*}\big\|_{L^2(\PP)} 
        &\leq C N^{-1} \sup_{u \in [0,t]}
         \bigg(\E \Big\| \int_0^u \sum_{i=1}^N
            |\hat{P}^*_{t-s} \nabla V^N(\cdot - 
            X_s^{i,N})|^2 \dd s
            \Big\|_{L^{k^*/2}}\bigg)^{\frac{1}{2}} \\
        &\leq C N^{-1} 
         \bigg(\E \Big\| \int_0^t \lVert \hat{P}^*_{t-s} \nabla V^N(\cdot - 
            X_s^{i,N}) \rVert_{L^\infty} \sum_{i=1}^N
            |\hat{P}^*_{t-s} \nabla V^N(\cdot - 
            X_s^{i,N})| \dd s
            \Big\|_{L^{k^*/2}}\bigg)^{\frac{1}{2}} .
\end{align*}
First, from Lemma~\ref{lem:OU-bounds}\ref{item:OU-Lp} with $p_{1}=p_{2}=\infty$, we get $\lVert \hat{P}^*_{t-s} \nabla V^N(\cdot - X_s^{i,N}) \rVert_{L^\infty} \leq C \lVert \nabla V^N(\cdot - X_s^{i,N}) \rVert_{L^\infty} \lesssim N^{2\alpha}$ (recall that $d=1$). Alternatively, using integration by parts and Lemma~\ref{lem:OU-bounds}\ref{item:OU-Lp} again, but with $p_{1}=1, p_{2}=\infty$, we get $\lVert \hat{P}^*_{t-s} \nabla V^N(\cdot - X_s^{i,N}) \rVert_{L^\infty} \leq C (t-s)^{-1} \lVert V^N\rVert_{L^1} = C (t-s)^{-1}$. Thus
\begin{equation*}
\big\|\|M_t^N\|_{k^*}\big\|_{L^2(\PP)}  \leq C N^{-1}  \bigg(\E \Big\| \int_0^t \min\big(N^{2\alpha}, (t-s)^{-1} \big) \sum_{i=1}^N
            |\hat{P}^*_{t-s} \nabla V^N(\cdot - 
            X_s^{i,N})| \dd s
            \Big\|_{L^{k^*/2}}\bigg)^{\frac{1}{2}} .
\end{equation*}
Now for the term $|\hat{P}^*_{t-s} \nabla V^N(\cdot - X_s^{i,N})|$, use integration by parts and Lemma~\ref{lem:OU-bounds}\ref{item:heat-bounds} to get that
\begin{align*}
\sum_{i=1}^N  |\hat{P}^*_{t-s} \nabla V^N(\cdot - X_s^{i,N})(x)| 
&\leq C (t-s)^{-\frac{1}{2}} \sum_{i=1}^N (G^{\kappa,F}_{t-s})^* \big[V^N(\cdot - X_s^{i,N})\big](x) \\
&= C (t-s)^{-\frac{1}{2}}\, N\, (G^{\kappa,F}_{t-s})^*u^N_{s}(x) .
\end{align*}
for some $\kappa>0$. It follows that
\begin{align*}
\big\|\|M_t^N\|_{k^*}\big\|_{L^2(\PP)}  
&\leq \frac{C}{\sqrt{N}}  \bigg(\E \Big\| \int_0^t (t-s)^{-\frac{1}{2}} \min\big(N^{2\alpha}, (t-s)^{-1} \big) (G^{\kappa,F}_{t-s})^*u^N_{s}\, \dd s
            \Big\|_{L^{k^*/2}}\bigg)^{\frac{1}{2}} \\
&\leq \frac{C}{\sqrt{N}}  \bigg(\int_0^t (t-s)^{-\frac{1}{2}} \min\big(N^{2\alpha}, (t-s)^{-1} \big) \E \| u^N_{s}\|_{L^{k^*/2}}\, \dd s\bigg)^{\frac{1}{2}}  ,
\end{align*}
using Minkowski's inequality and the fact that $ \lVert (G^{\kappa,F}_{t-s})^* \rVert_{L^{k^*/2} \to L^{k^*/2}} \leq C$ (see Lemma~\ref{lem:OU-bounds}\ref{item:sgG-Lp}). Now by interpolation, $\| u^N_{s}\|_{L^{k^*/2}} \leq 1 + \| u^N_{s}\|_{L^{k^*}}$ and $\| u^N_{s}\|_{L^{k^*}} \leq \| u_{s}\|_{L^{k^*}} + \| u^N_{s}-u_{s}\|_{L^{k^*}}$. Using~\eqref{ut-time-reg}, we have obtained that 
\begin{equation*}
\| u^N_{s}\|_{L^{k^*/2}} \leq 1 + C s^{-\frac{1}{2k}} + \| u^N_{s}-u_{s}\|_{L^{k^*}}, 
\end{equation*}
and therefore
\begin{align*}
\big\|\|M_t^N\|_{k^*}\big\|_{L^2(\PP)}  
\leq \frac{C}{\sqrt{N}}  \bigg(\int_0^t (t-s)^{-\frac{1}{2}} \min\big(N^{2\alpha}, (t-s)^{-1} \big) \big(1+ s^{-\frac{1}{2k}} + \E \| u^N_{s}-u_{s}\|_{L^{k^*}}\big)\, \dd s\bigg)^{\frac{1}{2}} .
\end{align*}
Use that 
\begin{align*}
\int_0^t (t-s)^{-\frac{1}{2}} \min\big(N^{2\alpha}, (t-s)^{-1} \big) \, \dd s \leq 3N^\alpha 
\end{align*}
and that for any $\gamma\in (0,1]$,
\begin{align}
\label{eq:weights}
\int_0^t (t-s)^{-\frac{1}{2}} \min\big(N^{2\alpha}, (t-s)^{-1} \big)\, s^{-\gamma} \, \dd s \leq C N^\alpha t^{-\gamma},
\end{align}
where the second inequality is obtained by cutting the integral in half at $s=t/2$, we deduce that for $\gamma = \frac{1}{2k}$, 
\begin{align*}
\big\|\|M_t^N\|_{k^*}\big\|_{L^2(\PP)}  
\leq C N^{-\frac{1}{2}+\frac{\alpha}{2}} \bigg(t^{-\frac{1}{4k}} + \bigg(N^{-\alpha} \int_0^t (t-s)^{-\frac{1}{2}} \min\big(N^{2\alpha}, (t-s)^{-1} \big)  \E \| u^N_{s}-u_{s}\|_{L^{k^*}}\, \dd s\bigg)^{\frac{1}{2}} \bigg).
\end{align*}
We may now use Proposition~\ref{prop:mild-formulation} to conclude as follows. Denote $\mathcal{S} \coloneqq \sup_{s\in (0,1]} \big(s^{-\frac12-\frac{1}{2k}}+s^{-\frac{1}{k}}\big)^{-1} \big\lVert \| u^N_{s}-u_{s}\|_{L^{k^*}} \big\rVert_{L^2(\PP)}$ and observe that for $k<2$ and $t\leq 1$,
\begin{align*}
N^{-\alpha} &\int_0^t (t-s)^{-\frac{1}{2}} \min\big(N^{2\alpha}, (t-s)^{-1} \big)  \E \| u^N_{s}-u_{s}\|_{L^{k^*}}\, \dd s \\
&\leq N^{-\alpha} \mathcal{S} \int_0^t (t-s)^{-\frac{1}{2}} \min\big(N^{2\alpha}, (t-s)^{-1} \big) s^{-\frac{1}{2} - \frac{1}{2k}}\, \dd s \\
&\leq C \big(t^{-\frac12-\frac{d}{2k}}+t^{-\frac{d}{k}}\big) \mathcal{S} ,
\end{align*}
using \eqref{eq:weights} with $\gamma = 1/2+1/(2k)$ in the last inequality, so that
\begin{align*}
\big\|\|M_t^N\|_{k^*}\big\|_{L^2(\PP)}  
\leq C N^{-\frac{1}{2}+\frac{\alpha}{2}} \sqrt{ t^{-\frac12-\frac{d}{2k}}+t^{-\frac{d}{k}}} \big(1 +\mathcal{S}^\frac{1}{2}) .
\end{align*}
 In view of the above and Proposition~\ref{prop:mild-formulation}, it thus comes
\begin{align*}
\big\lVert \| u^N_{t}-u_{t}\|_{L^{k^*}} \big\rVert_{L^2(\PP)} 
&\leq C (t^{-\frac{1}{2}- \frac{d}{2k}} +t^{-\frac{d}{k}})
    \left( \big\lVert\mathbb{W}_1(u_0^N,u_0) \big\rVert_{L^2(\PP)} + N^{-\alpha} \right)
    +\big\lVert \|M_t^N\|_{k^*} \big\rVert_{L^2(\PP)} \\
  &\quad  + C \int_0^t (t-s)^{-\frac{1}{2}-\frac{d}{2k}} \|M_s^N\|_{k^*} \dd s \\
 &\leq C (t^{-\frac{1}{2}- \frac{d}{2k}} +t^{-\frac{d}{k}})
    \left( \big\lVert\mathbb{W}_1(u_0^N,u_0) \big\rVert_{L^2(\PP)} + N^{-\alpha} + N^{-\frac{1}{2}+\frac{\alpha}{2}} +N^{-\frac{1}{2}+\frac{\alpha}{2}} \mathcal{S}^\frac{1}{2} \right), 
\end{align*}
and thus
\begin{align*}
\mathcal{S}
 &\leq C \left( \big\lVert\mathbb{W}_1(u_0^N,u_0) \big\rVert_{L^2(\PP)} + N^{-\alpha} + N^{-\frac{1}{2}+\frac{\alpha}{2}} +N^{-\frac{1}{2}+\frac{\alpha}{2}} \mathcal{S}^\frac{1}{2} \right)\\
 &\leq C \left( \big\lVert\mathbb{W}_1(u_0^N,u_0) \big\rVert_{L^2(\PP)} + N^{-\alpha} + N^{-\frac{1}{2}+\frac{\alpha}{2}} +\frac{1}{2} N^{-1+\alpha} \right)   + \frac{1}{2} \mathcal{S},
\end{align*}
from which the result follows.
\end{proof}

\section{Propagation of chaos for mollified particles: proof of Theorem~\ref{thm:main}}\label{sec:main-res}

We are now ready to prove Theorem~\ref{thm:main}. The scheme of proof consists in deriving a bound on $ \frac{\dd}{\dd t} |X_t^{i,N} - X_t^i|^2$, where $X^i$ is the McKean-Vlasov particle with the same noise as $X^{i,N}$, for $t\in (0,1]$, with constants that do no depend on the initial conditions of the system. Along the way, the Lipschitz bound on $K\ast u_{t}$ from Proposition~\ref{prop:lip-bound} will be used, then Grönwall's inequality on the sum over all particles will give an upper bound on $\frac{1}{N}\sum_{i=1}^N |X_t^{i,N}-X_t^i|^2$ that depend on $\frac{1}{N}\sum_{i=1}^N |X_0^{i,N}-X_0^i|^2$ and on $\|u_t^N - u_t\|_{k^*}$. Proposition~\ref{prop:mild-formulation}, Lemma~\ref{W1-bound} and Proposition~\ref{lemma:stochastic-integral} then permit to obtain an explicit estimate for $t\in (0,1]$. Leveraging precisely the contraction effect of $F$ in the previous step then allows to iterate this scheme on all subsequent time intervals of length $1$ thanks to the Markov property of the whole system.

\begin{proof}[Proof of Theorem \ref{thm:main}]

Consider $(X_t^i)_{i \in [1,N]}$, $N$ independent strong solutions of the 
 McKean-Vlasov
 SDE  \eqref{mckean-vlasov} driven by $N$ independent Brownian motions $(W_t^i)_{i \in [1,N]}$ 
 and with initial condition $(X_0^{i})_{i \in [1,N]}$.
Consider the system of $N$ interacting particles $(X_t^{i,N})_{i \in [1,N]}$ defined 
as strong solutions of
 \eqref{particle-system-intro}, 
driven by the same Brownian motions $(W_t^i)_{i \in [1,N]}$, and
having initial condition $(X_0^{i,N})_{i \in [1,N]}$ such that
 \((X_0^{1,N},\dots, X_0^{N,N}, X_0^1, \dots, X_0^N)\) realize the optimal coupling of 
 $\mathbb{W}_2^N(\mathcal{L}_{X_0^N}, u_0^{\otimes N}).$

Recall from Proposition~\ref{prop:lip-bound} that $K \ast u_t$ is Lipschitz with constant $C_{1} \|K\|_{L^k} t^{-\frac12-\frac{d}{2k}}$, for some $C_{1} = C_{0}(1+ \lVert K\rVert_{L^k}^2)$, see Remark~\ref{rem:track-K-2}.
Hence using Itô's formula, it comes that for any $i \in [1,N]$,
 \begin{align} 
 \nonumber
 \frac{\dd}{\dd t} |X_t^{i,N} - X_t^i|^2 &= -2 \big(F(X_t^{i,N}) + K\ast u_t^N(X_t^{i,N})
  - F(X_t^i) - K \ast u_t(X_t^i)\big) \cdot (X_t^{i,N} - X_t^i) \\ \nonumber
 & \leq - 2 \lambda |X_t^{i,N} - X_t^i|^2 -
 2 \big(K\ast u_t^N(X_t^{i,N}) - K \ast u_t(X_t^i)\big) 
 \cdot (X_t^{i,N} - X_t^i) \\ \label{ito-coupling}
 & \leq 2(C_{1} \|K\|_{L^k} t^{-\frac12-\frac{d}{2k}} -  \lambda) |X_t^{i,N} -
  X_t^i|^2 +
 2 \|K\|_{L^k} \|u_t^N - u_t\|_{k^*} |X_t^{i,N} - X_t^i|.
 \end{align}
Averaging \eqref{ito-coupling} over $i\in \{1,\dots,N\}$, we obtain
\begin{align*}
\frac{\dd}{\dd t} \Big(\frac{1}{N}\sum_{i=1}^N |X_t^{i,N}-X_t^i|^2\Big)
&\leq
2\big(C_1\|K\|_{L^k}t^{-\frac12-\frac{d}{2k}}
-\lambda\big) \Big(\frac{1}{N}\sum_{i=1}^N |X_t^{i,N}-X_t^i|^2\Big)\\
&\quad
+2\|K\|_{L^k}\|u_t^N-u_t\|_{k^*}
\Big(\frac1N\sum_{i=1}^N |X_t^{i,N}-X_t^i|\Big).
\end{align*}
Therefore, by the Cauchy-Schwarz inequality,
\begin{align}
\nonumber
\frac{\dd}{\dd t} \Big(\frac{1}{N}\sum_{i=1}^N |X_t^{i,N}-X_t^i|^2\Big)
&\leq 2\big(C_1\|K\|_{L^k}t^{-\frac12-\frac{d}{2k}}
-\lambda\big) \Big(\frac{1}{N}\sum_{i=1}^N |X_t^{i,N}-X_t^i|^2\Big)\,\\
&\quad
+2\|K\|_{L^k}\|u_t^N-u_t\|_{k^*}
\Big(\frac1N\sum_{i=1}^N |X_t^{i,N}-X_t^i|^2\Big)^{1/2}.
\label{eq:S-coupling}
\end{align}
Now, for $\delta>0$, set
\[g(t)\coloneqq\sqrt{\frac{1}{N}\sum_{i=1}^{N} |X_t^{i,N}-X_t^i|^2+\delta}.\]
Differentiating $g$ and using~\eqref{eq:S-coupling}, it comes
\begin{align*}
&\frac{\dd}{\dd t} g(t)
=
\frac{1}{2g(t)}\,\frac{\dd}{\dd t} \Big(\frac{1}{N}\sum_{i=1}^N |X_t^{i,N}-X_t^i|^2\Big)\\
&\leq\big(
C_1\|K\|_{L^k}t^{-\frac12-\frac{d}{2k}}
-\lambda\big)
\frac{\frac{1}{N}\sum_{i=1}^{N} |X_t^{i,N}-X_t^i|^2}{g(t)}\,
+\|K\|_{L^k}\|u_t^N-u_t\|_{k^*}
\frac{\big(\frac1N\sum_{i=1}^N |X_t^{i,N}-X_t^i|^2
\big)^{1/2}}{g(t)}\,\\
&\leq \big(C_1\|K\|_{L^k}t^{-\frac12-\frac{d}{2k}}
-\lambda\big)g(t) 
- \big(C_1\|K\|_{L^k}t^{-\frac12-\frac{d}{2k}}
-\lambda\big) \frac{\delta}{g(t)} +\|K\|_{L^k}
\|u_t^N-u_t\|_{k^*}\\
&\leq \big(C_1\|K\|_{L^k}t^{-\frac12-\frac{d}{2k}}
-\lambda\big) g(t)
+\|K\|_{L^k}\|u_t^N-u_t\|_{k^*} + \lambda\sqrt{\delta} .
\end{align*}
Using Gr\"onwall's lemma, it follows that for any $t\in(0,1]$,
\begin{align*}
g(t)&\leq g(0)
e^{\int_0^t
(C_1\|K\|_{L^k}s^{-\frac12-\frac{d}{2k}}
-\lambda)\dd s}
+\int_0^t e^{\int_s^t (C_1\|K\|_{L^k}u^{-\frac12-\frac{d}{2k}}
-\lambda)\dd u}
\big(\|K\|_{L^k}\|u_s^N-u_s\|_{k^*}
+\lambda\sqrt{\delta}\big)\, \dd s.
\end{align*}
Since $\big(\frac{1}{N}\sum_{i=1}^{N} |X_t^{i,N}-X_t^i|^2\big)^{1/2} \leq g(t)$ and
letting $\delta\to0$, we obtain for any $t\in(0,1]$,
\begin{align}
\Big(\frac{1}{N}\sum_{i=1}^{N} |X_t^{i,N}-X_t^i|^2\Big)^{1/2} 
&\leq
\Big(\frac{1}{N}\sum_{i=1}^{N} |X_0^{i,N}-X_0^i|^2\Big)^{1/2} 
e^{\int_0^t
(C_1\|K\|_{L^k}s^{-\frac12-\frac{d}{2k}}
-\lambda)\dd s}
\nonumber\\
&\quad +\|K\|_{L^k}
\int_0^t
e^{\int_s^t (C_1\|K\|_{L^k}u^{-\frac12-\frac{d}{2k}}-\lambda) \dd u}
\|u_s^N-u_s\|_{k^*}\,\dd s.
\label{eq:pathwise-coupling}
\end{align}
Taking the $L^2(\PP)$ norm in~\eqref{eq:pathwise-coupling}
and using Minkowski's integral inequality yields
\begin{align}\nonumber
\E\Big[ \Big(\frac{1}{N}\sum_{i=1}^{N} |X_t^{i,N}-X_t^i|^2\Big)\Big]^{1/2}
&\leq
\E\Big[ \Big(\frac{1}{N}\sum_{i=1}^{N} |X_0^{i,N}-X_0^i|^2 \Big)\Big]^{1/2}
\, e^{\int_0^t (C_1\|K\|_{L^k}s^{-\frac12-\frac{d}{2k}}
-\lambda)\dd s}\\\label{eq:W2-coupling-1}
&\quad+\|K\|_{L^k}
\int_0^te^{\int_s^t (C_1\|K\|_{L^k}u^{-\frac12-\frac{d}{2k}}-\lambda) \dd u}
\left\| \|u_s^N-u_s\|_{k^*}
\right\|_{L^2(\PP)}\,\dd s.
\end{align}
By definition of the coupling, 
\( \mathbb{W}_2^N(\mathcal{L}_{X_t^N},u_t^{\otimes N})
\leq
\E\big[\big(\frac{1}{N}\sum_{i=1}^{N} |X_t^{i,N}-X_t^i|^2\big)\big]^{1/2}\),
and the optimality of the initial coupling gives
\(\E\big[\big(\frac{1}{N}\sum_{i=1}^{N} |X_0^{i,N}-X_0^i|^2\big)\big]^{1/2}
=\mathbb{W}_2^N(\mathcal{L}_{X_0^N},u_0^{\otimes N}).\)
Hence for any $t\in(0,1]$,
\begin{align}
\mathbb{W}_2^N(\mathcal{L}_{X_t^N},u_t^{\otimes N})
&\leq \mathbb{W}_2^N(\mathcal{L}_{X_0^N},u_0^{\otimes N})\,
e^{\int_0^t (C_1\|K\|_{L^k}s^{-\frac12-\frac{d}{2k}}
-\lambda)\dd s}
\nonumber\\
&\quad +\|K\|_{L^k}
\int_0^t
e^{\int_s^t (C_1\|K\|_{L^k}u^{-\frac12-\frac{d}{2k}}
-\lambda)\dd u}
\left\|\|u_s^N-u_s\|_{k^*}
\right\|_{L^2(\PP)}\,\dd s.
\label{eq:W2-coupling}
\end{align}

Now in view of Proposition~\ref{prop:mild-formulation}, Lemma~\ref{W1-bound}
and Proposition~\ref{lemma:stochastic-integral}, we get that for any
$\varepsilon>0$ there exists $C_{2}>0$ such that for any $t\in(0,1]$,
\begin{align}
\left\|\|u_t^N-u_t\|_{k^*}\right\|_{L^2(\PP)}
&\leq C_{2}\big(t^{-\frac12-\frac{d}{2k}}
+t^{-\frac{d}{k}} \big)
\big(\big\|\mathbb{W}_1(u_0^N,u_0) \big\|_{L^2(\PP)}+N^{-\alpha}\big)
\nonumber\\
&\quad +C_{2} N^{-\min(\alpha, \frac{1}{2}-
\frac{d \alpha}{2}-\varepsilon)}
\Big(1+\mathbbm{1}_{k \in (1,2)}\sqrt{ t^{-\frac12-\frac{d}{2k}}+t^{-\frac{d}{k}}}\big\lVert\mathbb{W}_1(u_0^N,u_0) \big\rVert_{L^2(\PP)} \Big)
\nonumber\\
&\leq C_{2}\Theta(t)
\Big(\mathbb{W}_2^N(\mathcal{L}_{X_0^N},u_0^{\otimes N})
+N^{-\min\big(\alpha, \frac12-\frac{d\alpha}{2}-\varepsilon
\big)} +f(N)\Big),
\label{L2-diff}
\end{align}
where we recall that $f(N)$ was defined in~\eqref{eq:defFN} and 
\begin{equation}
\label{theta}
\Theta(t) \coloneqq t^{-\frac12-\frac{d}{2k}}+t^{-\frac{d}{k}}.
\end{equation}
Here and in this whole proof, Remark~\ref{rem:track-K-2} allows to write $C_{2}$ of the form $\exp(C_{0} \lVert K\rVert_{L^k})$, for some $C_{0}>0$ that does not depend on $K$ and whose value may change from line to line. 
Besides, observing that for any admissible $\alpha$,
\[f(N)\leq C
N^{-\min\big(\alpha,
    \frac12-\frac{d\alpha}{2}-\varepsilon\big)},\]
we rewrite \eqref{L2-diff} as
\begin{align}
\label{L2-diff-2}
\left\|\|u_t^N-u_t\|_{k^*}
\right\|_{L^2(\PP)}
\leq C_{2}\Theta(t) \Big(
\mathbb{W}_2^N(\mathcal{L}_{X_0^N},u_0^{\otimes N})
+N^{-\min\big( \alpha,
    \frac12-\frac{d\alpha}{2}-\varepsilon\big)}\Big).
\end{align}
Plugging~\eqref{L2-diff-2} into~\eqref{eq:W2-coupling}, we obtain that for any $t\in(0,1]$,
\begin{align}
\nonumber
&\mathbb{W}_2^N(\mathcal{L}_{X_t^N},u_t^{\otimes N})
\leq \mathbb{W}_2^N(\mathcal{L}_{X_0^N},u_0^{\otimes N})\, e^{\int_0^t (C_{1}\|K\|_{L^k}s^{-\frac12-\frac{d}{2k}}-\lambda)\dd s}
\\ \nonumber
&\quad +C_{2}\|K\|_{L^k}\int_0^t e^{\int_s^t (C_{1}\|K\|_{L^k} u^{-\frac12-\frac{d}{2k}}-\lambda
)\dd u}\,
\Theta(s)\Big(\mathbb{W}_2^N(\mathcal{L}_{X_0^N},
u_0^{\otimes N})+N^{-\min\big(\alpha,
    \frac12-\frac{d\alpha}{2}-\varepsilon\big)}
\Big)\,\dd s\\ \label{eq:gronwall-2}
& \leq \mathbb{W}_2^N(\mathcal{L}_{X_0^N},u_0^{\otimes N}) \Big(e^{\int_0^t
(C_{1}\|K\|_{L^k}s^{-\frac12-\frac{d}{2k}}-\lambda)\dd s} +C_{2}\|K\|_{L^k} e^{C_{1}\|K\|_{L^k}}\Big)
 + C_{2} N^{-\min\big(\alpha,
    \frac12-\frac{d\alpha}{2}-\varepsilon\big)}.
\end{align}    
Recalling that $C_{1}$ is of the form $C_{0}(1 + \lVert K\rVert_{L^k}^2)$ and $C_{2}$ of the form $\exp(C_{0} \lVert K\rVert_{L^k})$ for some $C_{0}>0$, for any $\lambda>0$, one can choose $\|K\|_{L^k}$ small enough so that
\[
\mathcal{A} \coloneqq e^{\int_0^1
(C_{1}\|K\|_{L^k}s^{-\frac12-\frac{d}{2k}}-\lambda)\dd s}+C_{2}\|K\|_{L^k} e^{C_{1}\|K\|_{L^k}}<1.
\]

To conclude, denote by $P^N_t$ the semigroup associated with the particle
system~\eqref{particle-system-intro}. Fix $n\in\mathbb{N}$, and define the
shifted flow
\( u^{(n)}_s \coloneqq  u_{n+s}, s\geq 0.\)
By time-homogeneity and uniqueness, $(u^{(n)}_s)_{s\geq 0}$ is the unique
solution of the McKean-Vlasov Equation~\eqref{mckean-vlasov} with initial
condition $u_n$. In particular,
\( u_{n+1}
    = P^{u^{(n)},*}_{0,1} u_n.\)
Hence, Inequality~\eqref{eq:gronwall-2} can be iterated to obtain
\begin{align*}
    \mathbb{W}^N_2\Big(
        P^{N,*}_{n+1}\mathcal{L}_{X_0^N},
        \big(P^{u^{(0)},*}_{0,n+1}u_0\big)^{\otimes N}
    \Big)
    &=\mathbb{W}^N_2\Big(
        P^{N,*}_{1}\mathcal{L}_{X_n^N},
        \big(P^{u^{(n)},*}_{0,1}u_n\big)^{\otimes N}
    \Big)
    \\
    &\leq
    \mathcal{A}\,\mathbb{W}^N_2\big(
        \mathcal{L}_{X_n^N},
        u_n^{\otimes N}\big)
    + CN^{-\min\left(\alpha,
    \frac12-\frac{d\alpha}{2}-\varepsilon\right)}
    \\
    &\leq
    \mathcal{A}^{n+1}
    \mathbb{W}^N_2\big(
 \mathcal{L}_{X_0^N},
        u_0^{\otimes N}\big)
    + CN^{-\min\left(
    \alpha,
    \frac12-\frac{d\alpha}{2}-\varepsilon
\right)}
    \sum_{m=0}^{n} \mathcal{A}^m.
\end{align*}
Since $\mathcal{A}<1$, it follows that
\[\mathbb{W}^N_2
\bigl(\mathcal{L}_{X_n^N},u_n^{\otimes N}\bigr)
\leq
\mathcal{A}^n
\mathbb{W}^N_2
\bigl(\mathcal{L}_{X_0^N},u_0^{\otimes N}\bigr)
+\frac{C}{1-\mathcal{A}}N^{-\min\left(
    \alpha,
    \frac12-\frac{d\alpha}{2}-\varepsilon
\right)},
\quad n\in\mathbb N.\]
Now relying on the previous bound and \eqref{eq:gronwall-2},  it is clear that 
there exist constants
$C,c>0$ such that for any $t>0$,
\begin{equation}\label{eq:first-estim}
\mathbb{W}^N_2
\bigl(\mathcal{L}_{X_t^N},u_t^{\otimes N}\bigr)
\leq C\left(
e^{-ct}
\mathbb{W}^N_2
\bigl(\mathcal{L}_{X_0^N},u_0^{\otimes N}\bigr)
+CN^{-\min\left(
    \alpha,
    \frac12-\frac{d\alpha}{2}-\varepsilon
\right)}
\right).
\end{equation}

\medskip

It remains to prove the estimate on the mollified empirical density.
For $t\in(0,1]$, since $k>d$ we have
\begin{equation*}
    \Theta(t)
    = t^{-\frac12-\frac{d}{2k}}+t^{-\frac{d}{k}}
    \leq
    C t^{-\frac12-\frac{d}{2k}}.
\end{equation*}
Hence, by \eqref{L2-diff-2},
\begin{align*}
    &
    \left\|\|u_t^N-u_t\|_{k^*}\right\|_{L^2(\PP)}\leq C (1\wedge t)^{-\frac12-\frac{d}{2k}}
    \left(
    \mathbb W_2^N\left(\mathcal L_{X_0^N},u_0^{\otimes N}\right)
    +
    N^{-\min\left(\alpha,\frac12-\frac{d\alpha}{2}
    -\varepsilon\right)}
    \right).
\end{align*}
For $t>1$, we apply \eqref{L2-diff-2} on the time interval
$[t-1,t]$ and obtain
\begin{align*}
    \left\|\|u_t^N-u_t\|_{k^*}\right\|_{L^2(\PP)}
    &\leq
    C\left(
    \mathbb W_2^N\left(\mathcal L_{X_{t-1}^N},
    u_{t-1}^{\otimes N}\right)
    +
    N^{-\min\left(\alpha,\frac12-\frac{d\alpha}{2}
    -\varepsilon\right)}
    \right)
    \\
    &\le
    C\left(
    e^{-c(t-1)}
    \mathbb W_2^N\left(\mathcal L(X_0^N),u_0^{\otimes N}\right)
    +
    N^{-\min\left(\alpha,\frac12-\frac{d\alpha}{2}
    -\varepsilon\right)}
    \right),
\end{align*}
where we used \eqref{eq:first-estim} in the last inequality.
Combining the two cases gives, for any $t>0$,
\begin{align*}
    \left\|\|u_t^N-u_t\|_{k^*}\right\|_{L^2(\PP)}
    \leq C(1\wedge t)^{-\frac12-\frac{d}{2k}}
    \left( e^{-ct}
    \mathbb W_2^N\left(\mathcal L(X_0^N),u_0^{\otimes N}\right)
    +
    N^{-\min\left(\alpha,\frac12-\frac{d\alpha}{2}
    -\varepsilon\right)}
    \right).
\end{align*}
This concludes the proof.
\end{proof}

\section{Propagation of chaos for non-mollified particles: Proof of Theorem~\ref{thm:main-non-mollified}}\label{sec:non-mollified}

Let $(F,K)$ satisfying Assumption~\ref{assum:potentials}, with the 
additional assumption that $k>d\vee2$. Without loss of generality, 
we also set \(K(0) = 0\). We consider the following non-mollified system of particles~\eqref{particle-system-intro-gen}:
\begin{equation}\label{non-mollified-particle-system}
    \dd Y_t^{i,N}
    =
    -F(Y_t^{i,N}) \dd t
    -
    \frac{1}{N}\sum_{j=1}^N K(Y_t^{i,N}-Y_t^{j,N}) \dd t
    +
    \sqrt{2} \dd W_t^i,
    \qquad i=1,\ldots,N, 
\end{equation}
Also assume that the initial conditions $Y_0^N$ satisfy
Assumption~\ref{assum:initial}, with $X_0^N$ replaced by $Y_0^N$.
For each fixed $N$, weak existence for
\eqref{non-mollified-particle-system} follows from
Lemma~\ref{lem:martingale-bounds} by taking $r=N$. Indeed, in this
case the auxiliary system consists of $N$ independent generalized
Ornstein-Uhlenbeck processes, and the Girsanov change of measure
introduced in Appendix~\ref{annex:part-girsanov} recovers
\eqref{non-mollified-particle-system}.
Patwhise uniqueness can then be obtained from \cite{HaoRocknerZhang}, using 
a localization in time argument. 
Hence, strong well-posedness holds for \eqref{non-mollified-particle-system}.
Note that in fact weak well-posedness is already sufficient for the coupling 
argument below, since, after enlarging the probability space
if necessary, we may realize an optimal coupling of the initial conditions and then construct,
 on the same space and driven by the same Brownian 
 motions $(W^i)_{i=1,\ldots,N}$, the $N$ independent strong 
 solutions of the McKean-Vlasov equation.
 Let us re-examine the argument from Section~\ref{sec:main-res},
 with the non-mollified system of
particles~\eqref{non-mollified-particle-system} instead of the mollified one~
\eqref{particle-system-intro}. A synchronous coupling argument, It\^o's formula
and Proposition~\ref{prop:lip-bound}
allows to re-write~\eqref{ito-coupling} and obtain for any $t\in(0,1]$,
\begin{align*}
    \frac{\dd}{\dd t}|Y_t^{i,N}-X_t^i|^2
    &= -2\Big( F(Y_t^{i,N})
        +K*\mu_t^N(Y_t^{i,N})
        -F(X_t^i) -K*u_t(X_t^i)
      \Big)\cdot(Y_t^{i,N}-X_t^i)
    \\
    &\leq
    -2\lambda |Y_t^{i,N}-X_t^i|^2-
    2\Big( K*\mu_t^N(Y_t^{i,N})
        -K*u_t(X_t^i)
      \Big)\cdot(Y_t^{i,N}-X_t^i)
    \\
    &\leq 2\left( C\|K\|_{L^k}t^{-\frac12-\frac{d}{2k}}-\lambda\right)
      |Y_t^{i,N}-X_t^i|^2
    \\& \quad + 2
    \big|K*(\mu_t^N-u_t)(Y_t^{i,N})\big|\,
    |Y_t^{i,N}-X_t^i|,
\end{align*}
where we recall that $\mu_t^N$ is the empirical 
measure of the particle system \eqref{non-mollified-particle-system} at time $t$,
and that $(X_t^i)_{i=1,\ldots,N}$ are $N$ independent strong solutions 
of the McKean-Vlasov SDE~\eqref{mckean-vlasov} driven 
by $(W_t^i)_{i=1,\ldots,N}$.
Averaging over $i\in\{1,...,N\}$ and using the Cauchy-Schwarz inequality, we obtain
\begin{align*}
    \frac{\dd}{\dd t}
    \Big(
        \frac1N\sum_{i=1}^N|Y_t^{i,N}-X_t^i|^2
    \Big)
    &\leq
    2\Big(
        C\|K\|_{L^k}t^{-\frac12-\frac{d}{2k}}
        -\lambda
      \Big)
      \frac1N\sum_{i=1}^N|Y_t^{i,N}-X_t^i|^2
    \\
    &\quad  + 2
    \Big(\frac1N\sum_{i=1}^N
        |K*(\mu_t^N-u_t)(Y_t^{i,N})|^2
    \Big)^{1/2}
    \Big( \frac1N\sum_{i=1}^N|Y_t^{i,N}-X_t^i|^2
    \Big)^{1/2}.
\end{align*}
Using the same argument as in Section~\ref{sec:main-res}, meaning that we use an auxiliary function
\(g(t) = \Big(\frac1N\sum_{i=1}^N|Y_t^{i,N}-X_t^i|^2+\varepsilon
    \Big)^{1/2},\)
then make $\varepsilon\to0$, we obtain
\begin{align*}
    \frac{\dd}{\dd t}
    \Big(
        \frac1N\sum_{i=1}^N|Y_t^{i,N}-X_t^i|^2
    \Big)^{1/2}
    &\leq
    \Big(
        C\|K\|_{L^k}t^{-\frac12-\frac{d}{2k}}
        -\lambda
    \Big)
    \Big(
        \frac1N\sum_{i=1}^N|Y_t^{i,N}-X_t^i|^2
    \Big)^{1/2}
    \\ &+
    \Big(\frac1N\sum_{i=1}^N |K*(\mu_t^N-u_t)(Y_t^{i,N})|^2
    \Big)^{1/2}.
\end{align*}
Using Gr\"onwall's lemma and taking the $L^2(\PP)$ norm, we obtain
for any $t\in(0,1]$,
\begin{align*}
    W_2^N(\mathcal L Y_t^N,u_t^{\otimes N})
    &\leq
    e^{\int_0^t
    ( C\|K\|_{L^k}s^{-\frac12-\frac{d}{2k}} -\lambda)\dd s}
    W_2^N(\mathcal L Y_0^N,u_0^{\otimes N})
    \\&\quad +\int_0^t
    e^{\int_s^t
    ( C\|K\|_{L^k}u^{-\frac12-\frac{d}{2k}}
        -\lambda)\dd u}
    \Big( \mathbb E \Big[
            \frac1N\sum_{i=1}^N
            |K*(\mu_s^N-u_s)(Y_s^{i,N})|^2\Big]
    \Big)^{1/2}    \dd s,
\end{align*}
where we recall that we choose
 $(X_0^1,\ldots,X_0^N)$ and $(Y_0^{1,N},\ldots,Y_0^{N,N})$ to realize the optimal coupling of
\(\mathbb{W}_2^N(\mathcal{L}_{Y_0^N},u_0^{\otimes N})\).
Hence, in all the following, the goal is to control the term
$\big(\mathbb E \big[\frac1N\sum_{i=1}^N
            |K*(\mu_s^N-u_s)(Y_s^{i,N})|^2
        \big]\big)^{1/2}$.
The idea is to introduce the pivot term $K*v_s^N(Y_s^{i,N})$, where $v^N$ is the mollified empirical measure (still of the non-mollified system). 
Precisely, we have
\begin{align}\nonumber
    \Big(\mathbb E&
        \Big[ \frac1N\sum_{i=1}^N |K*(\mu_s^N-u_s)(Y_s^{i,N})|^2
        \Big] \Big)^{1/2}
    \\ \nonumber
    &\leq
    \Big(\mathbb E
        \Big[
            \frac1N\sum_{i=1}^N
            |K*(\mu_s^N-v_s^N)(Y_s^{i,N})|^2
        \Big]\Big)^{1/2}
    +
    \|K\|_{L^k}
    \big\|\|v_s^N-u_s\|_{k^*}\big\|_{L^2(\PP)}
    \\ \label{two-side-decomposition}
    &\leq
    \Big(\frac1{N^2}
        \sum_{i,j=1}^N
        \mathbb E
        \big[|(K-K*V^N)(Y_s^{i,N}-Y_s^{j,N})|^2 \big] \Big)^{1/2}
    + \|K\|_{L^k}
    \big\|\|v_s^N-u_s\|_{k^*}\big\|_{L^2(\PP)}.
\end{align}
On the one hand, we will adapt Section \ref{sec:micro-macro}
to control $\left\|\|v_s^N-u_s\|_{k^*}\right\|_{L^2(\PP)}$, and on the other hand we will derive heat-like uniform-in-$N$ estimates for the density of the difference of particles in order to bound
$ \mathbb E\big[|(K-K*V^N)(Y_s^{i,N}-Y_s^{j,N})|^2 \big]$. 
Let us first focus, in Step~1, on the first term
\(\mathbb E \big[  |(K-K*V^N)(Y_s^{i,N}-Y_s^{j,N})|^2\big]\)
and then in Step~2, we will focus on the second term
$\|K\|_{L^k} \left\|\|u_s^N-u_s\|_{k^*}\right\|_{L^2(\PP)}$.

\noindent\textbf{Step~1.}
For $i\neq j$ and $x\in(\mathbb R^d)^N$, denote by
$q_s^{N,i,j}(x,\cdot)$ the density of
$Y_s^{i,N}-Y_s^{j,N}$ when the particle system starts from
$Y_0^N=x$.
Then, for $i\neq j$, H\"older's inequality gives
\begin{align*}
    \mathbb E
    \left[
        |(K-K*V^N)(Y_s^{i,N}-Y_s^{j,N})|^2
    \right]
    &=
    \int_{(\mathbb R^d)^N}
    \int_{\mathbb R^d}
        |(K-K*V^N)(z)|^2
        q_s^{N,i,j}(x,z)\,\dd z\,\mathcal L(Y_0^N)(\dd x)
    \\
    &\leq
    \|K-K*V^N\|_{L^k}^2
    \sup_{x\in(\mathbb R^d)^N}
    \|q_s^{N,i,j}(x,\cdot)\|_{L^{\frac{k}{k-2}}}.
\end{align*}
The rest of Step~1 is therefore devoted to proving that, for
$k>(d \vee 2)$ and every $p\in[1,\infty)$, there exists a constant $C>0$
such that
\begin{equation}
\label{eq:heat-like-estimate}
    \sup_{N\in\N^*}
    \sup_{i\neq j}
    \sup_{x\in(\mathbb R^d)^N}
    \|q_s^{N,i,j}(x,\cdot)\|_{L^p}
    \leq
    C s^{-\frac d2(1-\frac1p)},
    \qquad s\in(0,1].
\end{equation}
By relabelling the coordinates
of the particles, and since the estimate we aim for
is uniform with respect to the initial configuration $x$, it is enough
to consider the pair $(i,j)=(1,2)$. 
In order to obtain~\eqref{eq:heat-like-estimate}, we use
Lemma~\ref{lem:martingale-bounds} with $r=2$, and use in the following
the notations introduced in Appendix~\ref{annex:part-girsanov}.
Fix \(x=(x_1,\ldots,x_N)\in(\mathbb R^d)^N,\)
and consider the auxiliary system
$\widehat Y^N$ defined in~\eqref{eq:auxiliary-particle-system} with
$r=2$ and $\widehat Y_0^N=x$. Under $\hat{\mathbb P}$,
$\widehat Y^{1,N}$ and $\widehat Y^{2,N}$ are independent generalized
Ornstein-Uhlenbeck processes. On the other hand, by
Lemma~\ref{lem:martingale-bounds}, under the probability measure
$\mathbb Q^{2,N}$ defined by
\[ \dd\mathbb Q^{2,N} =Z_1^{2,N}\,\dd \hat{\mathbb P},\]
the process $\widehat Y^N$ has the same law as the solution of~\eqref{non-mollified-particle-system} starting from $x$.
Since
\[ \left.
    \frac{\dd\mathbb Q^{2,N}}{\dd\mathbb P}
    \right|_{\mathcal F_t}
    =Z_t^{2,N},\]
H\"older's inequality, the independence of
$\widehat Y^{1,N}$ and $\widehat Y^{2,N}$ under $\hat{\mathbb P}$,
Lemma~\ref{lem:OU-bounds} and Lemma~\ref{lem:martingale-bounds} give, 
for any $p\in(1,\infty)$,
\begin{align*}
    \|q_t^{N,1,2}(x,\cdot)\|_{L^p}
    &= \sup_{\|\varphi\|_{L^{p^*}}\leq1}
    \left| \mathbb E_{\mathbb Q^{2,N}}
        \left[\varphi( \widehat Y_t^{1,N}-\widehat Y_t^{2,N} ) \right]
    \right|
    \\
    &= \sup_{\|\varphi\|_{L^{p^*}}\leq1}
    \left| \mathbb E_{\hat{\mathbb P}} 
    \left[Z_t^{2,N} \varphi( \widehat Y_t^{1,N}-\widehat Y_t^{2,N})
    \right]\right|
    \\
    &\leq
    \|Z_t^{2,N}\|_{L^p(\hat{\mathbb P})}
    \sup_{\|\varphi\|_{L^{p^*}}\leq1}
    \left[ \mathbb E_{\hat{\mathbb P}}
    \left|\varphi(\widehat Y_t^{1,N}-\widehat Y_t^{2,N})
    \right|^{p^*}\right]^{1/p^*}
    \\
    &=\|Z_t^{2,N}\|_{L^p(\hat{\mathbb P})}
    \sup_{\|\varphi\|_{L^{p^*}}\leq1}
    \left( \int_{(\mathbb R^d)^2}
    |\varphi(y_1-y_2)|^{p^*}
    \widehat p_t(x_1,y_1)
    \widehat p_t(x_2,y_2)
    \dd y_1\dd y_2 \right)^{1/p^*}
    \\
    &\leq
    C
    t^{-\frac{d}{2p^*}}
    \|Z_t^{2,N}\|_{L^p(\hat{\mathbb P})} \leq
    C t^{-\frac d2(1-\frac1p)} \sup_{N \geq 2}
    \sup_{x\in(\mathbb R^d)^N}
    \|Z_t^{2,N}\|_{L^p(\hat{\mathbb P})}
    \leq
    Ct^{-\frac d2(1-\frac1p)}.
\end{align*}
For $p=1$, the estimate follows directly from the fact that
$q_t^{N,1,2}(x,\cdot)$ is a probability density. We can hence conclude
that estimate~\eqref{eq:heat-like-estimate} holds.

Going back to the first term in \eqref{two-side-decomposition}, we
separate the diagonal and off-diagonal contributions. Using the estimate
\eqref{eq:heat-like-estimate}
proved above for $i\neq j$, H\"older's
inequality, and the scaling of $V^N$, we obtain
\begin{align}
\Big(
    \frac1{N^2}&
    \sum_{i,j=1}^N
    \E\left[
        |(K-K*V^N)(Y_s^{i,N}-Y_s^{j,N})|^2
    \right]
\Big)^{1/2}
\nonumber\\
&\leq
\Big(
    \frac1{N^2}
    \sum_{i\neq j}
    \E\left[
        |(K-K*V^N)(Y_s^{i,N}-Y_s^{j,N})|^2
    \right]
\Big)^{1/2}
\nonumber\\
&\quad+
\Big(
    \frac1{N^2}
    \sum_{i=1}^N
    |(K-K*V^N)(0)|^2
\Big)^{1/2}
\nonumber\\
&\leq
\|K-K*V^N\|_{L^k}
\Big(
    \frac1{N^2}
    \sum_{i\neq j}
    \sup_{x\in(\mathbb R^d)^N}
    \|q_s^{N,i,j}(x,\cdot)\|_{L^{\frac{k}{k-2}}}
\Big)^{1/2}
+\frac1{\sqrt N}|(K*V^N)(0)|
\nonumber\\
&\leq
C s^{-\frac d{2k}}
\|K-K*V^N\|_{L^k}
+\frac1{\sqrt N}
\|K\|_{L^k}
\|V^N\|_{L^{k^*}}
\nonumber\\
&\leq
C s^{-\frac d{2k}}
\|K-K*V^N\|_{L^k}
+C\|K\|_{L^k}
N^{-\frac12+\frac{d\alpha}{k}}.
\label{eq:first-term-final}
\end{align}

\noindent\textbf{Step~2.}
We now adapt Section~\ref{sec:micro-macro} in order to estimate the second term in~\eqref{two-side-decomposition} which is
$\left\|\|v_t^N-u_t\|_{k^*}\right\|_{L^2(\PP)}$.
The only difference with the proof of Proposition~\ref{prop:mild-formulation} is that the drift of the particles is now
$K*\mu_t^N$ instead of $K*v_t^N$. We write
\[ R_t^N\coloneqq  K*(\mu_t^N-v_t^N)=(K-K*V^N)*\mu_t^N.\]
Therefore, applying It\^o's formula exactly as in the proof of Proposition~\ref{prop:mild-formulation} gives the same mild
formulation, with one additional term, namely
\begin{align*}
u_t^N(x) &= \hat{P}_t^*u_0^N(x)
+ \int_0^t \hat{P}_{t-s}^*\nabla \cdot
 \left\langle \mu_s^N, V^N(x-\cdot)\, K\ast 
u_s^N(\cdot) \right\rangle \,\dd s  \\
&\quad
+ \int_0^t \hat{P}_{t-s}^*\nabla \cdot
\langle \mu_s^N, V^N(x-\cdot)\,\left(F(\cdot)-
F(x) \right) \rangle \,\dd s \\
&\quad
 + \mathcal R_t^N(x) + \widetilde{M}_{t}^N(x),
\end{align*}
where
\begin{equation*}
    \mathcal R_t^N(x)
    \coloneqq 
    \int_0^t
    \widehat P_{t-s}^*
    \nabla\cdot
    \left\langle
        \mu_s^N,
        V^N(x-\cdot)R_s^N(\cdot)
    \right\rangle \dd s,
\end{equation*}
and
\begin{equation*}
\widetilde{M}_{t}^N(x) = - \frac{\sqrt{2}}{N}\sum_{i=1}^N \int_0^t 
 \hat{P}_{t-s}^* \nabla
V^N(\cdot - Y_s^{i,N})(x) \cdot \dd W_s^i .
\end{equation*}
We will derive rates successively for $\big\|\|\mathcal R_t^N\|_{k^*}\big\|_{L^2(\PP)}$ and $\big\lVert \lVert \widetilde{M}_{t}^N \rVert_{k^*}\big\rVert_{L^2(\PP)}$.

Using the duality estimate introduced in Step~2 of the proof of 
Proposition~\ref{prop:mild-formulation}, namely that
\(\|\widehat P_{t-s}^*\nabla\cdot h\|_{k^*}
    \leq
C(t-s)^{-\frac12-\frac{d}{2k}}\|h\|_{L^1}\),
we have
\begin{align*}
    \|\mathcal R_t^N\|_{k^*}
    &\leq
    C \int_0^t
    (t-s)^{-\frac12-\frac{d}{2k}}
    \left\langle
        \mu_s^N,
        |R_s^N|
    \right\rangle \dd s
    \\
    &\leq
    C \int_0^t
    (t-s)^{-\frac12-\frac{d}{2k}}
    \frac1{N^2}
    \sum_{i,j=1}^N
    |(K-K\ast V^N)(Y_s^{i,N}-Y_s^{j,N})|\, \dd s.
\end{align*}
Hence, using Jensen's inequality and~\eqref{eq:heat-like-estimate}, and
separating the diagonal and off-diagonal contributions as in Step~1, we obtain that
for any $t\in(0,1]$,
\begin{align*}
\big\|\|\mathcal R_t^N\|_{k^*}\big\|_{L^2(\PP)}
&\leq
C
\int_0^t
(t-s)^{-\frac12-\frac{d}{2k}}
\Big(\frac1{N^2}
 \sum_{i,j=1}^N
 \E\left[
  |(K-K*V^N)(Y_s^{i,N}-Y_s^{j,N})|^2 \right]\Big)^{1/2}
\,\dd s
\\
&\leq
C\|K-K*V^N\|_{L^k}
\int_0^t
(t-s)^{-\frac12-\frac{d}{2k}}
s^{-\frac{d}{2k}}
\,\dd s
 +C\|K\|_{L^k}
N^{-\frac12+\frac{d\alpha}{k}}
\int_0^t
(t-s)^{-\frac12-\frac{d}{2k}}
\,\dd s
\\
&\leq C \big( t^{\frac{1}{2}-\frac{d}{k}}
\|K-K*V^N\|_{L^k}
+\|K\|_{L^k}
N^{-\frac12+\frac{d\alpha}{k}}\big).
\end{align*}

Now for $\big\lVert \lVert \widetilde{M}_{t}^N \rVert_{k^*}\big\rVert_{L^2(\PP)}$ we proceed as in Proposition~\ref{lemma:stochastic-integral}. The only difference with $M_{t}^N$ is that the particles are the $X^{i,N}$ there, and $Y^{i,N}$ in $\widetilde{M}^N$. In the proof of Proposition~\ref{lemma:stochastic-integral}, the only estimate on the particle dynamics which is used is the uniform moment bound of 
Lemma~\ref{lem:uniform_bounds}. 
We can recover such a bound by the same partial Girsanov argument as above, this time decoupling only one particle.
Let us detail this argument here. By relabelling the coordinates, it is enough to consider the first particle $i=1$.
We apply Lemma~\ref{lem:martingale-bounds} with $r=1$. Under the
corresponding reference probability measure $\hat{\PP}$, the first particle solves
\[\dd\widehat Y_t^{1,N}
    = -F(\widehat Y_t^{1,N})\dd t+\sqrt2 \dd W_t^1\]
and is independent of the other particles. Moreover, for every $p>1$,
Lemma \ref{lem:martingale-bounds} gives
\begin{equation}\label{eq-1-moments}
    \sup_{N\in\mathbb N^*}
    \sup_{x\in(\mathbb R^d)^N}
    \sup_{t\in[0,1]}
    \|Z_t^{1,N}\|_{L^p(\hat{\PP})}
    \leq C_p.
\end{equation}
Motivated by Assumption~\ref{assum:initial}, fix 
\( m = 2 \vee (d+1)\left(1-\frac{2}{k}\right)>0\)
and
$\widehat Y_0^N=x=(x_1,\ldots,x_N)\in(\mathbb R^d)^N$.
By Lemma~\ref{lem:martingale-bounds}, under the probability measure
$\Q^{1,N}$ defined by the Girsanov density $Z^{1,N}$, the process
$\widehat Y^N$ has the same law as the original particle 
system~\eqref{non-mollified-particle-system} starting
from $x$. Hence, for any $t\in(0,1]$ and $p>1$ such that $mp^*\geq1$,
\begin{align}\nonumber
    \left(
        \mathbb E_{\Q^{1,N}}|\widehat Y_t^{1,N}|^m
    \right)^{1/m}
    &=\left(
        \mathbb E_{\hat{\PP}}
        \left[
            Z_t^{1,N}|\widehat Y_t^{1,N}|^m
        \right]
    \right)^{1/m}
    \\ \label{first-bound-one-part}
    &\leq
    \|Z_t^{1,N}\|_{L^p(\hat{\PP})}^{1/m}
    \left(
        \mathbb E_{\hat{\PP}}
        |\widehat Y_t^{1,N}|^{mp^*}
    \right)^{1/(mp^*)}.
\end{align}
Since $\widehat Y^{1,N}$ is a generalized Ornstein-Uhlenbeck process
starting from $x_1$, the contractivity assumption on $F$ allows easily to obtain
that there exists a constant $C>0$ such that
for any $t\geq0$,
\begin{equation}
\label{eq:ou-contr}
\left(
 \mathbb E_{\hat{\PP}}
  |\widehat Y_t^{1,N}|^{mp^*}
    \right)^{1/(mp^*)}
    \leq
    e^{-\lambda t}|x_1|+C.
\end{equation} 
Moreover, since $Z^{1,N}$ is a positive martingale,
\(\|Z_1^{1,N}\|_{L^1(\hat{\PP})}=1\) and
\eqref{eq-1-moments} holds for every $p>1$. Fixing $p_0>1$,
for $1<p<p_0$, interpolation between $L^1$ and $L^{p_0}$ gives, for \(\frac1p=\theta+\frac{1-\theta}{p_0},\)
\[ \|Z_1^{1,N}\|_{L^p(\hat{\PP})}
    \leq
    \|Z_1^{1,N}\|_{L^1(\hat{\PP})}^\theta
    \|Z_1^{1,N}\|_{L^{p_0}(\hat{\PP})}^{1-\theta}
    \leq
    C_{p_0}^{1-\theta}.\]
Since $\theta\to1$ as $p\to 1$, we obtain
\begin{equation}\label{limit}
    \sup_{N\in\mathbb N^*}
    \sup_{x\in(\mathbb R^d)^N}
    \|Z_1^{1,N}\|_{L^p(\hat{\PP})}
    \longrightarrow1
    \qquad\text{as }p\to 1.
\end{equation}
Since $e^{-\lambda}<1$ in \eqref{eq:ou-contr} and using \eqref{limit},
 we may
therefore choose $p>1$ sufficiently close to $1$ so that $mp^*\geq1$ and
\[ \rho\coloneqq  e^{-\lambda}
    \sup_{N\in\mathbb N^*}
    \sup_{x\in(\mathbb R^d)^N}
    \|Z_1^{1,N}\|_{L^p(\hat{\PP})}^{1/m}
    <1.\]
Returning to \eqref{first-bound-one-part}, 
we thus obtain that
for the original particle system 
\eqref{non-mollified-particle-system} started from the
deterministic configuration $x$, 
\[\left( \mathbb E \left[ |Y_1^{1,N}|^m
\,\mid \,
Y_0^N=x\right]\right)^{1/m} \leq \rho |x_1|+C.\]
Therefore, by conditioning with respect to $Y_0^N$, since $m\geq1$,
\begin{align*}
    \left(
        \mathbb E|Y_1^{1,N}|^m
    \right)^{1/m}
    &\leq
    \left(
        \mathbb E
        \left[
            \big(
                \rho|Y_0^{1,N}|+C
            \big)^m
        \right]
    \right)^{1/m}
    \leq
    \rho \left(\mathbb E|Y_0^{1,N}|^m \right)^{1/m} +C,
\end{align*}
where we used Minkowski's inequality in the last step.
Since the constants are uniform with respect to the initial
configuration, the Markov property gives the same estimates on each
interval of length one. By relabelling the coordinates, the same
estimates hold for every $i\in\{1,\ldots,N\}$. As $\rho<1$, iteration
together with Assumption~\ref{assum:initial} yields a uniform bound at
integer times. Moreover, using~\eqref{first-bound-one-part} for
$s\in (0,1)$, without requiring a contraction, we obtain
\[  \mathbb E
    \left[
        |Y_{n+s}^{i,N}|^m
        \,\mid\,
        Y_n^N
    \right]
    \leq
    C\left(1+|Y_n^{i,N}|^m\right).\]
Hence,
\[\sup_{N\in\mathbb N^*}
    \max_{1\leq i\leq N}
    \sup_{t\geq0}
    \mathbb E|Y_t^{i,N}|^m
    <\infty.\]
Consequently, the uniform moment estimate needed in the proof of Proposition
\ref{lemma:stochastic-integral} remains valid for the non-mollified
particle system. 
Indeed, the partial Girsanov argument provides moment estimates
for the original particle system under $\PP$, by equality in law
with the auxiliary system under the transformed measure.
The stochastic integral can therefore be estimated directly
under $\PP$, using the original Brownian motions and the same
BDG argument as in
Proposition \ref{lemma:stochastic-integral}.
Hence, repeating
the proof of Proposition 
\ref{prop:mild-formulation},
 we obtain that for any
$\varepsilon>0$ and any $t\in(0,1]$,
\begin{align} \nonumber
    \left\|\|v_t^N-u_t\|_{k^*}\right\|_{L^2(\PP)}
    &\leq
    C\Theta(t)
    W_2^N(\mathcal L Y_0^N,u_0^{\otimes N})
    \\ \label{bound-2}
    &\quad +C\big(\Theta(t) N^{-\min\left(
            \alpha,
            \frac12-\frac{d\alpha}{2}-\varepsilon
        \right)}+ t^{\frac{1}{2}-\frac{d}{k}} \|K-K*V^N\|_{L^k}
    + \|K\|_{L^k}N^{-\frac12+\frac{d\alpha}{k}}\big).
\end{align}

\noindent\textbf{Conclusion.} Now plug \eqref{bound-2} and \eqref{eq:first-term-final} into 
\eqref{two-side-decomposition} and, similarly to Section~\ref{sec:main-res},
iterate the resulting estimate. Since $N^{-\frac12+\frac{d\alpha}{k}} \leq
C N^{-\min\left(
            \alpha,
            \frac12-\frac{d\alpha}{2}-\varepsilon
        \right)}$,
and taking $\|K\|_{L^k}$ sufficiently small,
there exist constants $C,c>0$ such that for any $t>0$,
\begin{align*}
    \mathbb W_2^N(\mathcal L_{Y_t^N},u_t^{\otimes N})
    &\leq
    C\left(
    e^{-ct}
    \mathbb W_2^N(\mathcal L_{Y_0^N},u_0^{\otimes N})
    +
     N^{-\min\left(
            \alpha,
            \frac12-\frac{d\alpha}{2}-\varepsilon
        \right)}+
    \|K-K*V^N\|_{L^k}
    \right).
\end{align*}
It remains to prove the estimate on the mollified empirical density.
We proceed as in the end of Section~\ref{sec:main-res}. For $t\in(0,1]$, using~\eqref{bound-2}, the definition of $\Theta$ and the fact that $k>d$,
we obtain
\begin{align*}
    \left\|\|v_t^N-u_t\|_{k^*}\right\|_{L^2(\PP)}
    &\leq
    C t^{-\frac12-\frac{d}{2k}}
    \left(
    \mathbb W_2^N\left(\mathcal L_{Y_0^N},u_0^{\otimes N}\right)
    +
    N^{-\min\left(\alpha,\frac12-\frac{d\alpha}{2}
    -\varepsilon\right)}
    +
    \|K-K*V^N\|_{L^k}
    \right).
\end{align*}
For $t>1$, we apply \eqref{bound-2} on the time interval
$[t-1,t]$ and use the uniform-in-time Wasserstein estimate obtained
above. Arguing exactly as in the end of Section~\ref{sec:main-res}, we conclude that
for any $t>0$,
\begin{align*}
    \left\|\|v_t^N-u_t\|_{k^*}\right\|_{L^2(\PP)}
    &\leq
    C(1\wedge t)^{-\frac12-\frac{d}{2k}}
    \big( e^{-ct}
    \mathbb W_2^N(\mathcal L_{Y_0^N},u_0^{\otimes N})
    +N^{-\min\left(\alpha,\frac12-\frac{d\alpha}{2}
    -\varepsilon\right)}
    +\|K-K*V^N\|_{L^k}
    \big).
\end{align*}

\section{Proof of Corollary~\ref{cor:pathwise-poc}}
\label{sec:corollary}

The proof of Corollary~\ref{cor:pathwise-poc} is an almost immediate consequence of the proof of Theorem~\ref{thm:main} and of Theorem~\ref{thm:main-non-mollified}. The only difference is that we do not consider an optimal coupling for the initial conditions, but  directly proceed to bound $\big(\frac1N\sum_{i=1}^N
\E\big[|X_t^{i,N}-X_t^i|^2\big]\big)^{1/2}$.

\begin{proof}[Proof of Corollary~\ref{cor:pathwise-poc}]
We first consider the mollified particle system.
Returning to \eqref{eq:W2-coupling-1}, combining it with estimate
\eqref{L2-diff-2} which, since we are not considering an optimal coupling for the initial conditions, now reads
\begin{align*}
\left\|\|u_t^N-u_t\|_{k^*}
\right\|_{L^2(\PP)}
\leq C_{2}\Theta(t) \bigg(
\Big(\frac1N\sum_{i=1}^N\E\left[|X_0^{i,N}-X_0^i|^2\right]\Big)^{1/2}
+N^{-\min\big( \alpha,
    \frac12-\frac{d\alpha}{2}-\varepsilon\big)}\bigg).
\end{align*}
Thus we obtain for any $t\in(0,1]$,
\begin{align*}
&\Big(\frac1N\sum_{i=1}^N
\E\big[|X_t^{i,N}-X_t^i|^2\big]\Big)^{1/2}\\
&\leq e^{\int_0^t (C\|K\|_{L^k}s^{-\frac12-\frac{d}{2k}}-\lambda)\dd s}
\Big(\frac1N\sum_{i=1}^N\E\big[|X_0^{i,N}-X_0^i|^2\big]\Big)^{1/2}\\
&\quad+\|K\|_{L^k}\int_0^t
e^{\int_s^t (C\|K\|_{L^k}r^{-\frac12-\frac{d}{2k}}-\lambda)\dd r}
\left\|\|u_s^N-u_s\|_{k^*}\right\|_{L^2(\PP)}\,\dd s\\
& \leq \Big(\frac1N\sum_{i=1}^N\E\big[|X_0^{i,N}-X_0^i|^2\big]\Big)^{1/2} \Big(e^{\int_0^t
(C_{1}\|K\|_{L^k}s^{-\frac12-\frac{d}{2k}}-\lambda)\dd s} +C_{2}\|K\|_{L^k} e^{C_{1}\|K\|_{L^k}}\Big)
 + C_{2} N^{-\min\big(\alpha, \frac12-\frac{d\alpha}{2}-\varepsilon\big)}.
\end{align*}     
Therefore, arguing exactly as in the proof of
Theorem~\ref{thm:main}, at time $t=1$ we obtain
\begin{align*}
\Big(
\frac1N\sum_{i=1}^N
\E\big[|X_1^{i,N}-X_1^i|^2\big]
\Big)^{1/2}
&\leq
\mathcal{A}
\Big(\frac1N\sum_{i=1}^N
\E\big[|X_0^{i,N}-X_0^i|^2\big]\Big)^{1/2}
+
C N^{-\min\left(
\alpha,\frac12-\frac{d\alpha}{2}-\varepsilon
\right)},
\end{align*}
where $\mathcal{A}<1$ for $\|K\|_{L^k}$ sufficiently small.
We then iterate this estimate on intervals of length $1$ and conclude, 
as in the proof of Theorem~\ref{thm:main}, that there exist constants $C,c>0$
such that for any $t>0$,
\begin{align*}
\Big(\frac1N\sum_{i=1}^N
\E\big[|X_t^{i,N}-X_t^i|^2\big]\Big)^{1/2}
&\leq
C e^{-ct}
\Big(\frac1N\sum_{i=1}^N \E\big[|X_0^{i,N}-X_0^i|^2\big]\Big)^{1/2}
+
C N^{-\min\left(
\alpha,\frac12-\frac{d\alpha}{2}-\varepsilon
\right)}.
\end{align*}

For the non-mollified particle system, the proof is identical,
using the estimates obtained in the proof of
Theorem~\ref{thm:main-non-mollified}. The corresponding local
estimate contains the additional error
$\|K-K*V^N\|_{L^k}$. Iterating on intervals of length $1$ therefore
gives the same conclusions with
$\|K-K*V^N\|_{L^k}$ added to the right-hand side.
\end{proof}

\appendix
\section{Proof of Proposition~\ref{prop:ex-uniq}}\label{app:ex-uniq}
\begin{proof}
\textbf{Existence.}
     Assumptions \ref{assum:potentials}
     imply assumption ($A_1$) in \cite{HRW25} with $p=\infty$. 
    \cite[Theorem 2.1]{HRW25} provides, for any initial law 
     $\mathcal{L}_{X_0} \in \mathcal{P}(\R^d)$, a unique
     global strong solution to \eqref{mckean-vlasov} 
    in a localized class (see \cite[Definition~1.2]{HRW25}) which verifies 
\begin{equation*}
        \sup_{t \in (0,T]} t^{\frac{d}{2k}}\|u_t\tilde\|_{{k}^*} < \infty, \quad
         \forall T>0.
\end{equation*}
The goal is to show that existence and uniqueness hold in 
the larger class $\mathcal{C}_T$ introduced in Definition~\ref{def-mckean}.
Since $\|u_t\|_{k^*} \leq \|u_t\tilde\|_{{k}^*}$, we have
\begin{equation*}
        \sup_{t \in (0,T]} t^{\frac{d}{2k}}\|u_t\|_{k^*} < \infty, 
        \quad \forall T>0.
\end{equation*}

\textbf{Uniqueness.}
     Now, let us prove uniqueness in the sense of Definition~\ref{def-mckean},
      by proving first uniqueness in law, and then pathwise uniqueness of~\eqref{mckean-vlasov}.
    Using \cite[Theorem 2.1]{HW25} and \cite[(5.33)]{HRW25}, 
   we know that for any $u_0 \in \mathcal{P}(\R^d)$, the PDE~\eqref{nlfp-intro} admits a solution 
   verifying the Duhamel formula~\eqref{eq:duhamel} and such that it is 
   in \(\{ \mu \in \mathcal{C}{([0,T];\mathcal{P}(\R^d))} : 
    \sup_{t \in (0,T]} t^{\frac{d}{2k}}\|\mu_t\tilde\|_{{k}^*} < 
    \infty \} \subset \mathcal{C}_{T}\), where we recall that $\mathcal{C}_{T}$ was defined in~\eqref{def:C_T}. Assume that there are two solutions $u, v$ of \eqref{nlfp-intro}  in $\mathcal{C}_{T}$
  with the same initial condition $u_0 \in \mathcal{P}(\R^d)$
    and verifying the Duhamel formula \eqref{eq:duhamel-u}. Then, using 
    Lemma~\ref{lem:OU-bounds} and Hölder's inequality, 
    there exists $C>0$ such that for any $t \in (0,T]$, we have
    \begin{align*}
        \|u_t - v_t\|_{k^*} &\leq
        \|\hat{P}_t^*u_0 - \hat{P}_t^*v_0\|_{k^*}
        + \left\|\int_0^t 
    \hat{P}_{t-s}^* (\nabla \left(u_s (K * u_s)\right)
    - \nabla \left(v_s (K * v_s)\right)) \dd s\right\|_{k^*} \\
    &= \left\|\int_0^t 
    \hat{P}_{t-s}^* (\nabla \left(u_s (K * u_s)\right)
    - \nabla \left(v_s (K * v_s)\right)) \dd s\right\|_{k^*} \\
    &\leq C \int_0^t \|\nabla \hat{P}_{t-s} \|_{L^{k} \to L^{\infty}}
    \|u_s (K * u_s) - v_s (K * v_s)\|_{L^1} \dd s\\
    &\leq C \int_0^t (t-s)^{-\frac{1}{2}-\frac{d}{2k}} \left( \|K\|_{L^k} \|u_s\|_{L^1}
    \|u_s-v_s\|_{k^*} + \|v_s\|_{L^1} \|K\|_{L^k}
     \|u_s-v_s\|_{k^*} \right) \dd s\\
    &\leq C \int_0^t (t-s)^{-\frac{1}{2}-\frac{d}{2k}} 
    \|u_s-v_s\|_{k^*} \dd s .
    \end{align*}
    Since $\int_{0}^{t} (t-s)^{-\frac{1}{2}-\frac{d}{2k}} \dd s < \infty$ for $k > d$,
    Gr\"onwall's lemma allows to conclude that for any 
    $T>0$, $u =v$ in \(\mathcal{C}_{T}\). 
     Now let $(X_t)_{t \geq 0}$ and
      $(Y_t)_{t \geq 0}$ be two global strong solutions of 
      \eqref{mckean-vlasov} in the sense of Definition~\ref{def-mckean} such that $X_0 = Y_0$ a.s. and 
       defined on the same probability space $(\Omega, \mathcal{F}, \mathbb{P})$
       with the same Brownian motion $(W_t)_{t \geq 0}$. By
       uniqueness in law, we have $\mathcal{L}_{X_t} = u_t 
       = \mathcal{L}_{Y_t}$ for any $t \geq 0$. Moreover, 
       using the uniqueness result of \cite[Theorem~2.1]{HW25}, 
       we have $u  \in  \mathcal{C}_{T}.$
       Then, Itô's lemma together with Proposition~\ref{prop:lip-bound} and the 
       Lipschitz property of $F$ gives that for any $t \in (0,1]$,
   \begin{align*}
\E\big[|X_t-Y_t|^2\big]
&=-2\int_0^t \E\Big[
(X_s-Y_s)\cdot\big(F(X_s)-F(Y_s)\big)
\Big]\dd s\\
&\quad -2\int_0^t \E\Big[(X_s-Y_s)\cdot
\big((K\ast u_s)(X_s)-(K\ast u_s)(Y_s)\big)
\Big]\dd s\\
&\leq
2\int_0^t\left(
C_{\mathrm{Lip}(F)}
+C_{\mathrm{Lip}(K\ast u_s)}
\right)
\E\big[|X_s-Y_s|^2\big]\dd s\\
&=2\int_0^t
\left(C_{\mathrm{Lip}(F)}
+C s^{-\frac12-\frac{d}{2k}}
\right)
\E\big[|X_s-Y_s|^2\big]\dd s.
\end{align*}
Since $k>d$,
\(\int_0^t s^{-\frac12-\frac{d}{2k}}\,\dd s<\infty\).
Therefore, Gr\"onwall's lemma together with the fact that $\E[|X_0-Y_0|^2]=0$ gives
\(\E\big[|X_t-Y_t|^2\big]=0\) for all $t \in (0,1]$.
Consequently, 
\[X_t=Y_t,
\qquad \text{a.s.},\]
for any $t\in[0,1]$.
We can then obtain pathwise uniqueness for all $t>0$ by iterating the
argument on $(n,n+1]$, $n\in\mathbb{N}$, and considering the SDE starting
at time $n$ with the common initial condition $X_n=Y_n$.
\end{proof}

\section{Proofs of Lemma~\ref{some-bounds} and Lemma~\ref{W1-bound}}\label{app:lemmas}

We first proceed with the proof of Lemma~\ref{some-bounds}.

\begin{proof}[Proof of Lemma~\ref{some-bounds}]
Using Lemma~\ref{lem:OU-bounds}\ref{item:heat-bounds} in the first inequality below and Jensen's inequality for the last one, there exists $C, \kappa > 0$ such that for any $0\leq s<t\leq1$,
\begin{align*}
   &\left\| (1+|\cdot|)^{\frac{d+1}{q}} \hat{P}^*_{t-s} 
        \big(\nabla V^N(\cdot-a_s)\big)\right\|_{L^2}^2  =
        \int_{\R^d}  (1+|y|)^{2\frac{d+1}{q}}
            \left|\int_{\R^d} \hat{p}_{t-s}(x, y)
 \nabla V^N(x-a_s)\,\dd x\right|^2 \dd y\\
        &\quad\leq C  \int_{\R^d} 
             (1+|y|)^{2\frac{d+1}{q}}
            \left|\left(\int_{\R^d} g^F_\kappa(t-s, x, y)\, \dd x \right)
        \int_{\R^d} \frac{g^F_\kappa(t-s, x, y)}{\int_{\R^d}
         g^F_\kappa(t-s, x, y)\, \dd x}
     \nabla V^N(x-a_s)\, \dd x\right|^2 \dd y
           \\
     &\quad\leq C \int_{\R^d}  (1+|y|)^{2\frac{d+1}{q}}
            \left(\int_{\R^d} g^F_\kappa(t-s, x, y)\, \dd x \right)
    \int_{\R^d} g^F_\kappa(t-s, x, y)
    |\nabla V^N(x-a_s)|^2\, \dd x\,\dd y.
\end{align*}
Lemma~\ref{lem:OU-bounds}\ref{item:sgG-Lp} implies that \(\int_{\R^d} g^F_\kappa(t-s, x, y)\, \dd x  \leq 
\sup_{y \in \R^d}  |G_{t-s}^{\kappa,F,*} \mathbbm{1}(y)| \leq 
\|G_{t-s}^{\kappa,F,*}\|_{L^\infty \to L^\infty} \leq C\), for some $C<\infty$ independent of $t-s$. This yields 
\begin{align*}
    \Big\| (1+|\cdot|)^{\frac{d+1}{q}} \hat{P}^*_{t-s} 
        \big(\nabla V^N(\cdot-a_s)\big)\Big\|_{L^2}^2 
      \leq  C \int_{\R^d}  (1+|y|)^{2\frac{d+1}{q}}
    \int_{\R^d} g^F_\kappa(t-s, x, y)
    |\nabla V^N(x-a_s)|^2\, \dd x\,\dd y ,
\end{align*}
and in view of~\eqref{eq:gaussian-kernel}, one can perform a change of variables which, for $n_\kappa(t-s, x)$ the normal distribution with mean $0$ and covariance matrix $\kappa (t-s) I_d$, reads
\begin{align*}
    \Big\| (1+|\cdot|)^{\frac{d+1}{q}}& \hat{P}^*_{t-s} 
        \big(\nabla V^N(\cdot-a_s)\big)\Big\|_{L^2}^2  \\
    &\leq C \int_{\R^d} |\nabla V^N(x-a_s)|^2 
    \int_{\R^d} (1+|y-\theta_{t-s}(x)|)^{2\frac{d+1}{q}}
    n_\kappa(t-s, y)\, \dd y\,\dd x.
\end{align*}
Applying the inequality $(1+|y-
\theta_{t-s}(x)|)^{2\frac{d+1}{q}} 
\leq (1+|y|)^{2\frac{d+1}{q}}(1+
|\theta_{t-s}(x)|)^{2\frac{d+1}{q}}$ gives
\begin{align*}
       \Big\| (1+|\cdot|)^{\frac{d+1}{q}}& \hat{P}^*_{t-s} 
        \big(\nabla V^N(\cdot-a_s)\big)\Big\|_{L^2}^2 
      \\
      &\leq C \left(\int_{\R^d} |\nabla V^N(x-a_s)|^2 
      (1+|\theta_{t-s}(x)|)^{2\frac{d+1}{q}} \dd x\right)
    \left( \int_{\R^d} (1+|y|)^{2\frac{d+1}{q}} 
    n_\kappa(t-s, y)\, \dd y\right).
\end{align*}
Integrating the ODE for $\theta_t$ (see \eqref{eq:gaussian-kernel}), using the Lipschitz continuity of $F$ and applying 
Gr\"onwall's inequality, there is
\begin{equation}\label{eq:theta_bound}
|\theta_{t-s}(y)| \leq C (1+|y|), \quad 0\leq s< t\leq 1.
\end{equation}
Moreover, using
\[
1+|x|\leq (1+|a_s|)(1+|x-a_s|),
\]
and making the change of variables $z=x-a_s$, we obtain
\begin{align*}
\int_{\R^d} |\nabla V^N(x-a_s)|^2 
      (1+|\theta_{t-s}(x)|)^{2\frac{d+1}{q}} \dd x
\leq C(1+|a_s|)^{2\frac{d+1}{q}}
\int_{\R^d}|\nabla V^N(z)|^2
(1+|z|)^{2\frac{d+1}{q}}\dd z.
\end{align*}
Therefore, using the compact support of $V$ and the scaling of $V^N$,
\begin{equation}\label{eq:bound1}
 \left\| (1+|\cdot|)^{\frac{d+1}{q}} \hat{P}^*_{t-s} 
        \big(\nabla V^N(\cdot-a_s)\big)\right\|_{L^2}^2 
 \leq C(1+|a_s|)^{2\frac{d+1}{q}}N^{d \alpha + 2 \alpha}.
\end{equation}
Alternatively, 
using Lemma~\ref{lem:OU-bounds}\ref{item:heat-bounds} combined with Jensen's inequality and the change of variables derived from~\eqref{eq:gaussian-kernel} with
\eqref{eq:theta_bound}, it follows that
\begin{align*} 
    \Big\|& (1+|\cdot|)^{\frac{d+1}{q}} \hat{P}^*_{t-s} 
        \big(\nabla V^N(\cdot-a_s)\big)\Big\|_{L^2}^2 =
        \int_{\R^d}  (1+|x|)^{2\frac{d+1}{q}}
            \left|\int_{\R^d} \hat{p}_{t-s}(y, x)
 \nabla V^N(y-a_s)\,\dd y\right|^2 \dd x \\ 
 &= \int_{\R^d}  (1+|x|)^{2\frac{d+1}{q}}
            \left|\int_{\R^d} \nabla_y\hat{p}_{t-s}(y, x)
 V^N(y-a_s)\,\dd y\right|^2 \dd x \\ 
        &\leq C (t-s)^{-1}\int_{\R^d} 
             (1+|x|)^{2\frac{d+1}{q}}
            \left|\left(\int_{\R^d} g^F_\kappa(t-s, x, y) \dd y\right)
        \int_{\R^d} \frac{g^F_\kappa(t-s, x, y)}{\int_{\R^d}
         g^F_\kappa(t-s, x, y) \dd y}
    V^N(y-a_s) \dd y\right|^2 \dd x
           \\ 
     &\leq C (t-s)^{-1} \int_{\R^d}  (1+|x|)^{2\frac{d+1}{q}}
            \left(\int_{\R^d} g^F_\kappa(t-s, x, y) \dd y\right)
    \int_{\R^d} g^F_\kappa(t-s, x, y)
    |V^N(y-a_s)|^2 \dd y\,\dd x\\ 
    & \leq C (t-s)^{-1} \left(\int_{\R^d} |V^N(y-a_s)|^2 
    (1+|\theta_{t-s}(y)|)^{2\frac{d+1}{q}} \dd y\right)
    \left( \int_{\R^d} (1+|x|)^{2\frac{d+1}{q}}
     n_\kappa(t-s, x) \dd x \right).
\end{align*}
Using again \eqref{eq:theta_bound},
\(1+|y|\leq (1+|a_s|)(1+|y-a_s|),\)
and making the change of variables $z=y-a_s$, we obtain
\begin{align*}
\int_{\R^d} |V^N(y-a_s)|^2 
    (1+|\theta_{t-s}(y)|)^{2\frac{d+1}{q}} \dd y
    \leq C(1+|a_s|)^{2\frac{d+1}{q}}
\int_{\R^d}|V^N(z)|^2
(1+|z|)^{2\frac{d+1}{q}}\dd z.
\end{align*}
Therefore, using the compact support of $V$ and the scaling of $V^N$ so that $\| V^N\|_{L^2} \lesssim N^{\frac{d \alpha}{2}}$, it comes
\begin{equation}\label{eq:bound2}
    \left\| (1+|\cdot|)^{\frac{d+1}{q}} \hat{P}^*_{t-s} 
        \big(\nabla V^N(\cdot-a_s)\big)\right\|_{L^2}^2
    \leq C (t-s)^{-1}(1+|a_s|)^{2\frac{d+1}{q}}N^{d \alpha }.
\end{equation}
Let $\varepsilon\in (0,4\alpha)$.
Interpolating between~\eqref{eq:bound1} and~\eqref{eq:bound2},
there exists a constant $C>0$ such that for any $t \in (0,1]$,
\begin{align*}
 \int_0^t& \left\| (1+|\cdot|)^{\frac{d+1}{q}} \hat{P}^*_{t-s}
        \big(\nabla V^N(\cdot-a_s)\big)\right\|_{L^2}^2  \dd s  
        \\
&= \int_0^t \left\| (1+|\cdot|)^{\frac{d+1}{q}} \hat{P}^*_{t-s}
        \big(\nabla V^N(\cdot-a_s)\big)\right\|_{L^2}^{2\frac{\varepsilon}{4 \alpha}} 
         \left\| (1+|\cdot|)^{\frac{d+1}{q}} \hat{P}^*_{t-s}
        \big(\nabla V^N(\cdot-a_s)\big)\right\|_{L^2}^{2(1-\frac{\varepsilon}{4 \alpha})} \dd s\\
&\leq C \int_0^t 
\left((1+|a_s|)^{2\frac{d+1}{q}}
N^{d \alpha + 2 \alpha}\right)^{\frac{\varepsilon}{4 \alpha}} 
\left((t-s)^{-1}(1+|a_s|)^{2\frac{d+1}{q}}
N^{d \alpha }\right)^{1-\frac{\varepsilon}{4 \alpha}} \dd s\\
&= C N^{d \alpha + \frac{\varepsilon}{2}}
\int_0^t (t-s)^{-1+\frac{\varepsilon}{4 \alpha}}
(1+|a_s|)^{2\frac{d+1}{q}}\dd s,
\end{align*}
which concludes the proof.
\end{proof}

The following lemma is a technical computation 
concerning the Wasserstein distance between the initial 
empirical measure and the initial law of the particle system. 
One can see that
$\mathbb{W}_1(u_0^N,u_0)$ appears in the bounds obtained in Proposition 
\ref{prop:mild-formulation}. 
Precisely, in the proof of Theorem~\ref{thm:main}, 
we need to control the $L^2(\PP)$ norm of $\mathbb{W}_1(u_0^N,u_0)$. 

\begin{lemma}
\label{W1-bound}
Consider $u_0^N$ as defined in \eqref{eq:initial-empirical-measure}, and $u_0$ the initial 
law of the nonlinear Fokker-Planck equation~\eqref{nlfp-intro}. For any $\varepsilon>0$, there exists a constant $C>0$ such that, for any $N\geq 1$,
\begin{equation*}
\|\mathbb{W}_1(u_0^N,u_0)\|_{L^2(\PP)}
\leq \mathbb{W}^N_2 \bigl(\mathcal{L}_{X_0^N},u_0^{\otimes N}\bigr) +
C (N^{-\alpha} + f(N)),
\end{equation*}
where
\begin{equation}
\label{eq:defFN}
f(N) =
\begin{cases}
N^{-1/2+\varepsilon},
& d=1,\\[1mm] 
N^{-1/2}\log(1+N),
& d=2,\\[1mm]
N^{-1/d},
& d\geq3.
\end{cases}
\end{equation}
\end{lemma}

\begin{proof}
Consider $N$ i.i.d. random variables $(X_0^i)_{i=1,\ldots,N}$ 
with law $u_0$, such that 
\newline\((X_0^{1,N},\dots, X_0^{N,N}, X_0^1, \dots, X_0^N)\) 
realize the optimal coupling of $\mathbb{W}^N_2
(\mathcal{L}_{X_0^N}, u_0^{\otimes N})$.
 Using the triangle inequality and 
 the fact that $\mathbb{W}_1(V^N \ast \mu_0^N, \mu_0^N) \lesssim N^{-\alpha}$ (see 
 \cite[(2.15)]{ORT}), there exists a constant $C>0$ such that
\begin{align}\nonumber
\mathbb{W}_1(u_0^N, u_0) &\leq \mathbb{W}_1(V^N \ast \mu_0^N, \mu_0^N)
 + \mathbb{W}_1(\mu_0^N, \frac{1}{N}
\sum_{i=1}^N \delta_{X_0^i}) +
\mathbb{W}_1(\frac{1}{N}\sum_{i=1}^N \delta_{X_0^i}, u_0)  \\ \label{W1}
&\leq C N^{-\alpha} +
\mathbb{W}_1(\mu_0^N, \frac{1}{N}\sum_{i=1}^N \delta_{X_0^i}) + 
\mathbb{W}_1(\frac{1}{N}\sum_{i=1}^N \delta_{X_0^i}, u_0) .
\end{align}
First,
\begin{align}
\label{W1-1}
\left\|\mathbb{W}_1(\mu_0^N, 
\frac{1}{N}\sum_{i=1}^N \delta_{X_0^i})\right\|_{L^2(\PP)}
&= \left\| \sup_{\|f\|_{Lip} \leq 1} \left|\frac{1}{N} \sum_{i=1}^N (f(X_0^{i,N}) - 
f(X_0^i)) 
\right|\right\|_{L^2(\PP)}\nonumber \\
& \leq \mathbb{E} \left[ \left(\frac{1}{N}\sum_{i=1}^N |X_0^{i,N} - X_0^i| \right)^2 \right]^{1/2} \nonumber \\
&\leq \mathbb{W}^N_2(\mathcal{L}_{X_0^N}, u_0^{\otimes N}),
\end{align}
using Jensen's inequality in the final line.

Second, we need to bound the $L^2$ norm of \(
\mathbb{W}_1\left(\frac1N\sum_{i=1}^N \delta_{X_0^i},u_0\right)\).
We decompose it as follows:
\begin{align}\label{LLN-W2}
\left\|\mathbb{W}_1 \left(\frac1N\sum_{i=1}^N \delta_{X_0^i},u_0\right)\right\|_{L^2(\PP)}^2
= \E \left[\mathbb{W}_1 \left(\frac1N\sum_{i=1}^N \delta_{X_0^i},u_0\right)\right]^2 + 
\text{Var} \left[\mathbb{W}_1 \left(\frac1N\sum_{i=1}^N \delta_{X_0^i},u_0\right)\right].
\end{align}
Take $\varepsilon >0$.
For the square of the expectation in \eqref{LLN-W2}, since the considered measures have bounded 
second moments, we use \cite[Theorem 1]{FG15} with \(p=1\). More precisely, taking \(q=2-\varepsilon\) in
dimensions \(d\in \{1,2\}\), and \(q=2\) for \(d\ge3\), we obtain
\[
\E \left[\mathbb{W}_1 \left(\frac1N\sum_{i=1}^N \delta_{X_0^i},u_0\right)\right]
\le C f(N),
\]
where \(f(N)\) is defined in \eqref{eq:defFN}.
For the variance in \eqref{LLN-W2}, consider 
 \((Y_0^{i})_{1\leq i \leq N}\)  an independent copy of \((X_0^i)_{1\leq i \leq N}\), then
  use the Efron-Stein inequality (see
for example  
\cite[Theorem 5 p.213]{BLR04}) to obtain
\begin{align}\nonumber
\text{Var} \left[\mathbb{W}_1 \left(\frac1N\sum_{i=1}^N 
\delta_{X_0^i},u_0\right)\right]
&\le \frac{1}{2}\sum_{i=1}^N 
\E \left[\left(\mathbb{W}_1 \left(\frac1N\sum_{j=1}^N
 \delta_{X_0^j},u_0\right) -
\mathbb{W}_1 \left(\frac1N\sum_{j\neq i}^N
 \delta_{X_0^j} + \frac1N\delta_{Y_0^{i}},
 u_0\right)\right)^2\right]\\\nonumber
&\le \frac{1}{2} \sum_{i=1}^N \E \left[\mathbb{W}^2_1 
\left(\frac1N\sum_{j=1}^N \delta_{X_0^j},\frac1N\sum_{j\neq i}^N 
\delta_{X_0^j} + \frac1N\delta_{Y_0^{i}}\right)\right]\\\nonumber
& = \frac{1}{2} \sum_{i=1}^N \E \left[\left(\sup_{\|f\|_{Lip}
 \leq 1} \frac{1}{N} |f(X_0^i) - f(Y_0^{i})|\right)^2\right]\\\nonumber
&\le \frac{1}{2}\sum_{i=1}^N \E \left[\frac{1}{N^2}|X_0^i - Y_0^{i}|^2\right]
= \frac{1}{2N}\E \left[|X_0^1 - Y_0^{1}|^2\right]\\\label{W1-2}
&\leq C \frac{1}{N} \int_{\R^d} |x|^2 u_0(\dd x) \leq \frac{C}{N}.
\end{align}
Hence, combining the bounds on the expectation and the variance, we conclude that
\[
\left\|\mathbb{W}_1 \left(\frac1N\sum_{i=1}^N \delta_{X_0^i},u_0\right)\right\|_{L^2(\PP)}
\le C f(N),
\]
where \(f(N)\) is defined in \eqref{eq:defFN}.
The result follows by plugging the bounds from \eqref{W1-1} and \eqref{W1-2} into~\eqref{W1}.
\end{proof}
\begin{remark}\label{rem:const-1}
Notice that, under the assumptions of 
Theorem~\ref{thm:main},
\(\sup_{t\geq0}\int_{\R^d}|x|^2u_t(\dd x)<\infty\)
(see \cite[Theorem 3.1]{HW25}).
When Lemma~\ref{W1-bound} is applied on a shifted interval
$[n,n+1]$, the measure $u_n$ plays the role of the initial law.
Therefore, the above uniform moment bound ensures that the constant
$C$ in Lemma~\ref{W1-bound} does not depend on $n$. In the proof of 
Theorem \ref{thm:main},
this allows us
to iterate the argument on successive time intervals of length $1$.
\end{remark}

\section{Partial Girsanov estimates}\label{annex:part-girsanov}

We first give an analogue of \cite[Proposition 2.1]{T23} adapted to the generalized
Ornstein-Uhlenbeck process with drift $F$.

\begin{lemma}
\label{lem:exponential-OU}
Assume that Assumption~\ref{assum:potentials} holds
and that $k>(d \vee 2)$. Let
$(\Omega,\mathcal F,(\mathcal F_t)_{t\geq0}, \hat{\mathbb P})$ be a probability space 
on which a $d$-dimensional Brownian motion $W$ is defined. 
For $y\in\mathbb R^d$,
 let $\overline Y$ be the unique strong solution of
\[\dd\overline Y_t
=-F(\overline Y_t)\,\dd t
 +\sqrt2\,\dd W_t,
 \qquad
\overline Y_0=y.\]
Then for any $\alpha>0$, there exists $C_{\alpha}<\infty$ such that
\begin{equation}
\label{eq:exp-OU-K}
    \sup_{y\in\mathbb R^d}
    \sup_{w\in C([0,1];\mathbb R^d)}
    \mathbb E
    \left[
        \exp\left(
            \alpha\int_0^1
            |K(\overline Y_s-w_s)|^2\,\dd s
        \right)
    \right]
    \leq C_\alpha.
\end{equation}
\end{lemma}

\begin{proof}The proof follows the argument of \cite[Proposition 2.1]{T23}.
The only modification is that the Brownian motion appearing therein is replaced here
by the generalized Ornstein-Uhlenbeck process $\overline Y$.

Fix $w\in C([0,1];\mathbb R^d)$. For $0\leq t_1<t_2\leq1$,
the Markov property and Lemma~\ref{lem:OU-bounds} give
\begin{align*}
\mathbb E
\left[ \left.\int_{t_1}^{t_2}
    |K(\overline Y_s-w_s)|^2\,\dd s
    \,\right|\,
    \mathcal F_{t_1}\right]
&= \int_{t_1}^{t_2}
\int_{\mathbb R^d}
    |K(z-w_s)|^2
    \widehat p_{s-t_1}
    (\overline Y_{t_1},z)
    \,\dd z\,\dd s
\\
& \leq
\|K\|_{L^k}^2
\int_{t_1}^{t_2}
\left\| \widehat p_{s-t_1}
    (\overline Y_{t_1},\cdot)
\right\|_{L^{\frac{k}{k-2}}}
\,\dd s
\\
& \leq
C\|K\|_{L^k}^2
\int_{t_1}^{t_2}
(s-t_1)^{-\frac dk}\,\dd s
\\
& \leq
C\|K\|_{L^k}^2
(t_2-t_1)^{1-\frac dk}.
\end{align*}
Since the last estimate is uniform with respect to
$\overline Y_{t_1}$, successive conditioning in the iterated integrals
gives, for every $n\geq1$,
\begin{align*}
&\mathbb E
\left[ \left.
    \left(
        \int_{t_1}^{t_2}
        |K(\overline Y_s-w_s)|^2\,\dd s
    \right)^n
    \,\right|\,
    \mathcal F_{t_1}\right]
\\
& =
n!\,\mathbb E
\left[\left.
    \int_{t_1<s_1<\cdots<s_n<t_2}
    \prod_{\ell=1}^n
    |K(\overline Y_{s_\ell}-w_{s_\ell})|^2
    \,\dd s_1\cdots\dd s_n
    \,\right|\,
    \mathcal F_{t_1}\right]
\\
& \leq
n!\left( C\|K\|_{L^k}^2
    (t_2-t_1)^{1-\frac dk}\right)^n.
\end{align*}
Consequently, whenever $t_2-t_1$ is sufficiently small so that
\[\alpha C\|K\|_{L^k}^2
    (t_2-t_1)^{1-\frac dk}\leq\frac12,\]
we have
\begin{align*}
\mathbb E
\left[ \left.\exp\left(\alpha
        \int_{t_1}^{t_2}
        |K(\overline Y_s-w_s)|^2\,\dd s\right)
    \,\right|\,
    \mathcal F_{t_1}\right]
 \leq
\sum_{n=0}^{\infty}
\left(\alpha C\|K\|_{L^k}^2
    (t_2-t_1)^{1-\frac dk}\right)^n\leq2.
\end{align*}
Since $k>d$, the exponent $1-d/k$ is positive. We can therefore
subdivide $[0,1]$ into a finite number of intervals on which the
previous estimate applies. Successive conditioning over these
intervals then yields
\[ \sup_{y\in\mathbb R^d}
    \sup_{w\in C([0,1];\mathbb R^d)}
    \mathbb E
    \left[ \exp\left( \alpha\int_0^1
            |K(\overline Y_s-w_s)|^2\,\dd s \right)\right]
    \leq C_\alpha.\]
\end{proof}

We now introduce the partial Girsanov transform that is
 used in Section~\ref{sec:non-mollified}.
Fix $r\in\{1,\ldots,N\}$ and
$x=(x_1,\ldots,x_N)\in(\mathbb R^d)^N$.
Let
$(\Omega,\mathcal F,(\mathcal F_t)_{t\geq0},\hat{\mathbb P})$ be a complete
filtered probability space, consider $N$ independent
$d$-dimensional Brownian motions $(W^i)_{i=1,\ldots,N}$, and
let $\widehat Y^N$ be the solution of
\begin{equation}
\label{eq:auxiliary-particle-system}
\left\{
\begin{aligned}
    \dd\widehat Y_t^{i,N}
    &=
    -F(\widehat Y_t^{i,N})\,\dd t
    + \sqrt2\,\dd W_t^i,
    && i=1,\ldots,r,
    \\
    \dd\widehat Y_t^{i,N}
    &=
    -F(\widehat Y_t^{i,N})\,\dd t
    - \frac1N\sum_{j=r+1}^N
    K(\widehat Y_t^{i,N}-\widehat Y_t^{j,N})\,\dd t
    + \sqrt2\,\dd W_t^i,
    && i=r+1,\ldots,N,
\end{aligned}
\right.
\end{equation}
with deterministic initial condition
\( \widehat Y_0^N=x.\)
When $r=N$, the second subsystem is empty.
Note that the well-posedness ot this system
follows from the same well-posedness result as for
\eqref{non-mollified-particle-system}.
For $i\in\{1,\ldots,N\}$, set
\begin{equation}
\label{eq:def-B-girsanov}
    B_t^{i,r,N}
    \coloneqq  \begin{cases}
    \displaystyle
    \frac1N\sum_{j=1}^N
    K(\widehat Y_t^{i,N}-\widehat Y_t^{j,N}),
    & i=1,\ldots,r,
    \\[3mm] \displaystyle
    \frac1N\sum_{j=1}^r
    K(\widehat Y_t^{i,N}-\widehat Y_t^{j,N}),
    & i=r+1,\ldots,N.
    \end{cases}
\end{equation}
For $t\in[0,1]$, consider the exponential local martingale
\begin{equation}
\label{eq:girsanov-density}
\begin{aligned}
    Z_t^{r,N}
    &\coloneqq \exp\left(
        -\frac1{\sqrt2}
        \sum_{i=1}^N
        \int_0^t
        B_s^{i,r,N}\cdot\dd W_s^i
        -
        \frac14
        \sum_{i=1}^N
        \int_0^t
        |B_s^{i,r,N}|^2\,\dd s\right).
\end{aligned}
\end{equation}
The proof of the following lemma is partially
 similar to that of \cite[Proposition 3.1]{T23}, but 
adapted to our setting.

\begin{lemma}
\label{lem:martingale-bounds}
Suppose that Assumption~\ref{assum:potentials} holds,
that $k>(d\vee2)$ and that $K(0)=0$.
Then $(Z_t^{r,N})_{t\in[0,1]}$ is a true martingale
and for every fixed $r\in\mathbb N^*$ and every $p>1$, there 
exists a constant $C = C(r,p)>0$ such that
\begin{equation}
\label{eq:uniform-martingale-bound}
    \sup_{N\geq r}
    \sup_{x\in(\mathbb R^d)^N}
    \sup_{t\in[0,1]}
    \mathbb E_{\hat{\mathbb P}}
    \left[(Z_t^{r,N})^p \right]\leq C.
\end{equation}
Moreover, defining $\mathbb Q^{r,N}$ on $\mathcal F_1$ by
\[\dd\mathbb Q^{r,N}
    \coloneqq 
    Z_1^{r,N}\,\dd\hat{\mathbb P},\]
the process $\widehat Y^N$ has under $\mathbb Q^{r,N}$ the same law as
the solution of \eqref{non-mollified-particle-system} starting from $x$.
\end{lemma}

\begin{proof}
We first establish a uniform in $N$ and $x$ estimate for
\( \exp\left(\gamma
        \sum_{i=1}^N
        \int_0^t
        |B_s^{i,r,N}|^2\,\dd s \right).\)
Let $\gamma>0$.
By \eqref{eq:def-B-girsanov} and Jensen's
inequality, for $i\leq r$,
\[|B_s^{i,r,N}|^2
    \leq
    \frac1N\sum_{j=1}^N
    |K(\widehat Y_s^{i,N}-\widehat Y_s^{j,N})|^2,\]
whereas, for $i>r$,
\[|B_s^{i,r,N}|^2
    \leq
    \frac{r}{N^2}\sum_{j=1}^r
    |K(\widehat Y_s^{i,N}-\widehat Y_s^{j,N})|^2.\]
Consequently,
\begin{align*}
\sum_{i=1}^N\int_0^t|B_s^{i,r,N}|^2\,\dd s
&\leq
\frac1N\sum_{i=1}^r\sum_{j=1}^N
\int_0^t
|K(\widehat Y_s^{i,N}-\widehat Y_s^{j,N})|^2\,\dd s
\\
&\quad+
\frac r{N^2}\sum_{i=r+1}^N\sum_{j=1}^r
\int_0^t
|K(\widehat Y_s^{i,N}-\widehat Y_s^{j,N})|^2\,\dd s.
\end{align*}
Hence, by the Cauchy-Schwarz inequality,
\begin{align*}
\E_{\PP}&
\left[\exp\left(
        \gamma
        \sum_{i=1}^N
        \int_0^t
        |B_s^{i,r,N}|^2\,\dd s \right) \right]
\\
&\leq
\left(\E_{\PP}
\exp\left\{\frac{2\gamma}{N}
    \sum_{i=1}^r\sum_{j=1}^N
    \int_0^t
    |K(\widehat Y_s^{i,N}-\widehat Y_s^{j,N})|^2\,
    \dd s \right\}\right)^{1/2}
\\
&\quad\times
\left(\E_{\PP}\exp\left\{
    \frac{2\gamma r}{N^2}
    \sum_{i=r+1}^N\sum_{j=1}^r
    \int_0^t
    |K(\widehat Y_s^{i,N}-\widehat Y_s^{j,N})|^2\,\dd s\right\}\right)^{1/2}.
\end{align*}
For the first factor, H\"older's inequality applied to the product
over $i=1,\ldots,r$, followed by Jensen's inequality applied to the
convex function $x\mapsto e^x$, gives
\begin{align*}
&\E_{\PP}
\exp\left\{
    \frac{2\gamma}{N}
    \sum_{i=1}^r\sum_{j=1}^N
    \int_0^t
    |K(\widehat Y_s^{i,N}-\widehat Y_s^{j,N})|^2\,\dd s
\right\}
\\
&\leq
\prod_{i=1}^r
\left(
\E_{\PP}
\exp\left\{
    \frac{2\gamma r}{N}
    \sum_{j=1}^N
    \int_0^t
    |K(\widehat Y_s^{i,N}-\widehat Y_s^{j,N})|^2\,\dd s
\right\}
\right)^{1/r}
\\
&\leq
\prod_{i=1}^r
\left(
    \frac1N\sum_{j=1}^N
    \E_{\PP}
    \exp\left\{
        2\gamma r
        \int_0^t
        |K(\widehat Y_s^{i,N}-\widehat Y_s^{j,N})|^2\,\dd s
    \right\}
\right)^{1/r}.
\end{align*}
Similarly, when $N>r$, H\"older's inequality applied to the product
over $j=r+1,\ldots,N$, followed by Jensen's inequality, gives
\begin{align*}
\E_{\PP}&
\exp\left\{\frac{2\gamma r}{N^2}
    \sum_{j=r+1}^N\sum_{i=1}^r
    \int_0^t |K(\widehat Y_s^{j,N}-\widehat Y_s^{i,N})|^2\,\dd s\right\}
\\
&\leq
\prod_{j=r+1}^N
\left(
\E_{\PP}
\exp\left\{
    \frac{2\gamma r(N-r)}{N^2}
    \sum_{i=1}^r
    \int_0^t |K(\widehat Y_s^{j,N}-\widehat Y_s^{i,N})|^2\,\dd s
\right\}\right)^{1/(N-r)}
\\
&\leq
\prod_{j=r+1}^N
\left(\frac1r\sum_{i=1}^r
    \E_{\PP}
    \exp\left\{
        \frac{2\gamma r^2(N-r)}{N^2}
        \int_0^t
        |K(\widehat Y_s^{j,N}-\widehat Y_s^{i,N})|^2\,\dd s
    \right\}\right)^{1/(N-r)}.
\end{align*}
Combining the previous estimates, we obtain
\begin{align*}
\E_{\PP}&
\left[\exp\left(
        \gamma
        \sum_{i=1}^N
        \int_0^t
        |B_s^{i,r,N}|^2\,\dd s \right)\right]
\\
&\leq
\left(\prod_{i=1}^r
    \frac1N\sum_{j=1}^N
    \E_{\PP}
    \exp\left\{
        2\gamma r
        \int_0^t
        |K(\widehat Y_s^{i,N}-\widehat Y_s^{j,N})|^2\,\dd s
    \right\}\right)^{\frac1{2r}}
\\
&\quad\times
\left(\prod_{j=r+1}^N
    \frac1r\sum_{i=1}^r
    \E_{\PP}
    \exp\left\{
        \frac{2\gamma r^2(N-r)}{N^2}
        \int_0^t
        |K(\widehat Y_s^{j,N}-\widehat Y_s^{i,N})|^2\,\dd s
    \right\}\right)^{\frac1{2(N-r)}}.
\end{align*}
When $N=r$, the second factor is omitted.

For $i\leq r$, the process $\widehat Y^{i,N}$ is a generalized
Ornstein--Uhlenbeck process and is independent of
$\widehat Y^{j,N}$ for every $j\neq i$, whereas the term $j=i$
vanishes by the convention $K(0)=0$. Hence, conditioning on the
trajectory of $\widehat Y^{j,N}$ and using
Lemma~\ref{lem:exponential-OU}, each expectation appearing in the
first factor is bounded uniformly in $N$, in the initial
configuration and in $t\in[0,1]$.
For the second factor, for $i\leq r<j$ we write
\( K(\widehat Y_s^{j,N}-\widehat Y_s^{i,N})
    =
    \widetilde K(\widehat Y_s^{i,N}-\widehat Y_s^{j,N}),\) with
    \(\widetilde K(x)\coloneqq K(-x).\)
Since $\|\widetilde K\|_{L^k}=\|K\|_{L^k}$, conditioning again on the
trajectory of $\widehat Y^{j,N}$ allows us to apply
Lemma~\ref{lem:exponential-OU}. Moreover, for $N>r$,
\(\frac{2\gamma r^2(N-r)}{N^2}
    \leq 2\gamma r,\)
so the parameter in the exponential is bounded independently of
$N$. We therefore obtain, for any $\gamma>0$,
\begin{equation}
\label{eq:uniform-exp-B}
    \sup_{N\geq r}
    \sup_{x\in(\mathbb R^d)^N}
    \sup_{t\in[0,1]}
    \E_{\PP}
    \left[\exp\left(\gamma
            \sum_{i=1}^N \int_0^t
            |B_s^{i,r,N}|^2\,\dd s\right)\right] \leq C.
\end{equation}
Hence, Novikov's criterion
(see e.g. \cite[Chapter 3, Proposition 5.13]{KS91})
implies that $(Z_t^{r,N})_{t\in[0,1]}$ is a true martingale.
Now define
\[M_t \coloneqq  -\frac1{\sqrt2}
    \sum_{i=1}^N
    \int_0^t
    B_s^{i,r,N}\cdot\dd W_s^i.\]
Then
\[\langle M\rangle_t =\frac12
    \sum_{i=1}^N \int_0^t |B_s^{i,r,N}|^2\,\dd s,\]
and
\[Z_t^{r,N} =\exp\left(M_t-\frac12\langle M\rangle_t \right) =\mathcal E(M)_t.\]
Let now $p>1$. We have
\[(Z_t^{r,N})^p =
    \mathcal E(2pM)_t^{1/2}
    \exp\left( \left(p^2-\frac p2\right)
        \langle M\rangle_t\right).\]
Since \eqref{eq:uniform-exp-B} holds for any $\gamma>0$,
Novikov's criterion also applies to $2pM$. Thus
$\mathcal E(2pM)$ is a true martingale. By the Cauchy-Schwarz
inequality,
\begin{align*}
    \mathbb E_{\hat{\mathbb P}}
    \left[ (Z_t^{r,N})^p \right]
    &\leq
    \left( \mathbb E_{\hat{\mathbb P}} [\mathcal E(2pM)_t]\right)^{1/2}
    \mathbb E_{\hat{\mathbb P}}
    \left[\exp\left(
            (2p^2-p)\langle M\rangle_t \right) \right]^{1/2}
    \\
    &=\mathbb E_{\hat{\mathbb P}}
    \left[ \exp\left(
            \left(p^2-\frac p2\right)
            \sum_{i=1}^N
            \int_0^t
            |B_s^{i,r,N}|^2\,\dd s
        \right) \right]^{1/2}
    \\
    &\leq C.
\end{align*}
This proves \eqref{eq:uniform-martingale-bound}.
Finally, by Girsanov's theorem, under $\mathbb Q^{r,N}$ the processes
\[\widetilde W_t^i \coloneqq  W_t^i
    +\frac1{\sqrt2}
    \int_0^t
    B_s^{i,r,N}\,\dd s,
    \qquad i=1,\ldots,N,\]
are independent Brownian motions. Replacing $W^i$ by $\widetilde W^i$
in \eqref{eq:auxiliary-particle-system}, we obtain
\[ \dd\widehat Y_t^{i,N}
    =-F(\widehat Y_t^{i,N})\,\dd t
    - \frac1N\sum_{j=1}^N
    K(\widehat Y_t^{i,N}-\widehat Y_t^{j,N})\,\dd t
    +\sqrt2\,\dd\widetilde W_t^i,
    \qquad i=1,\ldots,N.\]
The conclusion follows by uniqueness in law.
\end{proof}

\bibliographystyle{amsplainhyper_m}
\bibliography{UniformPOC}

\end{document}